\documentclass[12pt,a4paper,reqno]{amsart}

\usepackage[margin=2.8cm,footskip=1cm]{geometry}
\usepackage[utf8]{inputenc}
\usepackage[T1]{fontenc}
\usepackage[english]{babel}
\usepackage{amsmath}
\usepackage{amsfonts}
\usepackage{amssymb}
\usepackage{amsthm}
\usepackage{ucs}
\usepackage{graphicx}
\usepackage{subfig}
\usepackage{xcolor}
\usepackage{dsfont}
\usepackage{wasysym}
\usepackage{physics}
\usepackage{mathtools}
\usepackage{xifthen}
\usepackage{enumitem}
\usepackage{stmaryrd}
\usepackage{tikz,pgfplots}
\usepackage{float}
\usepackage{soul}
\usepackage{hyperref}
\definecolor{colorlinks}{RGB}{0, 24, 168}
\definecolor{colorcites}{RGB}{124, 10, 2}
\hypersetup{
	colorlinks=true,
	linkcolor=colorlinks,
	citecolor=colorcites,
	urlcolor=colorlinks,
	pdfborder={0 0 0}
}
\usepackage[font={small}]{caption}
\usepackage{algorithm2e}
\usepackage{constants}

\newcommand{\given}{\, |\,}
\newcommand{\bgiven}{\, \big|\,}
\newcommand{\Bgiven}{\, \Big|\,}

\newcommand{\ind}{\mathds{1}}

\newconstantfamily{crmNum}{symbol=\mathrm{c}}

\newcommand{\N}{\mathbb{N}}

\newcommand{\R}{\mathbb{R}}
\newcommand{\Z}{\mathbb{Z}}

\newcommand{\bbE}{\mathbb{E}}
\newcommand{\bbG}{\mathbb{G}}
\newcommand{\bbH}{\mathbb{H}}
\newcommand{\bbL}{\mathbb{L}}

\newcommand{\bbV}{\mathbb{V}}

\newcommand{\calC}{\mathcal{C}}

\newcommand{\calF}{\mathcal{F}}

\newcommand{\calK}{\mathcal{K}}

\newcommand{\calP}{\mathcal{P}}

\newcommand{\calS}{\mathcal{S}}
\newcommand{\calT}{\mathcal{T}}

\newcommand{\calX}{\mathcal{X}}

\newcommand{\rmb}{\mathrm{b}}
\newcommand{\rmd}{\mathrm{d}}
\newcommand{\rme}{\mathrm{e}}

\newcommand{\rmi}{\mathrm{i}}
\newcommand{\rmw}{\mathrm{w}}

\newcommand{\rmC}{\mathrm{C}}

\newcommand{\rmT}{\mathrm{T}}

\newcommand{\mcal}[1]{\mathcal{#1}}
\newcommand{\msf}[1]{\mathsf{#1}}

\newcommand{\mrm}[1]{\mathrm{#1}}
\newcommand{\mfr}[1]{\mathfrak{#1}}

\newcommand{\fcone}{\mathcal{Y}^\blacktriangleleft}
\newcommand{\bcone}{\mathcal{Y}^\blacktriangleright}
\newcommand{\bend}{\mathbf{b}}
\newcommand{\fend}{\mathbf{f}}
\newcommand{\diam}{\mathrm{Diam}}
\newcommand{\DiaEnv}{\mathcal{D}}

\newcommand{\CPts}{\textnormal{CPts}}

\newcommand{\concatenate}{\circ}
\newcommand{\displace}{\mathrm{X}}

\newcommand{\SetRootMarkBackCont}{\mathfrak{B}_L}
\newcommand{\SetRootMarkForwCont}{\mathfrak{B}_R}
\newcommand{\SetRootDiaCont}{\mathfrak{A}}

\newcommand{\MixMeas}{\mathrm{Bnd}}

\newcommand{\cstFinEne}{c_{\mathrm{FE}}}

\newcommand{\OZDec}{\mathrm{OZDec}}
\newcommand{\OZwalk}{\mathrm{OZwalk}}
\newcommand{\OZ}{\mathrm{OZ}}

\newcommand{\OZRVwalk}{\mathcal{W}}
\newcommand{\OZRVchain}{\mathcal{OZ}}
\newcommand{\OZRVcluster}{\Upsilon}

\newcommand{\bndMeas}{\mathfrak{Q}}

\newcommand{\slab}{\mathcal{S}}
\newcommand{\slabCP}{\slab\CPts}
\newcommand{\goodCl}{\mathrm{GCl}}

\newcommand{\FVEdges}{\mathrm{E}}

\newcommand{\scale}{\mathrm{sc}}

\newcommand{\syncTime}{\mathrm{t}}
\newcommand{\HitEvent}{\mathrm{Hit}}
\newcommand{\syncHitEvent}{\mathrm{sHit}}

\newcommand{\syncHitTime}{\mathrm{sHT}}
\newcommand{\IntersectionTime}{\calT}

\newcommand{\connection}{\mathrm{Con}}

\newcommand{\spin}{\mathtt{Spin}}

\newcommand{\atrc}{\mathtt{ATRC}}
\newcommand{\matrc}{\mathtt{mATRC}}
\newcommand{\fk}{\mathtt{FK}}
\newcommand{\potts}{\mathtt{Potts}}
\newcommand{\qfk}{\mathtt{qFK}}

\newcommand{\fkfree}{\mathrm{f}}
\newcommand{\fkwired}{\mathrm{w}}
\newcommand{\atrcfree}{0}
\newcommand{\atrcwired}{1}

\newcommand{\clusters}{\kappa}

\newcommand{\loops}{\mathrm{loop}}

\newcommand{\tvd}{\mathrm{d}_{\mathrm{TV}}}

\newcommand{\svc}{\mathbf{c}}
\newcommand{\svcb}{\mathbf{c}_{\mathrm{b}}}

\newcommand{\qbfree}{{q_\mathrm{b}^{\mathrm{\scriptscriptstyle f}}}}
\newcommand{\qbwired}{{q_\mathrm{b}^{\mathrm{\scriptscriptstyle w}}}}

\newcommand{\Forget}{\mathrm{Forget}}

\newcommand{\condPrimalClusterMeas}{\mathrm{PCl}}
\newcommand{\condDualClusterMeas}{\mathrm{DCl}}
\newcommand{\condPrimalDualClusterMeas}{\mathrm{PDCl}}

\newcommand{\partialex}{\partial^{\mathrm{ex}}}
\newcommand{\partialin}{\partial^{\mathrm{in}}}
\newcommand{\partialedge}{\partial^{\mathrm{edge}}}

\def\Hloop#1#2{
\draw[blue] ({#1 + cos(-45)/(2*sqrt(2))},{#2+0.5 + sin(-45)/(2*sqrt(2))}) arc (-45:-135:{1/(2*sqrt(2))}) ;
\draw[blue] ({#1 + cos(45)/(2*sqrt(2))},{#2-0.5 + sin(45)/(2*sqrt(2))}) arc (45:135:{1/(2*sqrt(2))}) ;
}
\def\Vloop#1#2{
\draw[blue] ({#1 + 0.5 + cos(135)/(2*sqrt(2)) },{#2 + sin(135)/(2*sqrt(2))}) arc (135:225:{1/(2*sqrt(2))}) ;
\draw[blue] ({#1 - 0.5 + cos(-45)/(2*sqrt(2))},{#2 + sin(-45)/(2*sqrt(2))}) arc (-45:45:{1/(2*sqrt(2))}) ;
}

\def\DrawTile#1#2#3{
\draw[#3] ({#1-0.5},{#2}) -- ({#1},{#2+0.5}) -- ({#1+0.5},{#2}) -- ({#1},{#2-0.5}) -- ({#1-0.5},{#2}) ;
}

\def\Hedge#1#2#3{
\draw[#3] ({#1-0.5},#2) -- ({#1+0.5},#2) ;
}

\def\Vedge#1#2#3{
\draw[#3] (#1,{#2-0.5}) -- (#1,{#2+0.5}) ;
}

\def\UParrow#1#2#3{
\draw[#3] ({#1-0.1},{#2-0.1}) -- ({#1},{#2}) -- ({#1+0.1},{#2-0.1}) ;
}
\def\DOWNarrow#1#2#3{
\draw[#3] ({#1-0.1},{#2+0.1}) -- ({#1},{#2}) -- ({#1+0.1},{#2+0.1}) ;
}
\def\LEFTarrow#1#2#3{
\draw[#3] ({#1+0.1},{#2-0.1}) -- ({#1},{#2}) -- ({#1+0.1},{#2+0.1}) ;
}
\def\RIGHTarrow#1#2#3{
\draw[#3] ({#1-0.1},{#2-0.1}) -- ({#1},{#2}) -- ({#1-0.1},{#2+0.1}) ;
}

\theoremstyle{plain}
\newtheorem{theorem}{Theorem}[section]
\newtheorem{lemma}[theorem]{Lemma}

\newtheorem{corollary}[theorem]{Corollary}

\newtheorem{remark}[theorem]{Remark}

\newtheorem{claim}[theorem]{Claim}

\theoremstyle{definition}
\newtheorem{definition}[theorem]{Definition}

\makeatletter
\renewcommand{\section}{\@startsection{section}{3}%
\z@                     % Indent (none)
{\baselineskip}         % Space before (exactly one line)
{0.5\baselineskip}                  % Space after
{\normalfont\large\scshape\bfseries\centering}} % Bold font
\makeatother

\makeatletter
\renewcommand{\subsection}{\@startsection{subsection}{3}%
\z@                     % Indent (none)
{\baselineskip}         % Space before (exactly one line)
{0.5\baselineskip}                  % Space after (negative = run-in, 1em horizontal gap)
{\normalfont\large\bfseries}} % Bold font
\makeatother

\makeatletter
\renewcommand{\subsubsection}{\@startsection{subsubsection}{3}%
\z@                     % Indent (none)
{\baselineskip}         % Space before (exactly one line)
{-1em}                  % Space after (negative = run-in, 1em horizontal gap)
{\normalfont\normalsize\bfseries}} % Bold font
\makeatother
\newif\ifpic
\pictrue

\title[Order-Order Interface in 2D Potts]{Discontinuous transition in 2D Potts:\\ II. Wetting and Brownian watermelon.}

\author{Moritz Dober}
\address{Fakultät für Mathematik, Universität Wien, Vienna, Austria}
\email{moritz.dober@univie.ac.at}

\author{Alexander Glazman}
\address{Karlsruhe Institute of Technology, Karlsruhe, Germany and TU Wien, Vienna, Austria}
\email{alexander.glazman@kit.edu}

\author{Sébastien Ott}
\address{Institute of Mathematics, EPFL, 1015 Lausanne, Switzerland}
\email{ott.sebast@gmail.com}

\date{\today}

\begin{document}

\begin{abstract}
	The $q$-state Potts model is an archetypical model for various types of phase transitions.
	We consider it on the square grid and focus on the regime where it undergoes a discontinuous transition, that is~$q>4$.
	At the transition point~$T_c(q)$, there are exactly~$q+1$ extremal Gibbs measures (pure phases): $q$ ordered (monochromatic) and one disordered (free).
	This work establishes for the first time the wetting phenomenon in a precise geometric form and in the entire regime of discontinuity~$q>4$: at~$T_c(q)$, between two ordered phases a disordered layer emerges and, in the diffusive scaling, its boundaries converge to a pair of Brownian motions conditioned not to intersect.
	This is starkly different from the subcritical ($T<T_c(q)$) behaviour.
	At~$T_c(q)$, previous results (Bricmont--Lebowitz '87, Messager--Miracle-Sole--Ruiz--Shlosman '91) were limited to the construction and properties of the surface tension for large enough~$q$.
	
	In a companion work, we provide a detailed study of the Potts model under order-disorder Dobrushin conditions.
	That work also develops a ``renewal picture'' \emph{à la} Ornstein-Zernike for a suitable percolation model, which plays a central part in our study of the Potts interfaces.
	The latter is the random-cluster representation of an Ashkin--Teller model (ATRC), and is related to the Potts model via a chain of couplings going through the six-vertex model.
	
	In the current work, we extend the analysis to a pair of interacting order-disorder interfaces forming the separation between the two ordered phases, and couple them to a pair of well-behaved random walks conditioned not to intersect.
	The construction of the coupling is based on rigorously deriving entropic repulsion between the two interfaces.
	We also prove convergence of interfaces in the FK-percolation model at~$p_c(q)$ when~$q>4$.
\end{abstract}

\maketitle

\setcounter{tocdepth}{1}
\tableofcontents

\newpage

\section{Introduction and description of the main result}
\label{sec:intro}

The Potts model is a classical model of statistical mechanics introduced in 1952~\cite{Pot52}.
Each vertex of a graph is assigned one of $q$ states (colours), with states of adjacent vertices interacting with a strength depending on the temperature $T>0$ of the system. 
At~$q=2$, this corresponds to the seminal Ising model.
The Potts model becomes increasingly ordered as the temperature decreases, and a phase transition occurs on lattices $\Z^{d}$ with $d \geq 2$ at some transition temperature \(T_c(q,d)>0\).
We refer an interested reader to our companion work~\cite{DobGlaOtt25} and to a survey~\cite{Dum17a} for a detailed introduction and the historical background.

Our work is restricted to dimension~$d=2$ where more tools are available due to planar duality.
We now list several important results:
\begin{itemize}
	\item $T_c(q)=[\ln (1+\sqrt{q})]^{-1}$~\cite{BefDum12}, which is the self-dual point;
	\item the transition is continuous when~$q =2,3,4$, in a sense that, at~$T_c(q)$, there is a unique Gibbs measure~\cite{DumSidTas17} (see also~\cite{GlaLam23} for another argument);
	\item the transition is discontinuous when~$q > 4$~\cite{DumGagHar21} (see also~\cite{RaySpi20} for a short proof), and any Gibbs measure can be written as a linear combination of~$q+1$ extremal Gibbs measures ($q$ monochromatic and one free)~\cite{GlaMan23}.
\end{itemize}

We focus on discontinuous transitions ($q>4$) and study interfaces at~$T_c(q)$ separating different states.
The structure of extremal Gibbs measures (described above) leads to two natural definitions of Dobrushin boundary conditions, which is a natural way to model phase coexistence:
\begin{itemize}
	\item {\em order-disorder}: one half of the boundary is of a fixed colour and the other one is free (no colour assigned);
	\item {\em order-order}: both halves of the boundary are assigned different fixed colours.
\end{itemize}
The phenomenology is quite different in the two cases, see Fig.~\ref{Fig:Potts_interface_simul}. 
Indeed, the case of order-disorder phase coexistence has a behaviour similar to the one of order-order interfaces in the Potts or Ising model at \(T<T_c(q)\), where the region separating the two phases has width \(O(1)\) uniformly over the system size. In sharp contrast, at \(T_c(q)\), the order-order phase coexistence undergoes an \emph{interfacial wetting} phenomena: for \emph{microscopic} energetic reasons, a mesoscopic layer of disordered phase spontaneously appears between the two ordered phases. See Section~\ref{subsec:inter_wetting} for more details.

In the current paper we analyse the order-order boundary conditions on a box~$N\times N$.
Our main result is a precise description of the geometry of the interfacial wetting phenomenon: emergence of a free layer of width~$\sqrt{N}$ between the two ordered phases, and convergence, under diffusive scaling, of its boundaries to two Brownian bridges conditioned not to intersect ({\em Brownian watermelon}).
This behaviour is starkly different from the case~$T<T_c$, where the interface between two different colours was shown to converge to {\em one} Brownian bridge for all~$q\geq 1$~\cite{CamIofVel08}.
The difference is due to the fact that the disordered and ordered phases are thermodynamically stable simultaneously only at~$T_c(q)$ and only when~$q>4$.
As a consequence, the robust Ornstein--Zernike theory, available at \(T<T_c(q)\), does not apply directly to the Potts model at \(T_c(q)\), nor to its seminal graphical representation provided by the FK percolation~\cite{ForKas72,EdwSok88}.
Instead, we build on~\cite{BaxKelWu76} and~\cite{GlaPel23} to map the Potts model to a suitable percolation model that does exhibit a unique infinite-volume Gibbs measure.
This model goes under the name ATRC (Ashkin--Teller random-cluster model) and was introduced in~\cite{PfiVel97}.
Our study provides the first ever rigorous geometric description of interfacial wetting in a non-effective model.

\begin{figure}
	\includegraphics[scale=0.13]{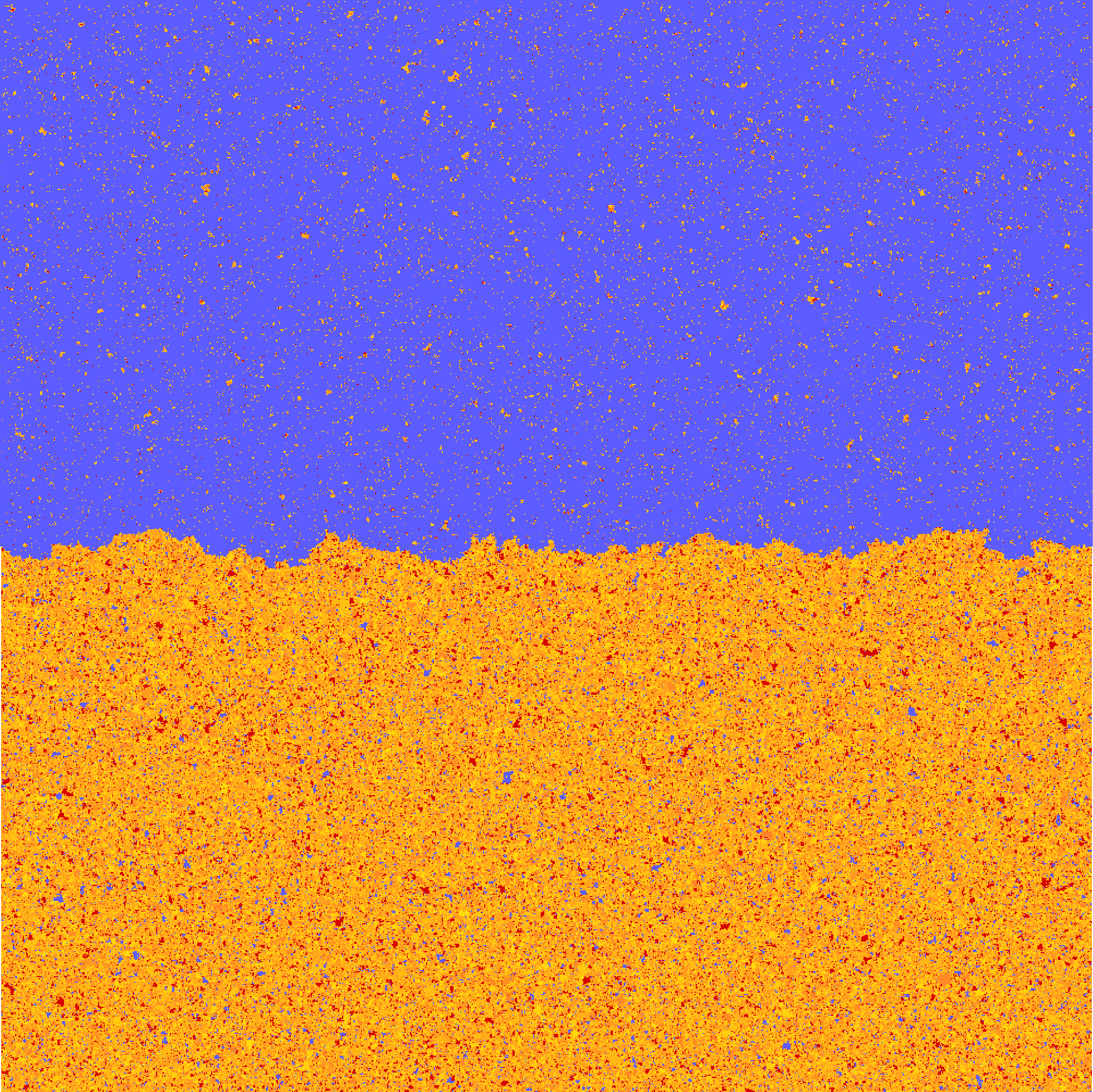}
	\hspace*{0.4cm}
	\includegraphics[scale=0.13]{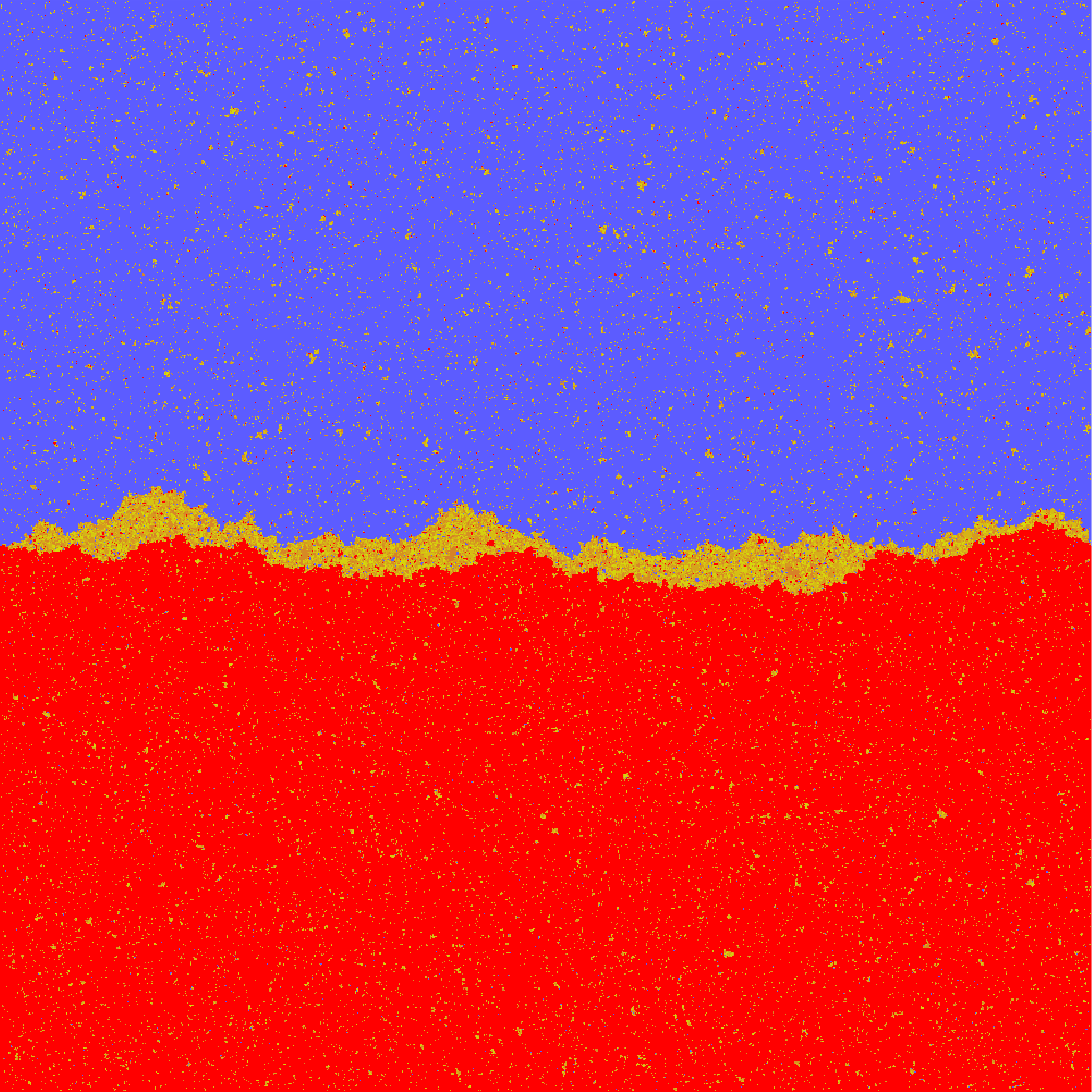}
	\caption{Sample of a 1000x1000 Potts model with 25 colours at \(T_c(25)\) with Dobrushin boundary conditions. Colours are: blue for the first, red for the second and interpolate between yellow and orange for colours 3 to 25. Left: order-disorder interface; upper part has blue b.c., bottom has white b.c. (no colour favoured). Right: order-order interface; upper part has blue b.c., bottom has red b.c..}
	\label{Fig:Potts_interface_simul}
\end{figure}

\subsection{Potts model}

We view~$\Z^2$ both as a set of points on the plane having integer coordinates and as a graph (square grid) with edges linking points at distance one.
Denote by~$\bbE$ the set of edges in~$\Z^2$ and write~$i\sim j$ if~$\{i,j\}\in\bbE$.
Let~$G=(V,E)$ be a subgraph of~$\Z^2$.
Take parameters $T>0,\,q\in \{2,3,4,\dots\}$, and boundary conditions $\eta\in\{0,1,\dots,q\}^{\bbV}$. 
The Potts model on \(G\) with boundary conditions \(\eta\) is the probability measure on \(\{1,\dots,q\}^{V}\) given by
\begin{equation*}
	\potts_{G;T,q}^{\eta}(\sigma):=
	\tfrac{1}{Z_{{\potts}}} \cdot 
	\exp\Bigg[\tfrac1T \cdot \Bigg(
	\sum_{\{i,j\}\in E} \delta(\sigma_i,\sigma_j) + \sum_{i\in V, j\in V^c:\, i\sim j} \delta(\sigma_i,\eta_j)\Bigg)\Bigg],
\end{equation*}
where~$Z_{{\potts}}=Z_{{\potts}}(G,T,q,\eta)$ is the unique normalising constant (called {\em partition function}) that renders the above a probability measure, and~$\delta(x,y) = 1$ if~$x=y$ and~$\delta(x,y) = 0$ otherwise. Value~\(0\) for~\(\eta\) corresponds to free boundary conditions favoring none of the \(q\) possible states of the spins.

We say that~$\eta$ defines the order-order (1-2) Dobrushin boundary conditions (and denote it by~$1/2$) if~$\eta((x,y))=\ind_{\geq 0}(y) + 2\ind_{< 0}(y)$.
Identify~$\Lambda_n:=\{-n,\dots, n\}^2$ with the induced subgraph of~$\Z^2$ on this set of vertices.
The Dobrushin boundary conditions on~$\Lambda_n$ impose existence of two interfaces: one going along the lower boundary of the cluster of~$1$s and the other going along the upper boundary of the  cluster of~$2$s.
This can be made explicit, but for brevity we choose to define directly the upper and lower discrete envelopes of these interfaces: $\Gamma_{\potts}^{1+}$ and~$\Gamma_{\potts}^{1-}$  respectively for the lower boundary of the cluster of~$1$s and $\Gamma_{\potts}^{2+}$ and~$\Gamma_{\potts}^{2-}$  respectively for the upper boundary of the cluster of~$2$s.
Given~$\sigma\in\{1,\dots,q\}^{\Lambda_n}$, define~$\bar\sigma \in\{1,\dots,q\}^{\Z^2}$ to be its extension outside of~$\Lambda_n$ by the Dobrushin boundary conditions: one on~$\Z\times \Z_{\geq 0}$ and two on~$\Z\times \Z_{< 0}$. 
For~\(k =-n,\dots, n\), define
\begin{align*}
	\Gamma_{\potts}^{1+,n}(k) & := \max\{y\in \Z:\ (k,y-1)\xleftrightarrow{\bar\sigma \neq 1}  (\Z\times \Z_{< 0})\setminus\Lambda_n\},\\
	\Gamma_{\potts}^{1-,n}(k) & := \min\{y\in \Z:\ (k,y+1)\xleftrightarrow{\bar\sigma = 1, \text{ diag}}  (\Z\times \Z_{\geq 0})\setminus\Lambda_n\},\\
	\Gamma_{\potts}^{2+,n}(k) & := \max\{y\in \Z:\ (k,y-1)\xleftrightarrow{\bar\sigma = 2}  (\Z\times \Z_{< 0})\setminus\Lambda_n\},\\
	\Gamma_{\potts}^{2-,n}(k) & := \min\{y\in \Z:\ (k,y+1)\xleftrightarrow{\bar\sigma \neq 2, \text{ diag}}  (\Z\times \Z_{\geq 0})\setminus\Lambda_n\},
\end{align*}
where~$(k,y-1)\xleftrightarrow{\bar\sigma \neq1}  (\Z\times \Z_{<0})\setminus\Lambda_n$ states the existence of a path in~$\Z^2$ in going from~$(k,y-1)$ to~$(\Z\times \Z_{<0})\setminus\Lambda_n$ and 
consisting of vertices where~$\bar\sigma\neq 1$ and~$(k,y)\xleftrightarrow{\bar\sigma = 1, \text{ diag}}  (\Z\times \Z_{\geq 0})\setminus\Lambda_n$ states the existence of a path in~$\Z^2$ {\em with diagonal connectivity} going from~$(k,y)$ to~$(\Z\times \Z_{\geq 0})\setminus\Lambda_n$ and consisting of vertices where~$\bar\sigma= 1$; and similarly for~$\bar\sigma \neq 2$ and~$\bar\sigma =2$.
By~$\Z^2$ with diagonal connectivity we mean a graph with the vertex-set~$\Z^2$ with edges linking vertices at distance at most~$\sqrt{2}$.
Define the rescaled linear interpolation of these envelops by
\begin{align*}
	\tilde{\Gamma}_{\potts}^{s\pm,n}(t) &:= \tfrac{1}{\sqrt{n}} \big((1-\{2tn -n\})\Gamma_{\potts}^{s\pm}(\lfloor 2tn -n\rfloor) + \{2tn -n\}\Gamma_{\potts}^{s\pm}(\lceil 2tn -n\rceil)\big)
\end{align*}
where~$s\in \{1,2\}$, \(\lfloor\,\rfloor\), \(\lceil\, \rceil\), \(\{\,\}\) denote respectively lower and upper roundings and fractional part.

\begin{figure}
\centering
\includegraphics[page=1,scale=0.5]{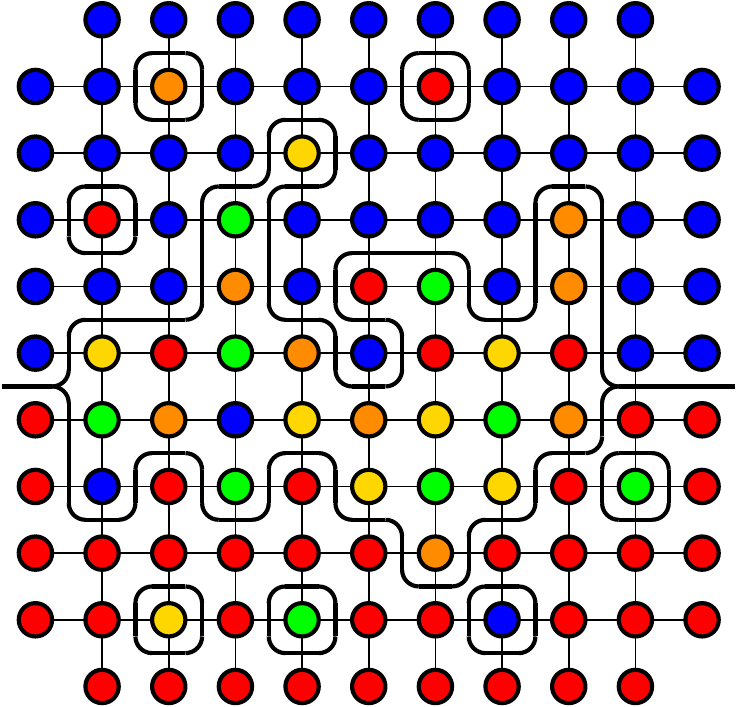}
\caption{\textit{Left:} A configuration of the Potts model on~\(\mathsf{B}_{n,n},\,n=4,\) with order-order Dobrushin b.c. (blue and red represent the constant 1 and 2 b.c. in the upper and lower half-planes, respectively). The `Peierls contours' separate the upper and lower boundary clusters (one in star connectivity) of the blue and red colours from the rest.}
\label{fig:order-order_potts}
\end{figure}

A standard Brownian watermelon with two bridges is a random variable~$(\mathrm{BW}^{(2)}_t)_{t\in[0,1]}$ with values in continuous functions from~$[0,1]$ to~$\R^2$ obtained by conditioning two independent standard Brownian bridges not to intersect when~$t\in (0,1)$.
Although this event has probability zero, one can make sense of the conditioning by means of a Doob transform; we refer to~\cite{DurWac20} for more details.

\begin{theorem}
	\label{thm:invariance_princ_potts}
	Let~$q>4$ be integer and take~$T=T_c(q)$.
	For~$n\in\N$ and~$s\in \{1,2\}$, sample~$\Gamma_{\potts}^{s\pm,n}$ and~$\tilde{\Gamma}_{\potts}^{s\pm,n}$ from~$\potts_{\Lambda_n;T_c(q),q}^{1/2}$ as described above.
	Then, as~$n$ tends to infinity,
	\begin{enumerate}
		\item the envelops~\(\big(\tilde{\Gamma}_{\potts}^{1\pm,n}(t),\tilde{\Gamma}_{\potts}^{2\pm,n}(t)\big)_{t\in [0,1]}\), for any choice of~$+$ and~$-$, converges in law to~$(c_q\mathrm{BW}^{(2)}_{t\in[0,1]}$, where~\(c_q >0\) is some constant;
		\item the probability that~$\max_{k \leq n, s\in\{1,2\}} \left|\Gamma_{\potts}^{s+,n}(k) -\Gamma_{\potts}^{s-,n}(k)\right|\geq C\ln^{111}(n)$ tends to zero.
	\end{enumerate}
\end{theorem}
The diffusivity constant \(c_q\) has an explicit characterization, see~\cite[Remark~$8.12$]{DobGlaOtt25}.
This Theorem completes the study of the Dobrushin interfaces in the Potts model when~$q>4$. 
Indeed, the regime \(T<T_c(q)\) was treated in~\cite{CamIofVel08}, and the order-disorder case in the regime \(T=T_c(q)\) is the object of the companion paper~\cite{DobGlaOtt25}. Theorem~\ref{thm:invariance_princ_potts} provides the missing, and most involved, case of the order-order interface at \(T_c(q)\). This is also the only one displaying the ``atypical'' wetting phenomenon, and the only one having a non-Gaussian limiting process in the diffusive scaling limit, i.e. the only one not scaling to a Brownian motion.

\subsection{FK percolation}

The main tool in analyzing the Potts model is the Fortuin--Kasteleyn (FK) percolation (or random-cluster) model~\cite{ForKas72}.
Take a finite subgraph $G=(V,E)$ of $\Z^2$, parameters $p\in (0,1),\,q>0$ and boundary conditions $\xi\in\{0,1\}^{\bbE}$. 
We identify any~$\omega\in\{0,1\}^{\bbE}$ with the set of edges~$e\in \bbE$ for which~$\omega_e=1$ ({\em open edges}) and with the spanning subgraph of~$\Z^2$ defined by the open edges.
The FK-percolation model on \(G\) with boundary conditions \(\xi\) is the probability measure on $\{0,1\}^{\bbE}$ given by 
\begin{equation*}
	\fk_{G;p,q}^{\xi}(\eta):=\tfrac{1}{Z_{\fk}}\cdot p^{|\eta\cap E|}(1-p)^{|E\setminus\eta|}\,q^{\clusters_V(\eta)}\, \ind_{\eta=\xi \text{ on } \bbE\setminus E},
\end{equation*}
where $Z_{\fk}=Z_{\fk}(G,p,q,\xi)$ is the partition function and~$\clusters_V(\eta)$ is the number of connected components ({\em clusters}) of $(\Z^2,\eta)$ that intersect $V$.
The \emph{free} and \emph{wired} measures correspond to the choices $\xi\equiv 0$ and $\xi\equiv 1$, respectively, and we simply write $\fkfree$ and $\fkwired$ instead of $\xi$, respectively.
We will be interested in the Dobrushin wired/wired boundary conditions (see Figure~\ref{fig:WiredWired_FK}):
\begin{equation*}
\xi_{1/1}(e)=\mathds{1}_{e\subset\bbH^+}+\mathds{1}_{e\subset\bbH^-},
\end{equation*}
Below we denote these conditions by~$1/1$ for brevity.

For \(n\geq 1\), define the graph \(G_{n}=(V_{n},E_{n})\) by
\begin{equation*}
E_{n}:=\{e\in\bbE: e\subset\Lambda_{n}\}\cup\{e\in\bbE: e\cap\Lambda_{n}\neq\varnothing\},\quad V_{n}:=\bigcup_{e\in E_{n}}e.
\end{equation*}
Define~$\partial_n^+$ and~$\partial_n^-$ as the set of vertices in~$\Z\times\Z_{\geq 0}\setminus \Lambda_n$ and~$\Z\times\Z_{< 0}\setminus \Lambda_n$ respectively incident to at least one edge in~$E_n$. Consider the event~$\partial_n^+ \nleftrightarrow \partial_n^-$ that~$\partial_n^+$ is not connected to~$\partial_n^-$ by a path of open edges.
The Edwards--Sokal coupling~\cite{EdwSok88} states that, when~$q\geq 2$ is integer and~$p=1- \exp[-\tfrac{1}{T}]$, a random sample of~$\potts_{\Lambda_n;T,q}^{1/2}$ can be obtained as follows: sample~$\omega \sim \fk_{G_n;p,q}^{1/1}(\,\cdot\, | \partial_n^+ \nleftrightarrow \partial_n^-)$ and color all of its clusters intersecting~$\Z\times\Z_{\geq 0}\setminus \Lambda_n$ (resp.~$\Z\times\Z_{<0}\setminus \Lambda_n$) in~$1$ (resp.~$2$) and color all other clusters independently in colors~$1,2,\dots,q$ (see Figure~\ref{fig:order-order_potts}).

As in~\cite{DobGlaOtt25}, we define the interfaces in the FK percolation forced by the Dobrushin boundary conditions using planar duality.
Indeed, the lattice dual to~$\Z^2$ is again a square lattice that we denote by~$(\Z^2)^*$; for each edge~$e$ of~$\Z^2$, denote by~$e^*$ the edge of~$(\Z^2)^*$ that is dual to~$e$, i.e. the unique edge of~$(\Z^2)^*$ that intersects~$e$; see Fig.~\ref{fig:midEdgeTiles}.
Given~$\omega\in \{0,1\}^{\bbE}$, define its dual by
\[
	\omega^*_{e^*}:= 1-\omega_e.
\]
The FK percolation at $p_c(q)$, given by
\[
	p_c(q) := 1- \exp[-\tfrac{1}{T_c(q)}] = \tfrac{\sqrt{q}}{\sqrt{q}+1},
\]
is known to enjoy the self-duality: if~$\omega$ is sampled from~$\fk_{G_n;T_c(q),q}^{1/1}(\,\cdot\, | \partial_n^+ \nleftrightarrow \partial_n^-)$ with the non-crossing conditioning mentioned above, then its dual~$\omega^*$ is also distributed as an FK-percolation with parameters~$q$ and~$p_c(q)$, but on a dual graph and under free boundary conditions with the conditioning that~$v_L:=(-n-1/2,-1/2)$ is connected to~$v_R:=(n+1/2,-1/2)$.
We now define the loop representation of~$\omega$ by drawing two arcs next to every primal or dual edge of~$\omega$, see Fig.~\ref{fig:midEdgeTiles}.
These arcs link together into loops separating primal and dual clusters, and one of these loops traces the boundary of the dual cluster containing~$v_L$ and~$v_R$.
Cutting the one arc to the left of~$v_L$ and one arc to the right of~$v_R$, we obtain two interfaces.
The upper interface by~$\Gamma_{\fk}^1$ and the lower interface by~$\Gamma_{\fk}^2$.

As for the Potts model, we define the upper and the lower discrete envelops of~$\Gamma_{{\fk}}^1$ and~$\Gamma_{{\fk}}^2$:
\begin{align*}
	\Gamma_{\fk}^{1+,n}(k) & := \max\{y\in \Z:\ (k\pm\tfrac12,y-\tfrac12)\xleftrightarrow{\omega^*} v_L\},\\
	\Gamma_{\fk}^{1-,n}(k) & := \min\{y\in \Z:\ (k,y+1)\xleftrightarrow{\omega} (\Z\times \Z_{\geq 0})\setminus\Lambda_n\},\\
	\Gamma_{\fk}^{2+,n}(k) & := \max\{y\in \Z:\ (k\pm\tfrac12,y-\tfrac12)\xleftrightarrow{\omega^*} v_L\},\\
	\Gamma_{\fk}^{2-,n}(k) & := \min\{y\in \Z:\ (k,y+1)\xleftrightarrow{\omega} (\Z\times \Z_{< 0})\setminus\Lambda_n\},
\end{align*}
where \(\Lambda_n^*:=[-n,n]^2\cap(\Z^2)^*\).
The rescaled linear interpolations~$\tilde{\Gamma}_{\fk}^{s\pm,n}$ of~$\Gamma_{s\fk}^{\pm,n}$ are defined in the same way as in the Potts model for~$s\in\{1,2\}$.
Our main result for the FK percolation is the following invariance principle for the pair~$(\tilde{\Gamma}_{\fk}^{1\pm,n},\tilde{\Gamma}_{\fk}^{2\pm,n})$:

\begin{theorem}
	\label{thm:invariance_princ_FK}
	Let~$q>4$ be a real number and take~$p=p_c(q)$.
	For~$n\in\N$ and~$s\in \{1,2\}$, sample~$\Gamma_{\fk}^{s\pm,n}$ and~$\tilde{\Gamma}_{\fk}^{s\pm,n}$ from~$\fk_{G_n;p_c(q),q}^{1/1}(\,\cdot\, | \partial_n^+ \nleftrightarrow \partial_n^-)$ as described above.
	Then, as~$n$ tends to infinity, the convergence results from Theorem~\ref{thm:invariance_princ_potts} hold for~$(\Gamma_{\fk}^{1\pm,n},\Gamma_{\fk}^{2\pm,n})$ and~$(\tilde{\Gamma}_{\fk}^{1\pm,n},\tilde{\Gamma}_{\fk}^{2\pm,n})$ in place of~$(\Gamma_{\potts}^{1\pm,n},\Gamma_{\potts}^{2\pm,n})$ and~$(\tilde{\Gamma}_{\potts}^{1\pm,n},\tilde{\Gamma}_{\potts}^{2\pm,n})$.
\end{theorem}

Note the precision of our control of the interfaces.
We now provide a historical context for the problem of interfacial wetting.

\subsection{Interfacial wetting}
\label{subsec:inter_wetting}

The phenomena of \emph{interfacial wetting} occurs in the case where at least three substances are coexisting, i.e. thermodynamically stable, and one looks at the microscopic interface between two of them. 
Typically, the interface between substances A and B is a microscopic object, and its energetic cost is the sum of the energies of the microscopic contacts between particles of the first substance and the one of the second. Interfacial wetting occurs when the microscopic contact energy between substances A and B is \emph{higher} than the sum of the contact energy between substances A and C and the contact energy between substances B and C. In this case, it is energetically favourable to separate substance A from substance B using a layer of substance C, forming a form of ``double microscopic interface'' between A and B. This phenomena was extensively studied by physicists, see for example~\cite{DerSch86,DieLat89} and references therein. Yet, only modest and partial mathematical results were available. \cite{BriLeb87} provides a picture of wetting in Potts and Blume--Capel models based on low temperature expansions. Their rigorous results are however limited to the construction of the surface tension and some control over it, and no results on the geometry of the separating region between two ordered phases is proved.
Similarly, \cite{deCMesMirRui88} also provides thermodynamic estimates and does not study the diffusive limit: using correlation inequalities, the surface tension between two ordered phases is shown to be at least twice the one between an ordered and the disordered phase. 
In~\cite{LaaMesMirRuiShl91,MesMirRuiShl91}, the Antonov's rule is proven for \(q\) large enough: the surface tension between two ordered phases is equal to twice the one between an ordered and the disordered phase.
These works also provide rigorous ``large \(q\) expansions'' for the interfaces, giving a representation of the order-order ``interface'' as two interacting microscopic contours.
Finally~\cite{HryKot02} obtain a marker of the wetting phenomenon: a detailed picture of the geometry of the \(O(1)\) disordered layer separating two ordered ones in the Blume-Capel model at a very low temperature where the disordered phase is unstable. The expected behaviour is that the mean size of that layer will blow up when approaching the temperature at which the disordered phase becomes stable.

As this article is concerned with interfacial wetting, we will not do a detailed review of all the works dealing with \emph{boundary wetting} and refer to the surveys~\cite{BriEMeFro86, Vel06} for references and details.

\subsection{Summary of the paper: what is new?}

In this work and in the companion paper~\cite{DobGlaOtt25}, we put forward the idea to study the two-dimensional Potts model via a percolation model called the ATRC model (see Section~\ref{sec:notations} for a definition).
This is done via the chain of couplings that first appeared in~\cite{GlaPel23} and was used, in particular, to establish exponential decay in an infinite-volume ATRC measure; in~\cite{AouDobGla24}, a stronger (finite-volume) form of the exponential decay was proven and this allowed to establish a complete phase diagram on the Ashkin--Teller model in the ferromagnetic case.
Both works relied on the results for the FK percolation~\cite{DumSidTas17,DumGagHar21}.

  this infinite-volume measure was shown to be in fact unique, which is the key to our analysis here.

In the current work and in~\cite{DobGlaOtt25}, we are going in the reverse direction and for the first time use the Ashkin--Teller model to analyse the Potts model.
Precisely, in~\cite{DobGlaOtt25}, we proved uniqueness of the infinite-volume measure and stronger mixing estimates for the ATRC using the previous works and~\cite{Ott25}. These are the key properties that were used in~\cite{DobGlaOtt25} to develop the Ornstein--Zernike theory for one interface in the ATRC model.
The main result in~\cite{DobGlaOtt25} is that this single interface corresponding to the order-disorder boundary conditions converges to the Brownian bridge.

In the current work, we have to deal with a pair of interacting interfaces.
The main point of mapping the Potts/FK model to the ATRC is that, due to positive correlation properties of the ATRC model, we can show that the interfaces have an effectively repulsive effect on one-another in the ATRC model. This property of effective repulsiveness is what makes our study possible, and is in general very hard to establish, even in perturbative regimes of parameters. See for example~\cite{IofShlTon15} where a similar problem is treated by perturbative methods in the simpler case of a very low temperature Ising interface above a wall.
Using this effectively repulsiveness allows us to make use of the ideas developed in~\cite{IofOttVelWac20} in the case of an interface above a wall. Whilst the underlying strategy is similar, several technical complications arise as the lower interface (playing the role of an ``effective wall'') is non-deterministic, and has unbounded fluctuations.
In short, due to entropic reasons, the interfaces stay far from one another.
This allows to use~\cite{DobGlaOtt25} to develop the renewal theory to each of them separately and then couple their joint law to a pair of random walks conditioned not to intersect.
The latter is known to converge to the Brownian watermelon in a general setting~\cite{DenWac15,DurWac20,DAl24} which finishes the argument.

An important technical issue is that the ATRC model is supported on pairs of percolation configurations and thus has a weaker domain Markov property.
This makes the study of the model significantly more complicated.

\subsection{Other models and boundary conditions}

As described above, most of our work deals with the ATRC model (defined in Section~\ref{sec:notations}), which is related to the Potts model via a chain of couplings. 
An intermediate step in these couplings is the height function of the six-vertex model (defined in Section~\ref{sec:combinatorial_mappings}).
Thus, our results also have implications for this height function and for the Ashkin--Teller model.

{\bf Six-vertex model.}
Originally introduced by Pauling~\cite{Pau35} in~1935 on~$\Z^3$ to describe the residual entropy of ice, the six-vertex model was then extended to~$\Z^2$ where it can be naturally described via a height function: each face of~$[-N,N]^2$ is assigned an integer height in such a way that, between any two adjacent faces, the height differs by exactly one; each vertex surrounded by heights~$(h,h+1,h,h+1)$ for some~$h$ receives a particular weight~$\mathbf{c}$, and the probability distribution is proportional to the product of the weights.
The model has received a lot of attention due to its integrability properties.
The regime of the parameters considered corresponds to~$\mathbf{c}>2$ where the height function is known to localise~\cite{DumGagHar21,GlaPel23,RaySpi20} under flat boundary conditions: the variance of the height at a given face is bounded uniformly in~$N$.
Now consider the boundary conditions making two steps: $h(x,y)\in \{2,3\}$ for all~$y>0$ and~$h(x,y)\in \{0,1\}$ for all~$y<0$.
This forces the existence of two level lines: one between $h=3$ and~$h=1$; and the other between $h=2$ and~$h=0$.
We show that these level lines converge to the Brownian watermelon in the diffusive limit.

{\bf Ashkin--Teller model.} Named after two physicists who introduced~\cite{AshTel43} it in~1943, the Ashkin--Teller model can be viewed as a pair of interacting Ising models~$\tau,\tau'$: coupling constants~$J,U > 0$ describe the strength of the interaction within the models and between them respectively.
The regime of a discontinuous transition in the Potts model corresponds to the part of the self-dual curve~$\sinh (2J) = e^{-2U}$ when~$U>J$.
In this regime  of parameters, each of~$\tau$ and~$\tau'$ each exhibits exponential decay of correlations while~$\tau\tau'$ is ferromagnetically ordered~\cite{GlaPel23,AouDobGla24}.
Pfister and Velenik~\cite{PfiVel97} introduced a graphical representation of the Ashkin--Teller model.
This is called the ATRC model and is supported on pairs of bond percolation configurations $\omega_\tau$ and~$\omega_{\tau\tau'}$ describing correlations in~$\tau$ and in~$\tau\tau'$ respectively.
We consider the ATRC model on a box with wired boundary conditions for~$\omega_{\tau\tau'}$ and a suitably modified boundary coupling constant for~$\omega_\tau$ (see Definition~\ref{def:mATRC}) and condition on the existence of crossings: in~$\omega_\tau$ and~$\omega_{\tau\tau'}^*$ linking the midpoints of the left and right sides of the box.
We show that the pair of the clusters of~$\omega_\tau$ and~$\omega_{\tau\tau'}^*$ containing these crossings converge to the Brownian watermelon in the diffusive limit.

{\bf Potts model at~$T_c$ with tricolor boundary conditions.}
In the current work, we are dealing only with bicolor boundary conditions.
We believe that our analysis can be extended to more general boundary conditions.
For instance, consider the Potts model on a triangular domain of size~$N$ with boundary conditions given by~$1$, $2$, $3$ assigned to sides of the domain respectively.
Then, we expect that, almost the whole domain will look like a sample of a free measure, whilst the monochromatic boundary clusters will be limited to neighbourhoods of the boundary of width~$\sqrt{N}$. 
This should be compared to the case~$T<T_c$ where the same boundary conditions result in the domain being divided into three monochromatic parts.
There are two obstacles in extending our methods to this setting: 1) the BKW coupling to the six-vertex model does not seem to apply directly; 2) one needs to prove that in a suitably modified ATRC model, the surface tension along the boundary is not smaller than in the bulk.

\subsection{Derivation of Theorem~\ref{thm:invariance_princ_potts} and~\ref{thm:invariance_princ_FK}}
\label{subsec:proof-thms}

In this section, we give the proofs of the two main Theorems, relying on the rest of the paper. In particular, this gives a global plan of the paper, and describes the main output of the different sections.

\begin{proof}[Proof of Theorem~\ref{thm:invariance_princ_potts}]
	Theorem~\ref{thm:invariance_princ_potts} follows from Theorem~\ref{thm:invariance_princ_FK}, exponential decay of FK connections under free boundary conditions, and the Edward-Sokal coupling exactly as in~\cite[Section 9.2]{DobGlaOtt25}.
\end{proof}

\begin{proof}[Proof of Theorem~\ref{thm:invariance_princ_FK}]
	The proof goes in three steps:
	
	{\bf Step 1.} First, one can couple \(\omega\) sampled under \(\fk_{\Lambda_n}^{1/1}(\cdot \given \partial^+_n\nleftrightarrow \partial^-_n)\) to a modified ATRC random variable, \((\omega_{\tau},\omega_{\tau\tau'})\), conditioned on two connections events: a primal one, \(\{(-n-1,0)\xleftrightarrow{\omega_{\tau}} (n+1,0)\}\), and a dual one, \(\{(-n-\frac{3}{2}, -\frac{1}{2}) \xleftrightarrow{\omega_{\tau\tau'}^*} (n+\frac{3}{2}, -\frac{1}{2}) \}\). Denoting \(\calC,\calC'\) the primal, resp. dual, connected components realizing the connection events, the coupling is done in such a way that the one-sided Hausdorff distance between \(\Gamma_{\fk}^{1\pm, n} \), and \(\calC\), as well as the one-sided Hausdorff distance between \(\Gamma_{\fk}^{2\pm, n}\) and \(\calC'\), are at most logarithmic in \(n\) with probability \(1-o_n(1)\). This is the content of Lemmas~\ref{lem:proxi_upper_int}, and~\ref{lem:proxi_lower_int}.
	
	{\bf Step 2.} Then, \((\calC,\calC')\) can be coupled to a mixture of random walk bridges \(S,S'\) conditioned on mutual avoidance in such a way that the one-sided Hausdorff distance between \(\calC\) and \(S\), as well as the one-sided Hausdorff distance between \(\calC'\) and \(S'\), are at most \(\ln^{111}(n)\) with probability \(1-o_n(1)\): this is the content of Corollary~\ref{cor:coupling_with_avoiding_bridges}.
	
	{\bf Step 3.} Finally, Theorem~\ref{thm:InvPrinc_AvoidingBridges} implies that the diffusively rescaled linear interpolation of \((S,S')\) converges towards a Brownian watermelon.
\end{proof}

\subsubsection*{Acknowledgements}
Parts of the work were accomplished during the visits of some of us to the Universities of Fribourg, Geneva, Innsbruck and Vienna and at the NCCR SwissMAP research station in Les Diablerets. We want to thank these institutions and all people involved for their hospitality and creating excellent working conditions.
This research was partially funded by the Austrian Science Fund (FWF) 10.55776/P34713 and 10.55776/AST1405524.

\subsubsection*{AI disclosure}
AI/LLMs have not been used in this work.

\begin{figure}
	\centering
	\begin{tikzpicture}
		\draw (0,0)--(1,1)--(0,2)--(-1,1)--(0,0);
		\draw[ultra thick] (1,1)--(-1,1);
		\draw[ultra thick, dashed] (0,0)--(0,2);
		\draw (3,0)--(4,1)--(3,2)--(2,1)--(3,0);
		\draw[ultra thick] (3,0)--(3,2);
		\draw[ultra thick, dashed] (4,1)--(2,1);
		\filldraw[fill=white] (0,2) circle(3pt);
		\filldraw[fill=white] (0,0) circle(3pt);
		\filldraw[fill=black] (-1,1) circle(3pt);
		\filldraw[fill=black] (1,1) circle(3pt);
		\filldraw[fill=white] (2,1) circle(3pt);
		\filldraw[fill=white] (4,1) circle(3pt);
		\filldraw[fill=black] (3,2) circle(3pt);
		\filldraw[fill=black] (3,0) circle(3pt);
	\end{tikzpicture}
	\hspace*{1.5cm}
	\begin{tikzpicture}[scale=2]
		\draw (0,-0.5)--(0.5,0)--(0,0.5)--(-0.5,0)--(0,-0.5);
		\draw (1.5,-0.5)--(2,0)--(1.5,0.5)--(1,0)--(1.5,-0.5);
		
		\Vloop{1.5}{0}
		\Hloop{0}{0}
		
		\draw[very thick] (-0.5,0)--(0.5,0); 
		\draw[very thick, dashed] (1.5,-0.5)--(1.5,0.5);
		
		\foreach \i in {0,1.5}{
			\filldraw[fill=black] ({\i-0.5},0) circle(2pt) ;
			\filldraw[fill=black] ({\i+0.5},0) circle(2pt) ;			
			\filldraw[fill=white] (\i,-0.5) circle(2pt) ;
			\filldraw[fill=white] (\i,0.5) circle(2pt) ;
		}
	\end{tikzpicture}
	\caption{Left: Tile associated to a mid-edge. Right: Tile centred at the middle of a horizontal primal edge (solid black) or its associated vertical dual edge (dashed black), with its two possible local loop configurations.}
	\label{fig:midEdgeTiles}
\end{figure}
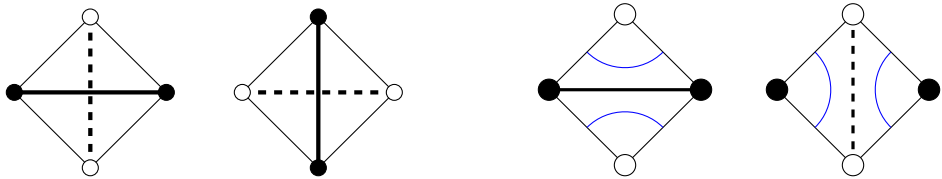

\section{Notations and conventions}
\label{sec:notations}

\vspace{5pt}

\noindent\textbf{General graphs.}
Let $\bbG=(\bbV,\bbE)$ be a graph. We simply write $xy=\{x,y\}$ for an edge $\{x,y\}\in\bbE$.
Given finite subsets $\Lambda\subset \bbV$ and $E\subset\bbE$, define
\begin{gather*}
	\bbE_\Lambda:=\{e\in\bbE :e\subset \Lambda\},\quad
	\bbV_E:=\bigcup_{e\in E}e,
\end{gather*}
\begin{gather*}
	\partialin\Lambda:=\{x\in \Lambda:\exists y\in\bbV\setminus \Lambda,\,xy\in\bbE\},\quad
	\partialex\Lambda:=\{y\in \bbV\setminus \Lambda:\exists x\in \Lambda,\,xy\in\bbE\},\\
	\partialin E:=\{e\in E:\exists f\in\bbE\setminus E, e\cap f\neq\varnothing\},\quad
	\partialex E:=\{e\in\bbE\setminus E:\exists f\in E, e\cap f\neq\varnothing\},\\
	\partialedge\Lambda:=\{xy\in\bbE:x\in\Lambda,y\in\bbV\setminus\Lambda\}.
\end{gather*}
In case of ambiguity, we add \(\bbG\) as a subscript (and write, for example, \(\partialin_\bbG\Lambda\)) to emphasise that the boundary is taken in \(\bbG\).
The \emph{interior} of \(\Lambda\) (in \(\bbG\)) is given by \(\Lambda\setminus\partialin\Lambda\).
The subgraph \emph{induced} by $\Lambda\subset\bbV$ is given by $(\Lambda,\bbE_\Lambda)$, and the subgraph induced by \(E\subset\bbE\) is given by \((\bbV_E,E)\).
\medskip

\begin{figure}
	\ifpic
	\begin{tikzpicture}[scale=0.8]
		\foreach \i in {-2,...,2}{
		\foreach \j in {-2,...,2}{
		\filldraw[fill=black] (\i,\j) circle(2pt);
		}
		}
		\foreach \i in {-3,...,2}{
		\foreach \j in {-3,...,2}{
		\filldraw[fill=white] ({\i+0.5},{\j+0.5}) circle(2pt);
		}
		}
		%to align with the right picture:
		{\filldraw[white](0,-3.5) circle(2pt);}
	\end{tikzpicture}
	\hspace{1cm}
	\begin{tikzpicture}[scale=0.8]
		\foreach \i\j in {-2.5/3, -1.5/3, -0.5/3, 0.5/3, 1.5/3, 2.5/3, -2.5/-3, -1.5/-3, -0.5/-3, 0.5/-3, 1.5/-3, 2.5/-3, -3/-2.5, -3/-1.5, -3/-0.5, -3/0.5, -3/1.5, -3/2.5, 3/-2.5, 3/-1.5, 3/-0.5, 3/0.5, 3/1.5, 3/2.5}{
		\filldraw[fill=lightgray] ({\i-0.5},\j)--(\i,{\j+0.5})--({\i+0.5},\j)--(\i,{\j-0.5})--({\i-0.5},\j);
		}
		
		\foreach \i in {-2,...,2}{
		\foreach \j in {-3,...,3}{
		\filldraw[fill=black] (\i,\j) circle(2pt);
		}
		}
		\foreach \i in {-3,3}{
		\foreach \j in {-3,...,3}{
		\filldraw[fill=black] (\i,\j) circle(2pt);
		}
		}
		\foreach \i in {-3,...,2}{
		\foreach \j in {-4,...,3}{
		\filldraw[fill=white] ({\i+0.5},{\j+0.5}) circle(2pt);
		}
		}
		\foreach \i in {-3,...,2}{
		\filldraw[fill=white] (-3.5,{\i+0.5}) circle(2pt);
		\filldraw[fill=white] (3.5,{\i+0.5}) circle(2pt);
		}
	\end{tikzpicture}
	\fi
	\caption{Left: the sets~\(\mathsf{B}_{2,2}\) and~\(\mathsf{B}'_{2,2}\). Right: the boundary tiles \( \partial \mathsf{A}_{2,2}\).}
	\label{fig:FK_tilesDomain}
\end{figure}
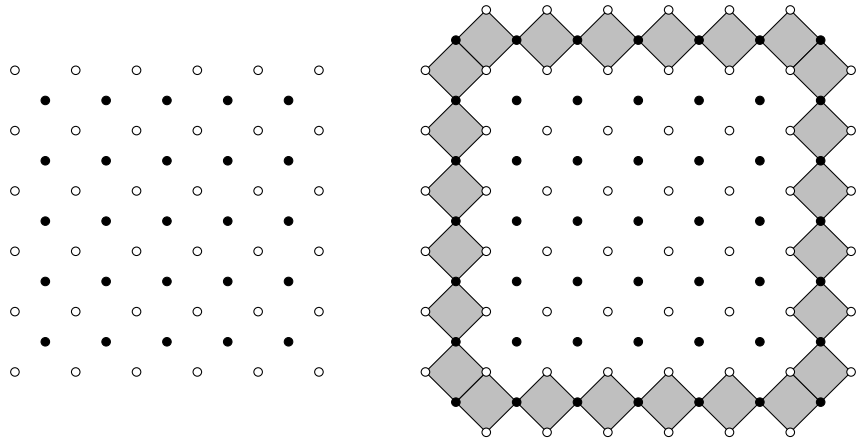

\noindent\textbf{Lattices.}
We will mainly work on \(\Z^2\) with nearest-neighbour edges, and on its dual. We will denote the \emph{primal lattice} by \(\bbL_{\bullet} = \Z^2\) and its {\em dual} by \(\bbL_{\circ}= (1/2,1/2)+\Z^2\). Denote by \(\bbE^{\bullet}\) the nearest-neighbour edges between sites in \(\bbL_{\bullet}\) (primal edges), by \(\bbE^{\circ}\) the nearest-neighbour edges between sites in \(\bbL_{\circ}\) (dual edges).

\medskip

\noindent\textbf{Duality.}
Each edge~\(e\in\bbE^{\bullet}\) intersects a unique edge of~\(\bbE^{\circ}\), we denote it by~\(e^*\).
For a set of primal (or dual) edges \(E\), define \(*E = \{e^*:\ e\in E\}\).
We say that $E\subset\bbE^\bullet$ is \emph{simply lattice-connected} if both the subgraphs induced by $E$ and by $*(\bbE^\bullet\setminus E)$ are connected.
As a convention, sets of edges or of dual edges will be identified with the corresponding sets of mid-points whenever the meaning is clear from the context.

\medskip

\noindent\textbf{Tiles.}
To each primal-dual pair of edges \(e,e^*\), associate a {\em tile} \(t\) given by the convex hull of their endpoints and define \(e_t:=e\); see Fig.~\ref{fig:midEdgeTiles}.
Define~\(\bbL_{\diamond}\) as the set of all tiles, and let~\(\bbE^\diamond\) be the set of all pairs of adjacent tiles.
Note that~\((\bbL_{\diamond},\bbE^{\diamond})\) is the medial graph of~\((\bbL_{\bullet},\bbE^{\bullet})\), since tiles can be identified with midpoints of edges in~\(\bbE^{\bullet}\).

\medskip

\noindent\textbf{Standard rectangular domains.}
Define the upper and lower half planes by 
\begin{equation*}
	\bbH^+ := \R\times \R_{\geq 0},\quad \bbH^- := \R\times \R_{<0},
\end{equation*}
and, for \(n,m\geq 0\), set (see the left side of Fig.~\ref{fig:FK_tilesDomain})
\begin{align*}
	\mathsf{B}_{n,m} &:=\{-n, \dots, n\}\times\{-m, \dots, m\},\\
	\mathsf{B}'_{n,m} &:=\{-n-\tfrac{1}{2},\dots,n+\tfrac{1}{2}\}\times \{-m-\tfrac{1}{2},\dots,m+\tfrac{1}{2}\}.
\end{align*}

Define the subgraph $\mathsf{G}_{n,m}=(\mathsf{V}_{n,m},\mathsf{E}_{n,m})$ of $\Z^2$ by (see Fig.~\ref{fig:WiredWired_FK})
\begin{equation*}
\mathsf{V}_{n,m}:=\mathsf{B}_{n,m}\cup\partialex\mathsf{B}_{n,m}\quad\text{and}\quad\mathsf{E}_{n,m}=\bbE_{\mathsf{B}_{n,m}}\cup\partialedge\mathsf{B}_{n,m}.
\end{equation*}

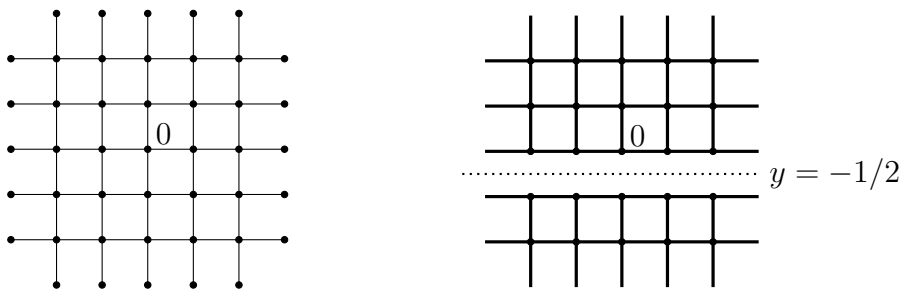
\begin{figure}
	\ifpic
	\begin{tikzpicture}[scale=0.6]
		\foreach \i in {-3,...,3}{
			\foreach \j in {-2,...,2}{
				\filldraw[black] (\i,\j) circle (2pt);
			}
		}
		\foreach \i in {-2,...,2}{
			\foreach \j in {-3,3}{
				\filldraw[black] (\i,\j) circle (2pt);
			}
		}
		\foreach \i in {-2,...,2}{
			\foreach \j in {-3,...,2}{
				\draw[black] (\i,\j) --(\i,{\j+1});
			}
		}
		\foreach \i in {-3,...,2}{
			\foreach \j in {-2,...,2}{
				\draw[black] (\i,\j) --({\i+1},\j);
			}
		}
		\draw (0.35,0.35) node{$0$};
	\end{tikzpicture}
	\hspace{2cm}
	\begin{tikzpicture}[scale=0.6]
		\foreach \i in {-2,...,2}{
			\foreach \j in {-2,...,2}{
				\filldraw[black] (\i,\j) circle (2pt);
			}
		}
		\foreach \i in {-3,...,2}{
			\foreach \j in {-1,-2,0,1,2}{
				\draw[very thick] (\i,\j)--({\i+1},\j) ;
			}
		}
		\foreach \i in {-2,...,2}{
			\foreach \j in {0,1,2}{
				\draw[very thick] (\i,\j)--(\i,{\j+1}) ;
			}
			\foreach \j in {-1,-2}{
				\draw[very thick] (\i,\j)--(\i,{\j-1}) ;
			}
		}
		\draw[thick, dotted] (-3.5,-0.5)--(3,-0.5);
		\draw (3,-0.5) node[right]{$y=-1/2$} ;
		\draw (0.35,0.35) node{$0$};
	\end{tikzpicture}
	\caption{Left: the graph \(\mathsf{G}_{2,2}\). Right: wired-wired Dobrushin boundary condition $\xi_{1/1}$.}
	\label{fig:WiredWired_FK}
\end{figure}

Moreover, set~\(\mathsf{D}_{n,m}:=\mathsf{B}_{n,m}\cup\mathsf{B}'_{n,m}\), and let~\(\mathsf{A}_{n,m}\subset\bbL_\diamond\) be the set of tiles with at least one corner in $\mathsf{D}_{n,m}$ and~\(\partial \mathsf{A}_{n,m}\subset \mathsf{A}_{n,m}\) (boundary tiles) be the set of tiles with precisely one corner in \(\mathsf{D}\); see the right side of Fig.~\ref{fig:FK_tilesDomain}.

Define also (see Figs.~\ref{fig:Kn}~and~\ref{fig:bc_FK_loops_spins_o-o})
\begin{gather*}
	\partial^\pm_{n,m}:=(\partialex \mathsf{B}_{n,m})\cap\bbH^\pm,\\
	v_L(n):= (-n-1,0),\quad v_R(n):=(n+1,0),\\
	v'_L(n):= (-n-\tfrac{3}{2},-\tfrac{1}{2}),\quad v'_R(n):=(n+\tfrac{3}{2},-\tfrac{1}{2}).
\end{gather*}

\medskip

\noindent\textbf{Augmented rectangular domains.}
We also introduce~\(\mathsf{E}_{n,m}\) augmented by some boundary edges. First, define the set of boundary edges by
\begin{equation*}
	\mathsf{E}_{\rmb,n,m}:=\{e_t:t\in\partial \mathsf{A}_{n,m}\}=\bbE_{\mathsf{B}_{n+1,m+1}}\setminus\mathsf{E}_{n,m}.
\end{equation*}
Then define the augmented sets of edges and vertices by
\begin{equation*}
\bar{\mathsf{E}}_{n,m}:=\mathsf{E}_{n,m} \cup \mathsf{E}_{\rmb,n,m}=\bbE_{\mathsf{B}_{n+1,m+1}}\qquad\text{and}\qquad\bar{\mathsf{V}}_{n,m}:=\bbV_{\bar{\mathsf{E}}_{n,m}}=\mathsf{B}_{n+1,m+1},
\end{equation*}
and the corresponding augmented graphs (see Fig.~\ref{fig:Kn}) as follows:
\begin{itemize}
\item set~\(\mathsf{K}_{n,m}:=(\bar{\mathsf{V}}_{n,m},\bar{\mathsf{E}}_{n,m})\), and let $\mathsf{K}_{n,m}^1$ be the graph obtained from $\mathsf{K}_{n,m}$ by identifying the vertices in $\partialin\bar{\mathsf{V}}_{n,m}\cap\bbH^+$ 
and those in $\partialin\bar{\mathsf{V}}_{n,m}\cap\bbH^-$
,
\item set~\(\mathsf{K}_{n,m}':=(\bbV_{*\bar{\mathsf{E}}_{n,m}},*\bar{\mathsf{E}}_{n,m})\), and let $(\mathsf{K}_{n,m}')^1$ be the graph obtained from $\mathsf{K}_{n,m}'$ by identifying the vertices in $\partialin\bbV_{*\bar{\mathsf{E}}_{n,m}}\cap\bbH^+$ and those in $\partialin\bbV_{*\bar{\mathsf{E}}_{n,m}}\cap\bbH^-$.
\end{itemize}

\medskip

\noindent\textbf{Connectivity events.}
Given $\Lambda,\Delta\subset\bbL_\bullet$ and a percolation configuration~$\omega \in \{0,1\}^{\bbE^\bullet}$, we write $\Lambda\xleftrightarrow{\omega}\Delta$ for the event that~$\Lambda$ and $\Delta$ are connected by a path in the graph~$(\bbL_\bullet,\omega)$.
If~$\Lambda=\{i\}$ and~$\Delta=\{j\}$, we simply write~$i \xleftrightarrow{\omega} j$. 
We omit $\omega$ from the notation when it cannot lead to any confusion. 	

\medskip

\noindent\textbf{Agreements of spins along edges.}
For an edge~$e=ij\in \mathbb{E}^\bullet$ and a spin configuration~$\sigma_\bullet\in \{+1,-1\}^{\bbL_\bullet}$, we write~$\sigma_\bullet \sim e$ if~$\sigma_\bullet(i)=\sigma_\bullet(j)$;
for~$\xi\subset\mathbb{E}^\bullet$, we write~$\sigma_\bullet \sim \xi$ if~$\sigma_\bullet\sim e$ for every~$e\in\xi$ (in other words, $\sigma_\bullet$ is constant on clusters of~$\xi$).
We use similar notation for edges in~$\mathbb{E}^\circ$ and~$\sigma_\circ\in \{+1,-1\}^{\bbL_\circ}$.

\medskip

\noindent\textbf{Parameters.}
The couplings go through standard ``expansion--resummation'' of Boltzmann weights combined with extensive use of planarity. As for each step one will write the weight associated with a given model as a sum of weights for a ``more expanded'' model, several parameters will come into play.
Below we list the parameters, as well as the algebraic relations linking them:
\begin{gather}
	q>4,\quad p=p_c(q)=\tfrac{\sqrt{q}}{1+\sqrt{q}},\quad \lambda>0,\quad \svc>2,\quad U>J>0,\nonumber\\
	\sqrt{q} = e^{\lambda} + e^{-\lambda},\quad \svc = e^{\lambda/2} + e^{-\lambda/2} = \coth(2J),\label{eq:parameters_bulk}\\
	\sinh(2J) = e^{-2U}.\nonumber
\end{gather}
The parameter \(\svc\) also has a ``boundary version'':
\begin{equation}
	\svcb = e^{\lambda/2}>1.\label{eq:parameters_bnd}
\end{equation}
For the remainder of the article, we fix \( q > 4 \), along with the corresponding parameters above (which are uniquely determined by \( q \)).

\medskip

\noindent\textbf{Constants.}
Constants like \(c,c_1,C,C_1,C',\dots\) are constants which can change from line to line and which can depend on the parameters unless explicitly stated. They are independent of the system size, \(n\), which will be our main ``variable'' quantity.

\medskip

\noindent\textbf{ATRC model.}
Let~$\xi_\tau,\xi_{\tau\tau'}\in\{0,1\}^{\bbE}$ with $\xi_\tau\subseteq\xi_{\tau\tau'}$.
The ATRC model on~$G$ with parameters~$U>J>0$ under boundary conditions~$(\xi_\tau,\xi_{\tau\tau'})$ is the probability measure on $\{0,1\}^{\bbE}\times\{0,1\}^{\bbE}$ given by 
\begin{equation}
\begin{multlined}
\label{eq:atrc_def}
	\atrc_{G;J,U}^{\xi_\tau,\xi_{\tau\tau'}}(\omega_\tau,\omega_{\tau\tau'})
	= \tfrac{1}{Z} \cdot \ind_{\omega_\tau\subseteq\omega_{\tau\tau'}} \cdot  \rmw_\tau^{\abs{\omega_\tau\cap E}}\, \rmw_{\tau\tau'}^{\abs{(\omega_{\tau\tau'}\setminus\omega_\tau)\cap E}}\,2^{\clusters_{V}(\omega_\tau)+\clusters_{V}(\omega_{\tau\tau'})}\\
	\cdot\ind_{\omega_\tau=\xi_{\tau}\text{ on }\bbE\setminus E}\,\ind_{\omega_{\tau\tau'}=\xi_{\tau\tau'}\text{ on }\bbE\setminus E},
\end{multlined}
\end{equation}
where~$Z=Z_{\atrc}(G,J,U,\xi_\tau,\xi_{\tau\tau'})$ is the partition function, $\clusters_{V}(\cdot)$ is the number of clusters that intersect $V$, and the weights are given by
\begin{equation}\label{eq:atrc_weights}
	\rmw_\tau=e^{2U}(e^{2J}-e^{-2J})\quad\text{and}\quad \rmw_{\tau\tau'}=e^{2(U-J)}-1.
\end{equation}
When $\xi_\tau\equiv 0$ (resp. $\xi_\tau\equiv 1$), we write $\atrcfree$ (resp. $\atrcwired$) instead of $\xi_\tau$ in the superscript, and analogously for $\xi_{\tau\tau'}$.
Given finite subsets $\Lambda\subset\bbV$ and $E\subset\bbE$, we write $\atrc_{\Lambda;J,U}^{\xi_\tau,\xi_{\tau\tau'}}$ and $\atrc_{E;J,U}^{\xi_\tau,\xi_{\tau\tau'}}$ for the measures on the graphs $(\Lambda,\bbE_\Lambda)$ and $(\bbV_E,E)$, respectively.

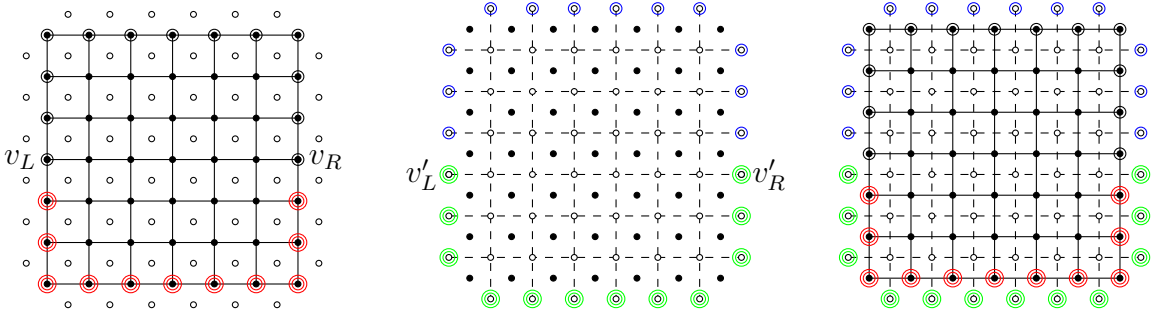
\begin{figure}
	\begin{tikzpicture}[scale=0.55]
		\foreach \j in {-3,...,3}{
			\draw[black] (-3,\j)--(3,\j);
			\draw[black] (\j,3)--(\j,-3);
		}
		
		\foreach \i in {-2,...,2}{
			\draw (\i,3) circle(4pt);
			\draw[red] (\i,-3) circle(4pt);
			\draw[red] (\i,-3) circle(6pt);
		}
		
		\foreach \j in {-3,...,-1}{
			\draw[red] (-3,\j) circle(4pt);
			\draw[red] (3,\j) circle(4pt);
			\draw[red] (-3,\j) circle(6pt);
			\draw[red] (3,\j) circle(6pt);
		}
		
		\foreach \j in {0,...,3}{
			\draw (-3,\j) circle(4pt);
			\draw (3,\j) circle(4pt);
		}
		
		\foreach \i in {-3,...,3}{
			\foreach \j in {-3,...,3}{
				\filldraw[fill=black] (\i,\j) circle(2pt);
			}
		}
		\foreach \i in {-4,...,3}{
			\foreach \j in {-3,...,2}{
				\filldraw[fill=white] ({\i+0.5},{\j+0.5}) circle(2pt);
			}
		}
		\foreach \i in {-3,...,2}{
			\foreach \j in {-4,3}{
				\filldraw[fill=white] ({\i+0.5},{\j+0.5}) circle(2pt);
			}
		}
		\filldraw[fill=black] (-3,0) circle(2pt) node[left]{$v_L$};
		\filldraw[fill=black] (3,0) circle(2pt) node[right]{$v_R$};
	\end{tikzpicture}
	\hspace{8pt}
	\begin{tikzpicture}[scale=0.55]
		\foreach \i in {-3,...,2}{
			\draw[dashed] ({\i+0.5},3.5)--({\i+0.5},-3.5);
		}
		\foreach \j in {-0.5,-1.5,-2.5,0.5,1.5,2.5}{
			\draw[dashed] (-3.5,\j)--(3.5,\j);
		}
		\foreach \j in {}{
			\draw[dashed] (-2.5,\j)--(2.5,\j);
		}
		
		\foreach \i in {-3,...,3}{
			\foreach \j in {-3,...,3}{
				\filldraw[fill=black] (\i,\j) circle(2pt);
			}
		}
		\foreach \i in {-4,...,3}{
			\foreach \j in {-3,...,2}{
				\filldraw[fill=white] ({\i+0.5},{\j+0.5}) circle(2pt);
			}
		}
		\foreach \i in {-3,...,2}{
			\foreach \j in {-4,3}{
				\filldraw[fill=white] ({\i+0.5},{\j+0.5}) circle(2pt);
			}
		}
				
		\foreach \i in {-2.5,...,2.5}{
			\draw[blue] (\i,3.5) circle(4pt);
			\draw[green] (\i,-3.5) circle(4pt);
			\draw[green] (\i,-3.5) circle(6pt);
		}
		
		\foreach \j in {0.5,...,2.5}{
			\draw[blue] (-3.5,\j) circle(4pt);
			\draw[blue] (3.5,\j) circle(4pt);
		}
		
		\foreach \j in {-2.5,...,-0.5}{
			\draw[green] (-3.5,\j) circle(4pt);
			\draw[green] (3.5,\j) circle(4pt);
			\draw[green] (-3.5,\j) circle(6pt);
			\draw[green] (3.5,\j) circle(6pt);
		}
		\filldraw[fill=white] (-3.5,-0.5) circle(2pt) node[left]{$v'_L$};
		\filldraw[fill=white] (3.5,-0.5) circle(2pt) node[right]{$v'_R$};
	\end{tikzpicture}
	\hspace{8pt}
	\begin{tikzpicture}[scale=0.55]
		\foreach \i in {-3,...,2}{
			\draw[dashed] ({\i+0.5},3.5)--({\i+0.5},-3.5);
		}
		\foreach \j in {-0.5,-1.5,-2.5,0.5,1.5,2.5}{
			\draw[dashed] (-3.5,\j)--(3.5,\j);
		}
		\foreach \j in {}{
			\draw[dashed] (-2.5,\j)--(2.5,\j);
		}
		
		\foreach \j in {-3,...,3}{
			\draw[black] (-3,\j)--(3,\j);
			\draw[black] (\j,3)--(\j,-3);
		}
		
		\foreach \i in {-2,...,2}{
			\draw (\i,3) circle(4pt);
			\draw[red] (\i,-3) circle(4pt);
			\draw[red] (\i,-3) circle(6pt);
		}
		
		\foreach \j in {-3,...,-1}{
			\draw[red] (-3,\j) circle(4pt);
			\draw[red] (3,\j) circle(4pt);
			\draw[red] (-3,\j) circle(6pt);
			\draw[red] (3,\j) circle(6pt);
		}
		
		\foreach \j in {0,...,3}{
			\draw (-3,\j) circle(4pt);
			\draw (3,\j) circle(4pt);
		}
		
		\foreach \i in {-3,...,3}{
			\foreach \j in {-3,...,3}{
				\filldraw[fill=black] (\i,\j) circle(2pt);
			}
		}
		\foreach \i in {-4,...,3}{
			\foreach \j in {-3,...,2}{
				\filldraw[fill=white] ({\i+0.5},{\j+0.5}) circle(2pt);
			}
		}
		\foreach \i in {-3,...,2}{
			\foreach \j in {-4,3}{
				\filldraw[fill=white] ({\i+0.5},{\j+0.5}) circle(2pt);
			}
		}
		\foreach \i in {-2.5,...,2.5}{
			\draw[blue] (\i,3.5) circle(4pt);
			\draw[green] (\i,-3.5) circle(4pt);
			\draw[green] (\i,-3.5) circle(6pt);
		}
		
		\foreach \j in {0.5,...,2.5}{
			\draw[blue] (-3.5,\j) circle(4pt);
			\draw[blue] (3.5,\j) circle(4pt);
		}
		
		\foreach \j in {-2.5,...,-0.5}{
			\draw[green] (-3.5,\j) circle(4pt);
			\draw[green] (3.5,\j) circle(4pt);
			\draw[green] (-3.5,\j) circle(6pt);
			\draw[green] (3.5,\j) circle(6pt);
		}
	\end{tikzpicture}
	\caption{The graphs \(\mathsf{K}_{2,2}\) (left), \(\mathsf{K}_{2,2}'\) (middle), and the planar duality relation between their edges (right). Solid vertices surrounded by a black circle and solid vertices surrounded by two red circles are identified in $\mathsf{K}_{2,2}^1$, respectively. Hollow vertices surrounded by a blue circle and hollow vertices surrounded by two green circles are identified in $(\mathsf{K}_{2,2}')^1$, respectively.}
	\label{fig:Kn}
\end{figure}

\section{Couplings}
\label{sec:couplings}

This section is concerned with the construction of a coupling of the FK-percolation and the Ashkin--Teller models, via the six-vertex model, and the derivation of its basic properties. 
The coupling of the FK and six-vertex measures is an adaptation of the Baxter--Kelland--Wu (BKW) coupling to the Dobrushin boundary conditions~\cite{BaxKelWu76}.
The relation of the six-vertex and AT measures has first been noticed in~\cite{Fan72} comparing their critical properties, and it was made explicit in~\cite{Fan72b,Weg72} on a level of partition functions.
We build on~\cite{GlaPel23}, where a coupling of the six-vertex model and a graphical representation of the AT model (a marginal of ATRC) was constructed.

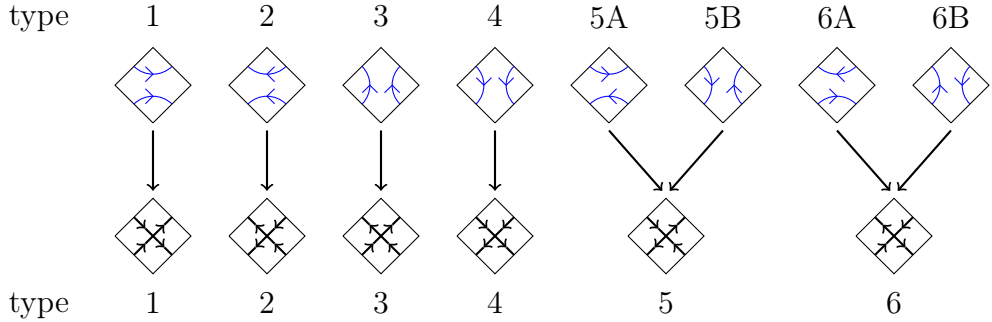
\begin{figure}
\centering
\ifpic
\begin{tikzpicture}
	\draw (-1.5,1.2) node[below]{type};
	\draw (0,1.2) node[below]{1};
	\draw (1.5,1.2) node[below]{2};
	\draw (3,1.2) node[below]{3};
	\draw (4.5,1.2) node[below]{4};
	\draw (6,1.2) node[below]{5A};
	\draw (7.5,1.2) node[below]{5B};
	\draw (9,1.2) node[below]{6A};
	\draw (10.5,1.2) node[below]{6B};

	\foreach \i in {0,1.5,3,4.5,6,7.5,9,10.5}{
		\draw (\i,-0.5)--(\i+0.5,0)--(\i,0.5)--(\i-0.5,0)--(\i,-0.5);
	}
	\foreach \i in {0,1.5,6,9}{
		\Hloop{\i}{0}
	}
	\foreach \i in {3,4.5,7.5,10.5}{
		\Vloop{\i}{0}
	}
		
	\RIGHTarrow{0}{0.5-1/(2*sqrt(2))}{blue}
	\RIGHTarrow{0}{1/(2*sqrt(2))-0.5}{blue}
	\LEFTarrow{1.5}{0.5-1/(2*sqrt(2))}{blue}
	\LEFTarrow{1.5}{1/(2*sqrt(2))-0.5}{blue}
	
	\UParrow{2.5+1/(2*sqrt(2))}{0}{blue}
	\UParrow{3.5-1/(2*sqrt(2))}{0}{blue}
	\DOWNarrow{4+1/(2*sqrt(2))}{0}{blue}
	\DOWNarrow{5-1/(2*sqrt(2))}{0}{blue}
	
	\LEFTarrow{6}{1/(2*sqrt(2))-0.5}{blue}
	\RIGHTarrow{6}{0.5-1/(2*sqrt(2))}{blue}	
	\DOWNarrow{7+1/(2*sqrt(2))}{0}{blue}
	\UParrow{8-1/(2*sqrt(2))}{0}{blue}
			
	\RIGHTarrow{9}{1/(2*sqrt(2))-0.5}{blue}
	\LEFTarrow{9}{0.5-1/(2*sqrt(2))}{blue}
	\UParrow{10+1/(2*sqrt(2))}{0}{blue}
	\DOWNarrow{11-1/(2*sqrt(2))}{0}{blue}	
	
	\foreach \i in {0,1.5,3,4.5,6.75,9.75}{
		\DrawTile{\i}{-2}{}
	}
	
	%% mapping
	\foreach \i in {0,1.5,3,4.5}{
		\draw[->,thick] (\i,-0.6)--(\i,-1.4);
	}
	\draw[->,thick] (6,-0.6)--(6.7,-1.4);
	\draw[->,thick] (7.5,-0.6)--(6.8,-1.4);
	\draw[->,thick] (9,-0.6)--(9.7,-1.4);
	\draw[->,thick] (10.5,-0.6)--(9.8,-1.4);
	
	%%types	
	
	\draw (-1.5,-2.6) node[below]{type};
	\draw (0,-2.6) node[below]{1};
	\draw (1.5,-2.6) node[below]{2};
	\draw (3,-2.6) node[below]{3};
	\draw (4.5,-2.6) node[below]{4};
	\draw (6.75,-2.6) node[below]{5};
	\draw (9.75,-2.6) node[below]{6};
	
	%% 6V arrows
	\foreach \i in {0,3,9.75}{
		\draw[thick] (\i,-2)--({\i-0.1},{-2-0.1});
		\draw[thick, <-] (\i-0.1,-2-0.1)--({\i-0.25},{-2-0.25});
	}

	\foreach \i in {1.5,4.5,6.75}{
		\draw[thick] (\i-0.15,-2-0.15)--({\i-0.25},{-2-0.25});
		\draw[thick, ->] (\i,-2)--({\i-0.15},{-2-0.15});
	}
	
	\foreach \i in {0,4.5,6.75}{
		\draw[thick] (\i,-2)--({\i-0.1},{-2+0.1});		
		\draw[thick, <-] (\i-0.1,-2+0.1)--({\i-0.25},{-2+0.25});
	}

	\foreach \i in {1.5,3,9.75}{
		\draw[thick] (\i-0.15,-2+0.15)--({\i-0.25},{-2+0.25});
		\draw[thick, ->] (\i,-2)--({\i-0.15},{-2+0.15});
	}
	
	\foreach \i in {1.5,3,6.75}{
		\draw[thick] (\i,-2)--({\i+0.1},{-2-0.1});
		\draw[thick, <-] (\i+0.1,-2-0.1)--({\i+0.25},{-2-0.25});
	}

	\foreach \i in {0,4.5,9.75}{
		\draw[thick] (\i+0.15,-2-0.15)--({\i+0.25},{-2-0.25});
		\draw[thick, ->] (\i,-2)--({\i+0.15},{-2-0.15});
	}
	
	\foreach \i in {1.5,4.5,9.75}{
		\draw[thick] (\i,-2)--({\i+0.1},{-2+0.1});
		\draw[thick, <-] (\i+0.1,-2+0.1)--({\i+0.25},{-2+0.25});
	}

	\foreach \i in {0,3,6.75}{
		\draw[thick] (\i+0.15,-2+0.15)--({\i+0.25},{-2+0.25});
		\draw[thick, ->] (\i,-2)--({\i+0.15},{-2+0.15});
	}
\end{tikzpicture}
\fi
\caption{Tiles of the oriented loop model and their types and weights, and the mapping from oriented loop arcs to six-vertex edge orientations.}
\label{fig:oriented_loop_arcs}
\end{figure}

\subsection{Different models and combinatorial mappings}
\label{sec:combinatorial_mappings}

This section provides an overview of the combinatorial objects that will be encountered, as well as a description of their relations.
We first discuss oriented loop configurations, which serve as an intermediate step in the BKW coupling of the FK-percolation and the six-vertex models.
This is followed by a description of two of the representations of the six-vertex model: edge orientations and spin configurations.

In Section~\ref{sec:intro}, we saw that percolation configurations are in bijection with configurations of (unoriented) loops and bi-infinite paths. To make the correspondence \(\omega\leftrightarrow \omega^* \leftrightarrow \ell=\loops(\omega)\) explicit, we can regard each of these models as an assignment of a local piece of drawing of an edge and two \emph{arcs} to lozenge tiles centred at the mid-edges as depicted in Fig.~\ref{fig:midEdgeTiles}. Clearly, retaining only either the primal or dual edges, or the arcs, provides complete information about all three.

\medskip

\noindent\textbf{Oriented loop configurations} are obtained from unoriented ones by assigning an orientation to each loop; formally, this is done by means of a sequence of independent uniform random variables on \([0,1]\) indexed by the set \(\mathcal{L}\) of all loops.
Alternatively, one can assign orientations to the loop arcs on each tile, subject to the constraint that the orientations of neighbouring arcs match. The eight local configurations that can occur at a tile are referred to as \emph{types}, see Fig.~\ref{fig:oriented_loop_arcs}.

\medskip

\noindent\textbf{The edge orientations of the six-vertex model}~\cite{Pau35,Rys63} are assignments of orientations to the edges in \(\bbE^\diamond\), obtained from oriented loop configurations via the natural  surjection; see Fig.~\ref{fig:oriented_loop_arcs}.
These edge orientations satisfy the {\em ice rule}: at every vertex, there are two incoming and two outgoing edges of~\(\bbE^\diamond\).
This constraint permits six possible local configurations at a tile, which are also called types; see Fig.~\ref{fig:oriented_loop_arcs}.
The local inverse operation can be considered as \emph{splitting} the oriented edges into two oriented loop arcs.
While tiles of types 1-4 permit a unique reconstruction of the loop arcs based on the edge orientations, there are two possibilities for tiles of types 5-6, giving the latter a special role. 

\medskip

\noindent\textbf{The six-vertex spin representation.}
The six-vertex model may be represented by pairs of spin-configurations \((\sigma_\bullet,\sigma_\circ)\in\{\pm1\}^{\bbL_\bullet}\times\{\pm1\}^{\bbL_\circ}\)~\cite{Wu71,KadWeg71,Lis22,GlaPel23}, obtained from the edge orientations via the following two-valued mapping.
Fix the value of \(\sigma_\bullet\) or \(\sigma_\circ\) at some arbitrary fixed vertex and proceed iteratively: given an edge \(e\in\bbE^\diamond\), denote the vertices that it separates by~$i\in\bbL_\bullet$ from a vertex~$u\in\bbL_\circ$; we impose~\(\sigma_\bullet(i)=\sigma_\circ(u)\) if~\(i\) is to the left side of \(e\) (with respect to its assigned orientation) and we impose~\(\sigma_\bullet(i)=-\sigma_\circ(u)\) otherwise; see Fig.~\ref{fig:six-vertex_arrows_spins_heights}.
The ice rule ensures that this mapping is well-defined.

This mapping is two-valued due to the liberty to choose the value of \(\sigma_\bullet\) or \(\sigma_\circ\) at one vertex, and the two images are related to each other by a global spin flip.
Furthermore, it is injective, meaning that the edge orientations can be reconstructed from the spins.
The type of a tile with respect to \((\sigma_\bullet,\sigma_\circ)\) is given by the type of the corresponding edge orientations; see Fig.~\ref{fig:six-vertex_arrows_spins_heights}.
It should also be noted that the ice rule can be translated as follows: for any tile \(t\in\bbL_\diamond\), either~\(\sigma_\bullet\) is constant on the endpoints of \(e_t\) or \(\sigma_\circ\) is constant on the endpoints of \(e_t^*\). Formally,
\begin{equation}\label{eq:ice-rule_spins}
\big(\sigma_\bullet(i)-\sigma_\bullet(j)\big)\big(\sigma_\circ(u)-\sigma_\circ(v)\big)=0\qquad\text{for any }t\in\bbL_\diamond\text{ with }e_t=ij,\,e_t^*=uv.
\end{equation}
In the context of spins, this property will henceforth be referred to as the ice rule. 

\begin{figure}
\centering
\includegraphics[scale=0.7]{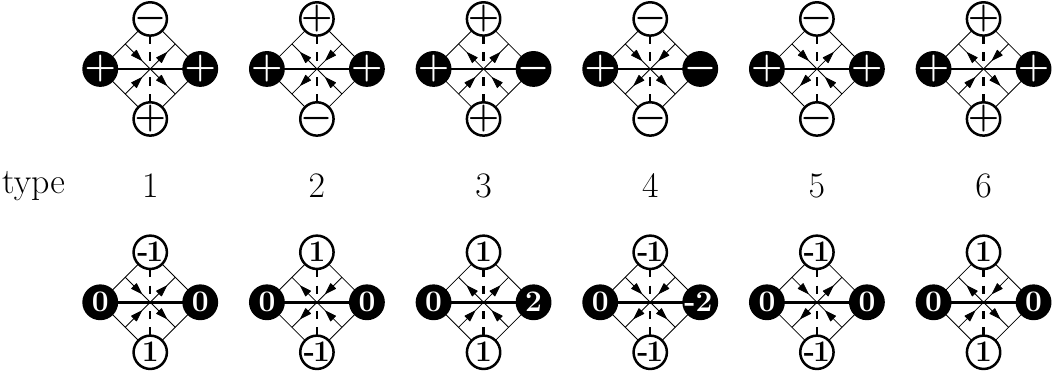}
\caption{The six-vertex types for all representations at a tile corresponding to a horizontal primal edge \(e\). Top: the spin at the left endpoint of \(e\) is fixed to be $+$. Bottom: the height at the left endpoint of \(e\) is fixed to be $0$.}
\label{fig:six-vertex_arrows_spins_heights}
\end{figure}

\medskip

\noindent\textbf{Baxter--Kelland--Wu correspondence~\cite{BaxKelWu76}.}
Given a measure on oriented loops, taking its pushforwards with respect to the above mappings, one obtains measures on the six-vertex edge orientations and spin configurations.

\subsection{Coupling under Dobrushin conditions}
\label{sec:coupling_under_dobr}

The idea is analogous to the order-disorder case in~\cite[Section 4.2]{DobGlaOtt25}. Take the FK measure with wired-wired Dobrushin boundary conditions and construct from it first a measure on pairs of spin configurations on \(\bbL_\bullet\) and \(\bbL_\circ\) that satisfy the ice rule and then a measure on ATRC configurations. We will then identify these two measures as the six-vertex and the ATRC measures, respectively, under suitable versions of Dobrushin boundary conditions. This gives a coupling between the FK and a modified ATRC, which will allow us to transfer the study of the former to the study of the latter.

Recall the definition of the subgraphs \(\mathsf{G}_{n,m}=(\mathsf{V}_{n,m},\mathsf{E}_{n,m})\) of \((\bbL_{\bullet},\bbE^{\bullet} )\) and the upper and lower boundaries~\(\partial^\pm_{n,m}\) given in Section~\ref{sec:notations}.
As \(n,m\) will be fixed in this section, we will omit them in the notation and simply write \(\mathsf{G}=(\mathsf{V},\mathsf{E})\) and~\(\partial^\pm\).
Recall also the definition of the FK measure~\(\fk_{\mathsf{G},p_c(q),q}^{1/1}\) on~\(\mathsf{G}\) with~\(q>4\) and under wired-wired Dobrushin boundary conditions in Section~\ref{sec:intro} (see Fig.~\ref{fig:WiredWired_FK}). Since~\(q\) and~\(p_c(q)\) will also be fixed, we will simply refer to it as~\(\fk_{\mathsf{G}}^{1/1}\).
Finally, recall the loop representation described in Section~\ref{sec:intro}.
For~\(\omega\in\{0,1\}^{\mathbb{E}^\bullet}\), define~\(\loops_{\mathsf{G}_{n,m}}(\omega)\) as the set of loops surrounding vertices in~\(\mathsf{V}_{n,m}\) obtained from~\(\omega\) by the mapping in Figure~\ref{fig:midEdgeTiles}.
\begin{lemma}\label{lem:FK_loop_expression}
If \(p = p_c(q) = \frac{\sqrt{q}}{1+\sqrt{q}}\), one has that
\begin{equation*}
\fk_{\mathsf{G}_{n,m};p,q}^{1/1}(\omega) \propto \sqrt{q}^{|\loops_{\mathsf{G}_{n,m}}(\omega)|}\,\mathds{1}_{\omega=\xi_{1/1}\text{ on }\bbE^\bullet\setminus \mathsf{E}_{n,m}}.
\end{equation*}
\end{lemma}

\noindent\textbf{Coupling measure.}
To couple \(\fk_\mathsf{G}^{1/1}(\,\cdot\given\partial^+\nleftrightarrow\partial^-)\) with both a six-vertex spin measure and a modified version of the ATRC measure, we augment our probability space to incorporate independent uniform \([0,1]\) random variables assigned to every loop and every tile.
Formally, let~\(\mathcal{L}\) be the set of all unoriented loops drawn on the tiles of~\(\bbL_{\diamond}\) (see the right of Fig.~\ref{fig:midEdgeTiles}).
Define~\(\Omega^{1/1} := \{0,1\}^{\bbE^\bullet}\times[0,1]^{\mathcal{L}}\times[0,1]^{\bbL_{\diamond}}\) equipped with the product of Borel sigma algebras.
Define \(Q\) and \(Q'\) respectively as the product measures on~\([0,1]^{\mathcal{L}}\) and~\([0,1]^{\bbL_{\diamond}}\).
Finally, the coupling measure is defined by
\[
	\Psi_\mathsf{G}^{1/1} := \fk_\mathsf{G}^{1/1}(\,\cdot\given\partial^+\nleftrightarrow\partial^-)\otimes Q\otimes Q'.
\]
In Lemmata~\ref{lem:fk11_to_sixv} and~\ref{lem:6V_spins_to_AT} below, we describe how to obtain the following two measures as push-forwards of~\(\Psi_\mathsf{G}^{1/1}\): the six-vertex spin measure and the modified ATRC measure, both under suitable Dobrushin boundary conditions (Definitions~\ref{def:six-vertex} and~\ref{def:mATRC}).

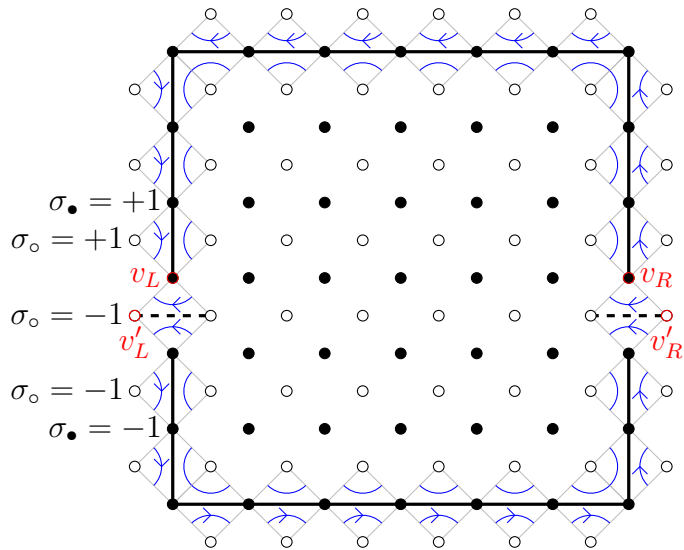
\begin{figure}
	\centering
	\begin{tikzpicture}[scale=1]
		%% vertical primal
		\foreach \i in {-3,3}{
			\foreach \j in {-3,-2,0,1,2}{
				\Vedge{\i}{\j+0.5}{very thick}
				\DrawTile{\i}{\j+0.5}{gray!50}
				\Vloop{\i}{\j+0.5}
			}
		}
		%% horizontal upper primal
		\foreach \i in {-3,...,2}{
			\foreach \j in {-3,3}{
				\Hedge{\i+0.5}{\j}{very thick}
				\DrawTile{\i+0.5}{\j}{gray!50}
				\Hloop{\i+0.5}{\j}
			}
		}
		%% interface insertion
		\foreach \i\j in {-3/-0.5, 3/-0.5}{
			\Hedge{\i}{\j}{dashed, very thick}
			\DrawTile{\i}{\j}{gray!50}
			\Hloop{\i}{\j}
			\LEFTarrow{\i}{\j-0.5+1/(2*sqrt(2))}{blue}
			\LEFTarrow{\i}{\j+0.5-1/(2*sqrt(2))}{blue}
		}
		
		%%%%% Orientations bc: half in
		\foreach \i in { -3,..., 2}{
			\LEFTarrow{\i+0.5}{3+0.5-1/(2*sqrt(2))}{blue}
			\RIGHTarrow{\i+0.5}{-3-0.5+1/(2*sqrt(2))}{blue}
		}
		\foreach \i in {-3,-2,0,1,2}{
			\UParrow{3+0.5-1/(2*sqrt(2))}{\i +0.5}{blue}
			\DOWNarrow{-3-0.5+1/(2*sqrt(2))}{\i +0.5}{blue}
		}

		%% primal sites
		\foreach \i in {-3,...,3}{
			\foreach \j in {-3,...,3} {
				\filldraw[black] (\i,\j) circle (2pt);
			}
		}
		%% dual sites
		\foreach \i in {-3,...,2}{
			\foreach \j in {-3,...,2} {
				\filldraw[fill=white] ({\i+0.5},{\j+0.5}) circle (2pt);
			}
		}
		\foreach \i in {-3,...,2}{
			\filldraw[fill=white] ({\i+0.5},3.5) circle (2pt);
			\filldraw[fill=white] ({\i+0.5},-3.5) circle (2pt);
			\filldraw[fill=white] (-3.5,{\i+0.5}) circle (2pt);
			\filldraw[fill=white] (3.5,{\i+0.5}) circle (2pt);
		}
		\draw (-3,1) node[left]{$\sigma_\bullet=+1$};
		\draw (-3.5,0.5) node[left]{$\sigma_\circ=+1$};
		\draw (-3.5,-0.5) node[left]{$\sigma_\circ=-1$};
		\draw (-3,-2) node[left]{$\sigma_\bullet=-1$};
		\draw (-3.5,-1.5) node[left]{$\sigma_\circ=-1$};
		\draw[red] (-3,0) circle(2pt) node[left]{$v_L$};
		\draw[red] (3,0) circle(2pt) node[right]{$v_R$};
		\draw[red] (-3.5,-0.5) circle(2pt) node[below]{$v'_L$};
		\draw[red] (3.5,-0.5) circle(2pt) node[below]{$v'_R$};
	\end{tikzpicture}
	\caption{Boundary conditions on oriented loops and on spins.}
	\label{fig:bc_FK_loops_spins_o-o}
\end{figure}

\subsubsection{From FK to six-vertex}
\label{sec:coupling:fk_to_sixv}
Recall the parameters~\(\svc\), \(\svcb\), and~\(\lambda\) from~\eqref{eq:parameters_bulk} and the standard rectangular domains from Section~\ref{sec:notations}. Below we omit~\(n,m\) everywhere.
\begin{definition}\label{def:six-vertex}
	The six-vertex spin model under \emph{double} Dobrushin conditions is a probability measure on \(\sigma=(\sigma_\bullet,\sigma_\circ)\in \{\pm 1\}^{\bbL_{\bullet}}\times \{\pm 1\}^{\bbL_{\circ}}\) defined by
	\begin{equation}
		\label{eq:def_6v-spin_dobrushin}
		\spin_\mathsf{D}^{+-,+-}(\sigma) \propto
		\svc^{|T_{5,6}^{\rmi}(\sigma)|}\,\svcb^{|T_{5,6}^{\rmb}(\sigma)|}\,\mathds{1}_{\Sigma_\mathsf{B}^{+-}\times\Sigma_{\mathsf{B}'}^{+-}}(\sigma)\,\mathds{1}_{\mathrm{ice}}(\sigma),
	\end{equation}
	where $\svc,\svcb$ are defined by~\eqref{eq:parameters_bulk}-\eqref{eq:parameters_bnd}, \(T_{5,6}^{\rmi}\) and \(T_{5,6}^{\rmb}\) are the sets of tiles of types 5-6 in \(\mathsf{A}^\rmi\) and \(\partial \mathsf{A}\), respectively (see Figures~\ref{fig:FK_tilesDomain} and~\ref{fig:six-vertex_arrows_spins_heights}), \(\mathds{1}_{\mathrm{ice}}\) is the indicator imposing the ice rule~\eqref{eq:ice-rule_spins}, and~\(\Sigma_{\mathsf{B}}^{+-}\) and~\(\Sigma_{\mathsf{B}'}^{+-}\) are the boundary conditions defined by
	\begin{align*}
	\Sigma_\mathsf{B}^{+-}&:=\{\sigma_\bullet\in\{\pm 1\}^{\bbL_{\bullet}}:\sigma_\bullet(i)=\mathds{1}_{\bbH^+}(i)-\mathds{1}_{\bbH^-}(i)\,\forall i\in\bbL_\bullet\setminus\mathsf{B}\},\\
	\Sigma_{\mathsf{B}'}^{+-}&:=\{\sigma_\circ\in\{\pm 1\}^{\bbL_{\circ}}:\sigma_\circ(u)=\mathds{1}_{\bbH^+}(u)-\mathds{1}_{\bbH^-}(u)\,\forall u\in\bbL_\circ\setminus\mathsf{B}'\}.
	\end{align*}	
\end{definition}
We now describe how to obtain~\(\spin_\mathsf{D}^{+-,+-}\) as a push-forward of~\(\Psi_\mathsf{G}^{1/1}\).

\begin{lemma}
\label{lem:fk11_to_sixv}
	Let \((\omega,U,U')\) be distributed according to \(\Psi_\mathsf{G}^{1/1}\), and let~\(\ell\) be the unoriented loop configuration associated to~$\omega$.
	Orient the loops of~\(\ell\) as follows (see Fig.~\ref{fig:bc_FK_loops_spins_o-o}):
	\begin{itemize}
		\item each loop~\(l\in \ell\) outside of~\(\mathsf{G}\) (i.e. surrounding a vertex in~\(\bbL_\circ\setminus\mathsf{B}'\)) is oriented clockwise;
		\item the two bi-infinite paths are oriented from right to left;
		\item each loop~\(l\in \ell\) inside of~\(\mathsf{G}\) (i.e. surrounding a vertex in \(\mathsf{B}\cup\mathsf{B}'\)) is oriented 
		clockwise if \(U_l<e^{\lambda}/\svc\) and counter-clockwise otherwise.
	\end{itemize}
	Denote by \(\ell_\shortrightarrow\) the obtained oriented loop configuration.
	Recall the combinatorial mappings introduced above, and let \((\sigma_\bullet,\sigma_\circ)\) be the associated six-vertex spin configurations with \(\sigma_\bullet((n+1,0))=+1\).
	Then, the law of \((\sigma_\bullet,\sigma_\circ)\) is given by \(\spin_\mathsf{D}^{+-,+-}\).
\end{lemma}

\begin{proof}
We adapt the proof of the order-disorder case, following the ideas of~\cite{BaxKelWu76}.
One has to examine which values of \((\omega, U)\) result in a given \((\sigma_\bullet,\sigma_\circ)\in \{\pm1\}^{\bbL_\bullet}\times\{\pm1\}^{\bbL_\circ}\).
The probability to obtain \((\sigma_\bullet,\sigma_\circ)\in\Sigma_\mathsf{B}^{+-}\times\Sigma_{\mathsf{B}'}^{+-}\) is the sum of the probabilities of all oriented loop configurations \(\ell_\shortrightarrow\) that induce the edge orientations corresponding to \((\sigma_\bullet,\sigma_\circ)\). The probability of a given oriented loop configuration \(\ell_\shortrightarrow\) satisyfing the boundary conditions in Fig.~\ref{fig:bc_FK_loops_spins_o-o} is proportional to
\begin{equation}\label{eq:prf:bkw3}
\sqrt{q}^{\,|\loops(\ell_\shortrightarrow)|}\cdot\Big(\tfrac{e^\lambda}{e^\lambda+e^{-\lambda}}\Big)^{|\loops_{\circlearrowright}(\ell_\shortrightarrow)|- |\loops_{\circlearrowleft}(\ell_\shortrightarrow)|}
\propto e^{\lambda(|\loops_{\circlearrowright}(\ell_\shortrightarrow)|- |\loops_{\circlearrowleft}(\ell_\shortrightarrow)|)},
\end{equation}
where \(\loops_{\circlearrowright}(\ell_\shortrightarrow)\) and \(\loops_{\circlearrowleft}(\ell_\shortrightarrow)\) are respectively the sets of clockwise and counter-clockwise oriented loops in \(\ell_\shortrightarrow\) not imposed by boundary conditions, and \(\loops(\ell_\shortrightarrow)\) is their union.
Indeed, by Lemma~\ref{lem:FK_loop_expression}, the first factor on the left side is proportional to \(\fk_\mathsf{G}^{1/1}(\omega(\ell))\), where \(\omega(\ell)\in\{0,1\}^{\bbE^\bullet}\) is the percolation configuration associated to the unoriented loop configuration \(\ell\) corresponding to \(\ell_\shortrightarrow\). The second factor comes from the values of the uniforms necessary to obtain the correct orientations of loops in~\(\loops(\ell_\shortrightarrow)\). The proportionality holds since \(\sqrt{q}=e^\lambda+e^{-\lambda}\) due to the choice of \(\lambda\).

Notice that each loop which is oriented clockwise does 4 more right quarter-turns than left quarter-turns, that the converse holds for counter-clockwise oriented loops, and that the numbers of left and right quarter-turns in the two bi-infinite paths differ by a universal constant.
Thus, the expression on the right side of~\eqref{eq:prf:bkw3} is proportional to
\begin{equation*}
\exp(\lambda (\#_{\curvearrowright}(\ell_\shortrightarrow)- \#_{\curvearrowleft}(\ell_\shortrightarrow))/4),
\end{equation*}
where \(\#_{\curvearrowright}(\ell_\shortrightarrow)\) and \(\#_{\curvearrowleft}(\ell_\shortrightarrow)\) are respectively the number of right and left quarter-turns in \(\ell_\shortrightarrow\).
The key idea is to count these oriented loop arcs locally at each tile in \(\mathsf{A}=\mathsf{A}^\rmi\cup \mathsf{A}^\rmb\).
Observe that, for tiles in \(\mathsf{A}^\rmi\), types 5B,6A correspond to a pair of right-oriented loop arcs and types 5A,6B to a pair of left-oriented loop arcs, whereas types 1-4 correspond to one right-oriented and one left-oriented loop arc each.
Moreover, due to the boundary conditions (see Fig.~\ref{fig:bc_FK_loops_spins_o-o}), a tile in \(\mathsf{A}^\rmb\) contains a right turn precisely if it is of type 5,6, and it contains a left turn otherwise.
We deduce that the probability of \(\ell_\shortrightarrow\) is proportional to
\begin{equation}\label{eq:prf:bkw4}
\prod_{t\in T_{5,6}(\ell_\shortrightarrow)\cap \mathsf{A}^\rmi} \big(e^{\lambda/2}\,\mathds{1}_{T_{5B,6A}(\ell_\shortrightarrow)}(t) + e^{-\lambda/2}\,\mathds{1}_{T_{5A,6B}(\ell_\shortrightarrow)}(t)\big)
\cdot \big(e^{\lambda/2}\big)^{|T_{5,6}(\ell_\shortrightarrow)\cap \partial \mathsf{A}|},
\end{equation}
where \(T_{5B,6A}(\ell_\shortrightarrow)\) and \(T_{5A,6B}(\ell_\shortrightarrow)\) are respectively the sets of tiles of types 5B and 6A and the set of tiles of types 5A and 6B in \(\ell_\shortrightarrow\), and \(T_{5,6}(\ell_\shortrightarrow)\) is their union.

Finally, fix a pair \((\sigma_\bullet,\sigma_\circ)\in\Sigma_\mathsf{B}^{+-}\times\Sigma_{\mathsf{B}'}^{+-}\) that satisfies the ice-rule, and consider its associated edge orientations.
It remains to identify all oriented loop configurations \(\ell_\shortrightarrow\) that satisfy the boundary conditions in Fig.~\ref{fig:bc_FK_loops_spins_o-o} and that induce these edge orientations. Observe that the boundary conditions and the spins \((\sigma_\bullet,\sigma_\circ)\) uniquely determine the oriented loop arcs at tiles in \((\bbL_\diamond\setminus \mathsf{A})\cup \mathsf{A}^\rmb\) and at tiles in \(\mathsf{A}^\rmi\setminus T_{5,6}^\rmi(\sigma_\bullet,\sigma_\circ)\). For a tile in \(T_{5,6}^\rmi(\sigma_\bullet,\sigma_\circ)\), one can split the oriented edges either into a pair of right-oriented loops arcs (types 5B,6A) or into a pair of left-oriented loop arcs (types 5A,6B). Summing the probabilities~\eqref{eq:prf:bkw4} of all oriented loop configurations obtained in that way, we obtain that the probability of \((\sigma_\bullet,\sigma_\circ)\) is proportional to
\begin{equation*}
\big(e^{\lambda/2}+e^{-\lambda/2}\big)^{|T_{5,6}^\rmi(\sigma_\bullet,\sigma_\circ)|}\cdot\big(e^{\lambda/2}\big)^{|T_{5,6}^\rmb(\sigma_\bullet,\sigma_\circ)|}.
\end{equation*}
Recalling that \(\svc=e^{\lambda/2}+e^{-\lambda/2}\) and \(\svcb=e^{\lambda/2}\) finishes the proof. 
\end{proof}

\medskip

\subsubsection{From six-vertex to FK}
\label{sec:coupling:sixv_to_fk}
When proving that the respective interfaces in the FK model are close to those in the six-vertex and AT models in Section~\ref{sec:proxi_int}, we will also need the `reverse direction' of the procedure described in Lemma~\ref{lem:fk11_to_sixv}. In other words, we will need to describe how to obtain the FK measure~\(\fk_\msf{G}^{1/1}(\,\cdot\given\partial^+\nleftrightarrow\partial^-)\) from the six-vertex spin measure~\(\spin_\msf{D}^{+-,+-}\).
Since percolation configurations are in bijection with unoriented loop configurations, it suffices to construct an oriented loop configuration and then project it onto the corresponding unoriented one.

\begin{lemma}
\label{lem:6V_spins_to_FK}
Let \((\omega,U,U')\) be distributed according to \(\Psi_\mathsf{G}^{1/1}\), and let \((\sigma_\bullet,\sigma_\circ)\) with law~\(\spin_\msf{D}^{+-,+-}\) be as constructed in Lemma~\ref{lem:fk11_to_sixv}. Then~\(U'\) is independent of \((\sigma_\bullet,\sigma_\circ)\).
Consider the edge orientations corresponding to \((\sigma_\bullet,\sigma_\circ)\) (see Fig.~\ref{fig:six-vertex_arrows_spins_heights}), and construct a compatible oriented loop configuration~\(\ell_\shortrightarrow\) from them as follows:
\begin{itemize}
\item each vertex in~\(\bbL_\circ\setminus(\mathsf{B}'\cup (\R\times\{-\tfrac{1}{2}\}))\) is surrounded by a clockwise oriented loop of minimal length,
\item every tile~\(t\in\bbL_\diamond\setminus\msf{A}^\rmi\) with~\(e_t^*\subset \R\times\{-\tfrac{1}{2}\}\) is of type 2 (see Figures~\ref{fig:oriented_loop_arcs} and~\ref{fig:bc_FK_loops_spins_o-o}), so that the resulting infinite paths are oriented from right to left,
\item for \(t\in\msf{A}^\rmi\) of type 1-4, there is a unique way to construct the loop arcs (see Fig.~\ref{fig:oriented_loop_arcs}); for \(t\in\msf{A}^\rmi\) of type 5-6, split the oriented edges into two right-oriented arcs (types 5B,6A) if \(U'_t<e^{\lambda/2}/\svc\) and into two left-oriented arcs (types 5A,6B) otherwise (see Fig.~\ref{fig:6V_to_oriented_loops}).
\end{itemize}
Then the law of the unique percolation configuration in~\(\{0,1\}^{\bbE^\bullet}\) associated to \(\ell_\shortrightarrow\) is given by~\(\fk_{\msf{G}}^{1/1}(\,\cdot\given\partial^+\nleftrightarrow\partial^-)\), which is the law of~\(\omega\).

\end{lemma}

\begin{figure}
	\centering
	\ifpic
	\begin{tikzpicture}[scale=1]
		\draw (0.5,0)--(0,0.5)--(-0.5,0)--(0,-0.5)--(0.5,0);
		\draw (2.5,1)--(2,1.5)--(1.5,1)--(2,0.5)--(2.5,1);
		\draw (2.5,-1)--(2,-0.5)--(1.5,-1)--(2,-1.5)--(2.5,-1);
		
		\draw[thick] (-0.25,-0.25)--(-0.15,-0.15);
		\draw[<-, thick] (-0.15,-0.15)--(0,0);
		\draw[thick] (0.25,0.25)--(0.15,0.15);
		\draw[<-, thick] (0.15,0.15)--(0,0);
		\draw[thick] (0.1,-0.1)--(0,0);
		\draw[->, thick] (0.25,-0.25)--(0.1,-0.1);
		\draw[thick] (-0.1,0.1)--(0,0);		
		\draw[->, thick] (-0.25,0.25)--(-0.1,0.1);
			
		\Hloop{2}{1}
		\Vloop{2}{-1}
		\DOWNarrow{2-0.5+1/(2*sqrt(2))}{-1}{blue}
		\UParrow{2+0.5-1/(2*sqrt(2))}{-1}{blue}
		\RIGHTarrow{2}{1+0.5-1/(2*sqrt(2))}{blue}
		\LEFTarrow{2}{1-0.5+1/(2*sqrt(2))}{blue}
		
		\draw[->] (0.6,0.1)--(1.4,0.9);
		\draw (0.3,0.5) node[above]{$U'_t\geq \frac{e^{\lambda/2}}{\svc}$};
		\draw[->] (0.6,-0.1)--(1.4,-0.9);
		\draw (0.3,-0.5) node[below]{$U'_t < \frac{e^{\lambda/2}}{\svc} $};
		
		\draw (5,0)--(4.5,0.5)--(4,0)--(4.5,-0.5)--(5,0);
		\draw (7,1)--(6.5,1.5)--(6,1)--(6.5,0.5)--(7,1);
		\draw (7,-1)--(6.5,-0.5)--(6,-1)--(6.5,-1.5)--(7,-1);
		
		\draw[thick] (4.4,-0.1)--(4.5,0);
		\draw[->, thick] (4.25,-0.25)--(4.4,-0.1);
		\draw[thick] (4.6,0.1)--(4.5,0);
		\draw[->, thick] (4.75,0.25)--(4.6,0.1);
		\draw[thick] (4.75,-0.25)--(4.65,-0.15);		
		\draw[<-, thick] (4.65,-0.15)--(4.5,0);
		\draw[thick] (4.25,0.25)--(4.35,0.15);
		\draw[<-, thick] (4.35,0.15)--(4.5,0);
	
		\Hloop{6.5}{1}
		\Vloop{6.5}{-1}

		\LEFTarrow{6.5}{1+0.5-1/(2*sqrt(2))}{blue}
		\RIGHTarrow{6.5}{1-0.5+1/(2*sqrt(2))}{blue}
		\UParrow{6.5-0.5+1/(2*sqrt(2))}{-1}{blue}
		\DOWNarrow{6.5+0.5-1/(2*sqrt(2))}{-1}{blue}

		\draw[->] (5.1,0.1)--(5.9,0.9);
		\draw (4.8,0.5) node[above]{$U'_t < \frac{e^{\lambda/2}}{\svc} $};
		\draw[->] (5.1,-0.1)--(5.9,-0.9);
		\draw (4.8,-0.5) node[below]{$U'_t\geq \frac{e^{\lambda/2}}{\svc}$};
	\end{tikzpicture}
	\fi
	\caption{Splitting rule for tiles in \(\msf{A}^\rmi\) of types 5-6.}
	\label{fig:6V_to_oriented_loops}
\end{figure}
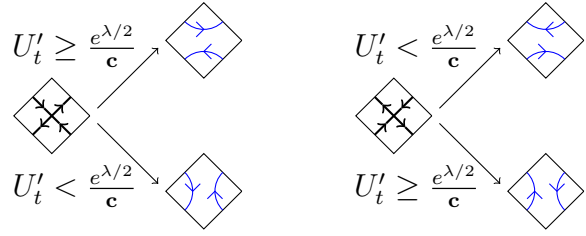

The proof is that of Lemma~\ref{lem:fk11_to_sixv} in reverse order.

\subsubsection{From six-vertex to modified ATRC}
\label{sec:coupling:sixv_to_matrc}
We adapt the coupling described in~\cite{DobGlaOtt25} to the double Dobrushin boundary conditions.
In particular, we need to define the associated modified ATRC measure.
Recall the parameters from~\eqref{eq:parameters_bulk} and the (augmented) rectangular domains from Section~\ref{sec:notations}. Again, we omit~\(n,m\) from the notation. 

\begin{definition}
	\label{def:mATRC}
	The modified ATRC model \(\matrc_{n,m} \equiv \matrc_\mathsf{K}\) is a probability measure on \(\{0,1\}^{\bar{\mathsf{E}}}\times\{0,1\}^{\bar{\mathsf{E}}}\) defined by
	\begin{equation}\label{eq:matrc_def}
	\begin{multlined}
\matrc_\mathsf{K}(\omega_\tau,\omega_{\tau\tau'})\\
\propto\mathds{1}_{\omega_\tau\subseteq\omega_{\tau\tau'}}\,\mathds{1}_{\omega_{\tau\tau'}\setminus\omega_{\tau}\subseteq \mathsf{E}}\,
2^{|\omega_\tau\cap\mathsf{E}|}\,\big(\tfrac{2}{\svcb-1}\big)^{|\omega_\tau\cap\mathsf{E}_\rmb|}\,(\svc-2)^{|\omega_{\tau\tau'}\setminus\omega_\tau|}\,2^{\clusters_{\mathsf{K}}(\omega_\tau)+\clusters_{\mathsf{K}^1}(\omega_{\tau\tau'})},
	\end{multlined}
	\end{equation}	
	where~\(\clusters_{\mathsf{K}}(\omega_\tau)\) and~\(\clusters_{\mathsf{K}^1}(\omega_{\tau\tau'})\) are the numbers of clusters of~\(\omega_\tau\) and~\(\omega_{\tau\tau'}\) viewed as spanning subgraphs of~\(\mathsf{K}\) and~\(\mathsf{K}^1\), respectively.
	We say that a cluster of~\(\omega_\tau\) or~\(\omega_{\tau\tau'}\) is an {\em inner cluster} if it is entirely contained in~$\mathsf{B}$ and a {\em boundary cluster} otherwise. 
\end{definition}

We now describe how to obtain~\(\matrc_\mathsf{K}\) as a push-forward of the coupling measure~\(\Psi^{1/1}_\mathsf{G}\). 

\begin{lemma}
\label{lem:6V_spins_to_AT}
	Let~\((\omega,U,U')\) be distributed according to \(\Psi_\mathsf{G}^{1/1}\).
	Define~\((\sigma_\bullet,\sigma_\circ)\) as described in Lemma~\ref{lem:fk11_to_sixv}.
	Now define \(\omega_\tau,\omega_{\tau\tau'}\in \{0,1\}^{\bar{\mathsf{E}}}\) as follows using the notation~\(e_t=ij\in \bar{\mathsf{E}}\) and \(e_t^*=uv\) for each tile \(t\in \mathsf{A}= \mathsf{A}^\rmi\sqcup\partial \mathsf{A}\):
	\begin{itemize}
		\item if \(\sigma_{\circ}(u)\neq \sigma_{\circ}(v)\), set \(\omega_\tau(e_t) = \omega_{\tau\tau'}(e_t) = 1\);
		\item if \(\sigma_{\bullet}(i) \neq \sigma_{\bullet}(j)\), set \(\omega_\tau(e_t) = \omega_{\tau\tau'}(e_t) = 0\);
		\item if both \(\sigma_{\circ}(u)= \sigma_{\circ}(v)\) and \(\sigma_{\bullet}(i) = \sigma_{\bullet}(j)\) hold (types 5-6), let
		\[(\omega_\tau(e_t),\omega_{\tau\tau'}(e_t))= \begin{cases}
			\mathds{1}_{[0,\frac{1}{\svc})}(U_t')\,{\cdot}\,(1,1)+\mathds{1}_{[\frac{1}{\svc},\frac{2}{\svc})}(U_t')\,{\cdot}\,(0,0)+\mathds{1}_{[\frac{2}{\svc},1]}(U_t')\,{\cdot}\,(0,1) & \text{if }t\in \mathsf{A}^\rmi,\\
			\mathds{1}_{[0,\frac{1}{\svcb})}(U_t')\,{\cdot}\,(1,1)+\mathds{1}_{[\frac{1}{\svcb},1]}(U_t')\,{\cdot}\,(0,0)& \text{if }t\in\partial \mathsf{A}.
		\end{cases}
		\]
	\end{itemize}
	Then, the law of~\((\omega_\tau,\omega_{\tau\tau'})\) is~$\matrc_{\mathsf{K}}(\, \cdot \given v_L\xleftrightarrow{\omega_\tau} v_R,\, v'_L\xleftrightarrow{\omega_{\tau\tau'}^*} v'_R)$.
	In particular, we have \(\omega_\tau\subseteq\omega_{\tau\tau'}\) and \(\omega_{\tau\tau'}\setminus\omega_\tau\subseteq \{e_t:t\in \mathsf{A}^\rmi\}=\mathsf{E}\); also~\(\sigma_\bullet\sim \omega_{\tau\tau'}\) and \(\sigma_\circ \sim \omega_\tau^*\).
\end{lemma}

\begin{remark}
The above sampling rule for \(t\in \mathsf{A}^\rmi\) applied to a six-vertex spin or height measure without modified boundary weight~\(\svcb\) yields an ATRC measure, as defined in Secton~\ref{sec:notations}.
\end{remark}

\begin{proof}
	We adapt the proof of the order-disorder case, building on~\cite[Proof of Lemma~7.1]{GlaPel23}.
	One has to examine which values of \((\sigma_{\bullet},\sigma_{\circ}, U)\) result in a given pair \(\omega_\tau,\omega_{\tau\tau'} \in \{0,1\}^{\bar{\mathsf{E}}}\). Recall that \(\sigma_{\bullet} \sim \omega_{\tau\tau'}\) and~\(\sigma_{\circ}\sim \omega_\tau^*\).
	By Lemma~\ref{lem:fk11_to_sixv} and~\eqref{eq:def_6v-spin_dobrushin}, the probability of a quadruplet \((\sigma_{\bullet},\sigma_{\circ},\omega_\tau,\omega_{\tau\tau'})\) is
	\begin{align}
		\label{eq:prf:6Vspins_to_AT:proba3}
		\nonumber
		&\spin_{\mathsf{D}}^{+-,+-}(\sigma_\bullet,\sigma_\circ)\,\mathds{1}_{\sigma_{\bullet}\sim \omega_{\tau\tau'}}\,\mathds{1}_{\sigma_{\circ}\sim \omega_\tau^*}\,\mathds{1}_{\omega_\tau\subseteq\omega_{\tau\tau'}}\,\mathds{1}_{\omega_{\tau\tau'}\setminus\omega_\tau\subseteq \mathsf{E}}
		\\
		\nonumber 
		&\qquad\cdot\prod_{t\in T_{5,6}^{\rmi}}\big(\tfrac{1}{\svc}\,(\mathds{1}_{t\in\omega_\tau}+\mathds{1}_{t\in\omega_{\tau\tau'}^*})+\tfrac{\svc-2}{\svc}\,\mathds{1}_{t\in\omega_{\tau\tau'}\setminus\omega_\tau}\big)
		\prod_{t\in T_{5,6}^{\rmb}}\big(\tfrac{1}{\svcb}\,\mathds{1}_{t\in\omega_\tau}+\tfrac{\svcb-1}{\svcb}\,\mathds{1}_{t\in\omega_{\tau}^*}\big)
		\\
		\nonumber &\propto\mathds{1}_{\Sigma_\mathsf{B}^{+-}}(\sigma_{\bullet})\mathds{1}_{\Sigma_{\mathsf{B}'}^{+-}}(\sigma_{\circ})\,\mathds{1}_{\sigma_{\bullet}\sim \omega_{\tau\tau'}}\,\mathds{1}_{\sigma_{\circ}\sim \omega_\tau^*}\,\mathds{1}_{\omega_\tau\subseteq\omega_{\tau\tau'}}\,\mathds{1}_{\omega_{\tau\tau'}\setminus\omega_\tau\subseteq \mathsf{E}}
		\\
		&\qquad\cdot\prod_{t\in T_{5,6}^{\rmi}}\big(\mathds{1}_{t\in\omega_\tau}+\mathds{1}_{t\in\omega_{\tau\tau'}^*}+(\svc-2)\mathds{1}_{t\in\omega_{\tau\tau'}\setminus\omega_\tau}\big)
		\prod_{t\in T_{5,6}^{\rmb}}\big(\mathds{1}_{t\in\omega_\tau}+(\svcb-1)\mathds{1}_{t\in\omega_{\tau}^*}\big),
	\end{align}
	where we used the shorthand \(T_{5,6}^{\#}\equiv T_{5,6}^{\#}(\sigma_{\bullet},\sigma_{\circ})\) and the fact that, for~\(\sigma_{\bullet}\in\Sigma_\mathsf{B}^{+-}\) and~\(\sigma_{\circ}\in\Sigma_{\mathsf{B}'}^{+-}\),
	\begin{equation*}
		\mathds{1}_{\mathrm{ice}}(\sigma_{\bullet},\sigma_{\circ})\,\mathds{1}_{\sigma_{\bullet}\sim \omega_{\tau\tau'}}\,\mathds{1}_{\sigma_{\circ}\sim \omega_\tau^*}\,\mathds{1}_{\omega_\tau\subseteq\omega_{\tau\tau'}} = \mathds{1}_{\sigma_{\bullet}\sim \omega_{\tau\tau'}}\,\mathds{1}_{\sigma_{\circ}\sim \omega_\tau^*}\,\mathds{1}_{\omega_\tau\subseteq\omega_{\tau\tau'}}.
	\end{equation*}
	Observe that, on the event \(\{\sigma_\bullet\sim\omega_{\tau\tau'},\,\sigma_\circ\sim\omega_\tau^*,\,\omega_\tau\subseteq\omega_{\tau\tau'}\}\), for any tile~\(t\in \mathsf{A}^\rmi\),
	\begin{equation*}
		\mathds{1}_{T_{5,6}^{\rmi}}(t)=\mathds{1}_{\omega_\tau}(t)\mathds{1}_{\sigma_{\circ}\sim e_t^*}+\mathds{1}_{\omega_{\tau\tau'}^*}(t)\mathds{1}_{\sigma_{\bullet}\sim e_t}+\mathds{1}_{\omega_{\tau\tau'}\setminus\omega_\tau}(t).
	\end{equation*}
	Let~\(t_1,t_2\in \partial \mathsf{A}\) be the two boundary tiles that intersect both~\(\bbH^+\) and~\(\bbH^-\). Then, on the same event as above intersected with \(\{\sigma_\bullet\in\Sigma_\mathsf{B}^{+-},\,\sigma_\circ\in\Sigma_{\mathsf{B}'}^{+-},\,\omega_{\tau\tau'}\setminus\omega_\tau\subseteq \mathsf{E}\}\), for any tile~\(t\in\partial \mathsf{A}\setminus\{t_1,t_2\}\),
	\begin{equation*}
		\mathds{1}_{T_{5,6}^{\rmb}}(t)=\mathds{1}_{\omega_\tau}(t)\mathds{1}_{\sigma_{\circ}\sim e_t^*}+\mathds{1}_{\omega_{\tau}^*}(t).
	\end{equation*}
	Notice also that, due to the boundary condition~\(\sigma_\bullet\in\Sigma_\mathsf{B}^{+-}\), the two boundary tiles~\(t_1,t_2\) cannot be of type 5,6. Moreover, since~\(\sigma_\bullet\sim\omega_{\tau\tau'}\), the associated edges~\(e_{t_1},e_{t_2}\in\mathsf{E}_\rmb\) must be closed in~\(\omega_{\tau\tau'}\) (and~\(\omega_\tau\), as they coincide on~\(\mathsf{E}_\rmb\)).
	Altogether,~\eqref{eq:prf:6Vspins_to_AT:proba3} becomes 
	\begin{align*}
		&\mathds{1}_{\Sigma_\mathsf{B}^{+-}}(\sigma_{\bullet})\mathds{1}_{\Sigma_{\mathsf{B}'}^{+-}}(\sigma_{\circ})\mathds{1}_{\sigma_{\bullet}\sim \omega_{\tau\tau'}}\mathds{1}_{\sigma_{\circ}\sim \omega_\tau^*}\mathds{1}_{\omega_\tau\subseteq\omega_{\tau\tau'}}\mathds{1}_{\omega_{\tau\tau'}\setminus\omega_\tau\subseteq \mathsf{E}}\\
		&\cdot\prod_{t\in \mathsf{A}^{\rmi}}\big(\mathds{1}_{t\in\omega_\tau}+\mathds{1}_{t\in\omega_{\tau\tau'}^*}+(\svc-2)\mathds{1}_{t\in\omega_{\tau\tau'}\setminus\omega_\tau}\big)^{\mathds{1}_{T_{5,6}^{\rmi}}(t)}
		\prod_{t\in \partial \mathsf{A}}\big(\mathds{1}_{t\in\omega_\tau}+(\svcb-1)\mathds{1}_{t\in\omega_{\tau}^*}\big)^{\mathds{1}_{T_{5,6}^{\rmb}}(t)}
		\\
		=\ 
		&\mathds{1}_{\Sigma_\mathsf{B}^{+-}}(\sigma_{\bullet})\mathds{1}_{\Sigma_{\mathsf{B}'}^{+-}}(\sigma_{\circ})\mathds{1}_{\sigma_{\bullet}\sim \omega_{\tau\tau'}}\mathds{1}_{\sigma_{\circ}\sim \omega_\tau^*}\mathds{1}_{\omega_\tau\subseteq\omega_{\tau\tau'}}\mathds{1}_{\omega_{\tau\tau'}\setminus\omega_\tau\subseteq \mathsf{E}}\,(\svc-2)^{|\omega_{\tau\tau'}\setminus\omega_\tau|}\,(\svcb-1)^{|\mathsf{E}_\rmb\setminus\omega_{\tau}|-2}.
	\end{align*}
	Observe that \(\sigma_\circ\in\Sigma_{\mathsf{B}'}^{+-}\) and \(\sigma_\circ\sim\omega_{\tau}^*\) imply that \(v_L\) and \(v_R\) are connected in \(\omega_\tau\) (see Fig.~\ref{fig:bc_FK_loops_spins_o-o}).
	To obtain the probability of a pair \((\omega_{\tau},\omega_{\tau\tau'})\) satisfying \(v_L\xleftrightarrow{\omega_{\tau}} v_R\), we need to sum the last expression over \((\sigma_\bullet,\sigma_\circ)\).
	The configurations \(\sigma_\circ\in\Sigma_{\mathsf{B}'}^{+-}\) with \(\sigma_\circ\sim\omega_\tau^*\) are in bijective correspondence with assignments of $\pm1$ to the clusters of \(\omega_\tau^*\) in \((\mathsf{K}')^1\) that are contained in \(\mathsf{B}'\), whence there exist~\(2^{\clusters_{(\mathsf{K}')^1}(\omega_\tau^*)-2}\) of them. By Euler's formula,
	\begin{equation*}
		\clusters_{(\mathsf{K}')^1}(\omega_\tau^*) = \clusters_{\mathsf{K}}(\omega_\tau) + |\omega_{\tau}| - \mathrm{const}(\mathsf{K'}).
	\end{equation*}
	Similarly, if~\(v'_L\) and~\(v'_R\) are connected in~\(\omega_{\tau\tau'}^*\), then there exist~\(2^{\clusters_{\mathsf{K}^1}(\omega_{\tau\tau'})-2}\) spin configurations~\(\sigma_\bullet\in\Sigma_\mathsf{B}^{+-}\) with~\(\sigma_\bullet\sim\omega_{\tau\tau'}\).
	Therefore, the probability of a tuple \((\omega_{\tau},\omega_{\tau\tau'})\) with \(v_L\leftrightarrow v_R\) in~\(\omega_\tau\) and~\(v'_L\leftrightarrow v'_R\) in~\(\omega_{\tau\tau'}^*\) is proportional to
	\begin{equation*}
		\mathds{1}_{\omega_\tau\subseteq\omega_{\tau\tau'}}\mathds{1}_{\omega_{\tau\tau'}\setminus\omega_\tau\subseteq \mathsf{E}}\,
		(\svc-2)^{|\omega_{\tau\tau'}\setminus\omega_\tau|}(\svcb-1)^{|\mathsf{E}_\rmb\setminus\omega_{\tau}|}\,2^{\clusters_{\mathsf{K}}(\omega_\tau)+|\omega_{\tau}|}\,2^{\clusters_{\mathsf{K^1}}(\omega_{\tau\tau'})},
	\end{equation*}
	which is proportional to~$\matrc_\mathsf{K}(\omega_{\tau},\omega_{\tau\tau'})$ as defined in~\eqref{eq:matrc_def}, and the proof is complete.
\end{proof}

\subsection{Positive association and repulsiveness}
\label{sec:fkg-repulsion}
We continue in the setting of the previous section. 
The measure $\matrc_\mathsf{K}$ is strongly positively associated. 
\begin{lemma}\label{lem:fkg_lattice_mod_atrc}
$\matrc_\mathsf{K}$ satisfies the FKG-lattice condition: for any $a,b,a',b'\in\{0,1\}^{\bar{\mathsf{E}}}$, 
\begin{equation}\label{eq:fkg_lattice_matrc}
\matrc_\mathsf{K}(a\cup a',b\cup b')\,\matrc_\mathsf{K}(a\cap a',b \cap b')\geq \matrc_\mathsf{K}(a,b)\,\matrc_\mathsf{K}(a',b').
\end{equation}
\end{lemma}

\begin{proof}
First observe that the indicators $\mathds{1}_{\omega_\tau\subseteq\omega_{\tau\tau'}},\,\mathds{1}_{\omega_{\tau\tau'}\setminus\omega_\tau\subseteq \mathsf{E}}$ satisfy~\eqref{eq:fkg_lattice_matrc}. Indeed, 
\begin{itemize}
\item $a\subseteq b$ and $a'\subseteq b'$ imply $a\cup a'\subseteq b\cup b'$ and $a\cap a'\subseteq b\cap b'$,
\item $b\setminus a\subseteq \mathsf{E}$ and $b'\setminus a'\subseteq \mathsf{E}$ imply $(b\cup b')\setminus (a\cup a')\subseteq \mathsf{E}$ and $(b\cap b')\setminus (a\cap a')\subseteq \mathsf{E}$.
\end{itemize}
From now on, assume that $a\subseteq b$ and $a'\subseteq b'$ (otherwise, the right side of~\eqref{eq:fkg_lattice_matrc} is zero). We now check~\eqref{eq:fkg_lattice_matrc} for each factor of~\eqref{eq:matrc_def}.
For the factor $2^{|\omega_\tau\cap\mathsf{E}|}$ in~\eqref{eq:matrc_def}, the exponent on the left side of~\eqref{eq:fkg_lattice_matrc} is
\begin{align*}
|(a\cup a')\cap\mathsf{E}|+|(a\cap a')\cap\mathsf{E}|&=|(a\cap\mathsf{E})\cup(a'\cap\mathsf{E})|+|(a\cap\mathsf{E})\cap(a'\cap\mathsf{E})|\\
&=|a\cap\mathsf{E}|+|a'\cap\mathsf{E}|,
\end{align*}
which is the exponent on the right side. Replacing $\mathsf{E}$ by $\mathsf{E}_\rmb$ in the above computation gives the same for the factor $(2/(\svcb-1))^{|\omega_\tau\cap\mathsf{E}_\rmb|}$ in~\eqref{eq:matrc_def}. For the factor $(\svc-2)^{|\omega_{\tau\tau'}\setminus\omega_\tau|}$ in~\eqref{eq:matrc_def}, the exponent on the left side of~\eqref{eq:fkg_lattice_matrc} is the sum of the following two:
\begin{align*}
|(b\cup b')\setminus (a\cup a')|&=|b\setminus (a\cup b')|+|b'\setminus (b\cup a')|+|(b\cap b')\setminus (a\cup a')|,\\
|(b\cap b')\setminus (a\cap a')|&=|(b\cap b')\setminus (a\cup a')|+|(a\cap b')\setminus a'|+|(b\cap a')\setminus a|.
\end{align*}
On the other hand, the exponent on the right side is the sum of the following two
\begin{align*}
|b\setminus a|&=|b\setminus (a\cup b')|+|(b\cap b')\setminus (a\cup a')|+|(b\cap a')\setminus a|,\\
|b'\setminus a'|&=|b'\setminus (b\cup a')|+|(b\cap b')\setminus (a\cup a')|+|(a\cap b')\setminus a'|.
\end{align*}
Clearly, these sums coincide. It remains to check the inequality for the last factor $2^{\kappa_{\mathsf{K}}(\omega_\tau)+\kappa_{\mathsf{K}^1}(\omega_{\tau\tau'})}$ in~\eqref{eq:matrc_def}. It is classical (see, e.g.,~\cite[Proof of Theorem 3.8]{Gri06}) that
\begin{equation*}
\kappa_{\mathsf{K}}(a\cup a')+\kappa_{\mathsf{K}}(a\cap a')\geq\kappa_{\mathsf{K}}(a)+\kappa_{\mathsf{K}}(a'),
\end{equation*}
and analogously for $\kappa_{\mathsf{K}^1}(\omega_{\tau\tau'})$. This completes the proof.
\end{proof}
This lemma implies that connections in $\omega_\tau$ induce `maximal boundary conditions'. Let us first define the notion of sets being above and below each other. 

\begin{definition}\label{def:above}
Given a bi-infinite connected set \(C\subset\bbL_\bullet\cap  (\R\times [-c,c])\), we say that a subset \(R\subset\R^2\) is (weakly) \emph{above} \(C\) if it is contained in (the closure of) the connected component of the point~\((0,c+1)\) in \(\R^2\setminus C\), where we identify \(C\) with the union of line segments between the endpoints of the edges in \(\bbE_C\). We say that \(R\) is (weakly) \emph{below} \(C\) if it is contained in (the closure of) the connected component of the point~\((0,-c-1)\) in \(\R^2\setminus C\). We make the analogous definitions for finite connected sets by extending them to bi-finite connected sets by attaching left-infinite horizontal lines to all the leftmost points of the finite set and right-infinite horizontal lines to all the rightmost points; and, analogously, for connected subsets \(C\subset\bbL_\circ\).
\end{definition}

\begin{corollary}
	\label{cor:repulsiveness}
	Let $\gamma$ be a deterministic set of edges forming a simple path that connects $v_L$ to $v_R$.
	Let ${\gamma_\shortuparrow}$ and ${\gamma_\shortdownarrow}$ be the sets of edges in $\bar{\mathsf{E}}\setminus \gamma$ that are above and below $\gamma$, respectively. For any $a_0,b_0\in\{0,1\}^{\gamma},\;a_{\shortuparrow},b_{\shortuparrow}\in\{0,1\}^{{\gamma_\shortuparrow}}$ with $a_0\cup a_{\shortuparrow}\subseteq b_0\cup b_{\shortuparrow}$,
	\begin{align*}
		&\matrc_{\mathsf{K}}\big( (\omega_\tau|_{\gamma_\shortdownarrow},\omega_{\tau\tau'}|_{\gamma_\shortdownarrow}) \in\cdot \bgiven (\omega_\tau|_{\gamma},\omega_{\tau\tau'}|_{\gamma})=(a_0,b_0),\, (\omega_\tau|_{\gamma_\shortuparrow},\omega_{\tau\tau'}|_{\gamma_\shortuparrow}) = (a_{\shortuparrow},b_{\shortuparrow})\big)
		\\
		&\preccurlyeq
		\matrc_{\mathsf{K}}\big( (\omega_\tau|_{\gamma_\shortdownarrow},\omega_{\tau\tau'}|_{\gamma_\shortdownarrow})\in\cdot \bgiven \omega_\tau|_{\gamma} \equiv 1,\,(\omega_\tau|_{\gamma_\shortuparrow},\omega_{\tau\tau'}|_{\gamma_\shortuparrow})=(a_{\shortuparrow},b_{\shortuparrow})\big)
		\\
		&=
		\matrc_{\mathsf{K}}\big( (\omega_\tau|_{\gamma_\shortdownarrow},\omega_{\tau\tau'}|_{\gamma_\shortdownarrow})\in\cdot \bgiven \omega_\tau|_{\gamma}\equiv 1\big).
	\end{align*}
	In the same fashion, for \(\gamma^*\) a path of dual edges connecting \(v_L'\) to \(v_R'\),
	\begin{align*}
		&\matrc_{\mathsf{K}}\big( (\omega_\tau|_{\gamma^*_\shortuparrow},\omega_{\tau\tau'}|_{\gamma^*_\shortuparrow}) \in\cdot \bgiven (\omega_\tau|_{*\gamma^*},\omega_{\tau\tau'}|_{*\gamma^*})=(a_0,b_0),\, (\omega_\tau|_{\gamma^*_\shortdownarrow},\omega_{\tau\tau'}|_{\gamma^*_\shortdownarrow}) = (a_{\shortdownarrow},b_{\shortdownarrow})\big)
		\\
		&\succcurlyeq
		\matrc_{\mathsf{K}}\big( (\omega_\tau|_{\gamma^*_\shortuparrow},\omega_{\tau\tau'}|_{\gamma^*_\shortuparrow})\in\cdot \bgiven \omega_{\tau\tau'}|_{*\gamma^*} \equiv 0,\,(\omega_\tau|_{\gamma^*_\shortdownarrow},\omega_{\tau\tau'}|_{\gamma^*_\shortdownarrow})=(a_{\shortdownarrow},b_{\shortdownarrow})\big)
		\\
		&=
		\matrc_{\mathsf{K}}\big( (\omega_\tau|_{\gamma^*_\shortuparrow},\omega_{\tau\tau'}|_{\gamma^*_\shortuparrow})\in\cdot \bgiven \omega_{\tau\tau'}|_{*\gamma^*}\equiv 0\big).
	\end{align*}
\end{corollary}

\begin{remark}
In regard to Lemma~\ref{lem:6V_spins_to_AT}, the corollary implies that the connection imposed by $v_L\xleftrightarrow{\omega_\tau}v_R$ has a repulsive effect on the dual connection imposed by $v'_L\xleftrightarrow{\omega_{\tau\tau'}^*}v'_R$.
\end{remark}

\begin{proof}
We only prove the first display, the second is identical. Since $\matrc_{\mathsf{K}}(\omega_\tau\subseteq\omega_{\tau\tau'})=1$, the stochastic domination statement is a consequence of Lemma~\ref{lem:fkg_lattice_mod_atrc} and~\cite[Theorem 2.27]{Gri06}.
The equality follows if, for $a_\shortuparrow,b_\shortuparrow\in\{0,1\}^{p_\shortuparrow}$ and $a_\shortdownarrow,b_\shortdownarrow\in\{0,1\}^{p_\shortdownarrow}$, the probability
\begin{equation*}
\matrc_{\mathsf{K}}\big((\omega_\tau|_{p_\shortdownarrow},\omega_{\tau\tau'}|_{p_\shortdownarrow})=(a_\shortdownarrow,b_\shortdownarrow),\,\omega_\tau|_{p}\equiv 1,\,(\omega_\tau|_{p_\shortuparrow},\omega_{\tau\tau'}|_{p_\shortuparrow})=(a_{\shortuparrow},b_{\shortuparrow})\big)
\end{equation*}
factorises into two parts that depend only on $(a_\shortdownarrow,b_\shortdownarrow)$ and $(a_\shortuparrow,b_\shortuparrow)$, respectively.

Recall that $\matrc_{\mathsf{K}}(\omega_\tau\subseteq\omega_{\tau\tau'})=1$, and set
\begin{equation*}
(a,b)=(a_\shortdownarrow\sqcup p\sqcup a_\shortuparrow,b_\shortdownarrow\sqcup p\sqcup b_\shortuparrow).
\end{equation*}
Then, by the definition~\eqref{eq:matrc_def} of the measure, the above probability equals 
\begin{multline*}
\matrc_{\mathsf{K}}\big((\omega_\tau,\omega_{\tau\tau'})=(a,b)\big)\propto
2^{|a\cap\mathsf{E}|}\big(\tfrac{2}{\svcb-1}\big)^{|a\cap\mathsf{E}_\rmb|}(\svc-2)^{|b\setminus a|}2^{\kappa_{\mathsf{K}}(a)+\kappa_{\mathsf{K}^1}(b)}\\
=2^{|a_\shortdownarrow\cap\mathsf{E}|+|p\cap\mathsf{E}|+|a_\shortuparrow\cap\mathsf{E}|}\big(\tfrac{2}{\svcb-1}\big)^{|a_\shortdownarrow\cap\mathsf{E}_\rmb|+|p\cap\mathsf{E}_\rmb|+|a_\shortuparrow\cap\mathsf{E}_\rmb|}(\svc-2)^{|b_\shortdownarrow\setminus a_\shortdownarrow|+|b_\shortuparrow\setminus a_\shortuparrow|}2^{\kappa_{\mathsf{K}}(a)+\kappa_{\mathsf{K}^1}(b)}.
\end{multline*}
For $\#\in\{\varnothing,1\}$, let $\mathsf{K}^\#_\shortuparrow$ (respectively, $\mathsf{K}^\#_\shortdownarrow$) be the graph obtained from $\mathsf{K}^\#$ by identifying all vertices on and below $p$ (respectively, on and above $p$). Since $p$ divides $\mathsf{K}^\#$ into two disjoint parts and since $p\subseteq a,b$, we clearly have 
\begin{equation*}
\kappa_{\mathsf{K}}(a)=\kappa_{\mathsf{K}_\shortdownarrow}(a_\shortdownarrow)+\kappa_{\mathsf{K}_\shortuparrow}(a_\shortuparrow)-1
\qquad\text{and}\qquad
\kappa_{\mathsf{K}^1}(b)=\kappa_{\mathsf{K}^1_\shortdownarrow}(b_\shortdownarrow)+\kappa_{\mathsf{K}^1_\shortuparrow}(b_\shortuparrow)-1.
\end{equation*}
Therefore, the probability factorises, and the proof is complete.
\end{proof}

\section{Proximity of interfaces}
\label{sec:proxi_int}
Recall the setting of Section~\ref{sec:coupling_under_dobr}.
In this section, we prove that in the coupling~\(\Psi_\msf{G}^{1/1}\), the respective `interfaces' in the FK random variable~\(\omega\) and the modified ATRC random variables~\((\omega_\tau,\omega_{\tau\tau'})\) are close to each other. Let us first define the relevant objects and distances.
Recall the interfaces~$\Gamma_\fk^1$ and~$\Gamma_\fk^2$ above Theorem~\ref{thm:invariance_princ_FK}.
Define
\begin{equation}\label{eq:interfaces_def}
\begin{aligned}
\mcal{C}_{v_L}&:=\text{ cluster of }v_L\text{ in }\omega_\tau,\\
\mcal{C}'_{v'_L}&:=\text{ cluster of }v'_L\text{ in }\omega_{\tau\tau'}^*.
\end{aligned}
\end{equation}
Recall the notion of sets being above each other, introduced in Definition~\ref{def:above}. Observe that~\(\Gamma_\fk^1\) is deterministically above~\(\Gamma_\fk^2\), and~\(\mcal{C}_{v_L}\) is deterministically above~\(\mcal{C}'_{v'_L}\), whence we will call them the upper and lower FK and AT interfaces, respectively. 
The goal is to prove that the upper FK interface~\(\Gamma_\fk^1\) is close to the upper AT interface~\(\mcal{C}_{v_L}\), and that the lower FK interface~\(\Gamma_\fk^2\) is close to the lower AT interface~\(\mcal{C}'_{v'_L}\).
We will work with the \emph{one-sided Hausdorff distance}, defined by
\begin{equation}\label{eq:hausdorff}
\rmd_{\mathrm{H}}(R,S):=\sup_{x\in R}\inf_{y\in S}\rmd_\infty(x,y),\qquad R,S\subset\R^2.
\end{equation}
Although the statements are morally similar, the proofs that the distances~\(\rmd_{\mathrm{H}}(\Gamma_\fk^1,\mcal{C}_{v_L})\) and~\(\rmd_{\mathrm{H}}(\Gamma_\fk^2,\mcal{C}'_{v'_L})\) are small turn out to be quite different.
They are given in Sections~\ref{sec:proxi_upper_int} and~\ref{sec:proxi_lower_int}, respectively, and are outlined below.
\medskip

\noindent\textbf{Upper interfaces.} The fact that~\(\rmd_{\mathrm{H}}(\Gamma_\fk^1,\mcal{C}_{v_L})\) is small is proved in a way exactly analogous to the proof of the proximity statement~\cite[Lemma 9.1]{DobGlaOtt25}. Let us informally sketch the argument here. In the coupling~\(\Psi_\msf{G}^{1/1}\), the Peierls contour between~\(\pm1\) in the six-vertex spin~\(\sigma_\circ\) is contained in~\(\mcal{C}_{v_L}\). Since~\(\sigma_\circ=-1\) is below~\(\Gamma_\fk^1\), this contour must be above~\(\Gamma_\fk^1\). Conditional on a realisation of~\(\Gamma_\fk^1\), the distribution of the six-vertex spins above have~`\(+,+\)' boundary conditions, and paths of~\(\sigma_\circ=-1\) induce circuits in~\(\omega_\tau\), which has uniform exponential decay by previous results in~\cite{DobGlaOtt25}.
\medskip

\noindent\textbf{Lower interfaces.} On the other hand, while the Peierls contour between~\(\pm1\) in the six-vertex spin~\(\sigma_\bullet\) is contained in~\(\mcal{C}'_{v'_L}\), it is not true that it is below~\(\Gamma_\fk^2\). In fact, the contour lies between~\(\Gamma_\fk^1\) and~\(\Gamma_\fk^2\). For this reason, proving that~\(\rmd_{\mathrm{H}}(\Gamma_\fk^2,\mcal{C}'_{v'_L})\) is small requires a different strategy, which we will also outline informally. Conditional on a realisation of the contour, the distribution of~\(\omega\) below is a modified FK measure, which by our results from~\cite{DobGlaOtt25} relaxes exponentially to the wired infinite-volume FK measure~\(\fk^\fkwired\). This rules out connections of the dual~\(\omega^*\) in the bulk below the contour. To rule out dual connections near the boundary for these modified FK measures, we prove a \emph{box-crossing} property at the boundary.

\subsection{Upper interfaces}\label{sec:proxi_upper_int}
In this section, we prove that the upper FK interface~\(\Gamma_\fk^1\) is close to the upper AT interface~\(\mcal{C}_{v_L}\) in the one-sided Hausdorff distance~\eqref{eq:hausdorff}.

\begin{lemma}\label{lem:proxi_upper_int}
There exist constants~\(C,c>0\) such that, for any~\(n,m,k\geq 1\),
\begin{equation*}
	\Psi^{1/1}\big(\rmd_{\mathrm{H}}(\Gamma_\fk^1,\mcal{C}_{v_L})>k\big)\leq Cnme^{-ck}.
\end{equation*}
\end{lemma}

The proof~\cite[Lemma 9.1]{DobGlaOtt25} extends readily to this case and we omit the details.

\subsection{Lower interfaces}\label{sec:proxi_lower_int}
In this section, we prove that the lower FK interface~\(\Gamma_\fk^2\) is close to the lower AT interface~\(\mcal{C}'_{v'_L}\) in the one-sided Hausdorff distance~\eqref{eq:hausdorff}.

\begin{lemma}\label{lem:proxi_lower_int}
There exist constants~\(C,c>0\) such that, for any~\(n,m,k\geq 1\),
\begin{equation*}
\Psi^{1/1}\big(\rmd_{\mathrm{H}}(\Gamma_\fk^2,\mcal{C}'_{v'_L})>k\big)\leq Cnm\max\{n,m\}^2e^{-c k}.
\end{equation*}
\end{lemma}

Before we can proceed with the proof as sketched above, we must first to introduce the quasi-FK measures, some notions of graphs and notation, and to prove and import some properties of the former that we will require.

\subsubsection{Quasi-FK measures.}
\label{sec:quasi_fk_def}
A modification of the classical BKW coupling (see Section~\ref{sec:combinatorial_mappings}) gives rise to a natural version of FK measures with a modified \emph{boundary cluster-weight}, introduced in~\cite{GlaPel23}, which can be viewed as a continuous interpolation between the free and wired FK measures. In the proof of Lemma~\ref{lem:proxi_lower_int}, we will encounter such measures with different cluster-weights on different parts of the boundary, whence we introduce them in this generality. 

Let $G=(V,E)$ be a finite graph, and let $p\in (0,1),\,q>0$.  
For the above reason, we consider these quasi-FK measures on $G$ together with two disjoint boundary parts $\mfr{b}_1,\mfr{b}_2\subseteq V$. Given in addition $q_1,q_2\in(0,\infty)$, the corresponding quasi-FK probability measure on $\{0,1\}^E$ is defined by
\[
\qfk_{G;p,q}^{\mfr{b}_1,\mfr{b}_2;q_{1},q_{2}}(\omega)\propto p^{\abs{\omega}}\,(1-p)^{\abs{E\setminus\omega}}\,q^{\clusters_\mrm{i}(\omega)}\,q_{1}^{\clusters_{\mfr{b}_1}(\omega)}\,q_{2}^{\clusters_{\mfr{b}_2}(\omega)},
\]
where $\clusters_{\mfr{b}_1}(\omega)$ is the number of clusters of $\omega$ that intersect $\mfr{b}_1$, $\clusters_{\mfr{b}_2}(\omega)$ is the number of clusters of $\omega$ that intersect $\mfr{b}_2$ but not $\mfr{b}_1$, and $\clusters_\mrm{i}(\omega)$ is the number of clusters that do not intersect $\mfr{b}_1 \cup \mfr{b}_2$.

If $\mfr{b}_2$ is empty, then we omit $\mfr{b}_2$ and $q_2$ from the superscript and simply write $\qfk_{G;p,q}^{\mfr{b}_1;q_{1}}$. 
Notice that in this case, if additionally~\(\mfr{b}_1=\partialin V\), the choices $q_1=q$ and $q_1=1$ lead to free and wired boundary conditions, respectively. 
For $q>4$ and $p=p_c(q)$, there are two more values of importance that originate from the BKW coupling. These are the \emph{quasi-free} and \emph{quasi-wired} weights given by
\[
\qbfree:=e^\lambda\sqrt{q}\qquad\text{and}\qquad\qbwired:=e^{-\lambda}\sqrt{q},
\]
respectively, where $\lambda>0$ is defined by~\eqref{eq:parameters_bulk}.
On the square lattice, it was shown in~\cite[Theorem 2.13]{GlaPel23} that $\qfk_{G;p,q}^{\partialin V;\qbfree}$ and $\qfk_{G;p,q}^{\partialin V;\qbwired}$ converge to the free and wired infinite-volume measures, respectively, as $G$ exhausts $\Z^2$.
We will require quantitative versions of these statements.

\subsubsection{Rotated square lattice and domains.}
\label{sec:rotated_lattice}
Let us introduce the rotated square lattice and a class of subgraphs of~\(\bbL_\bullet\), called weak~\(\bbL_\bullet\)-domains, on which we can control the quasi-free and quasi-wired FK measures by means of the BKW coupling.
\medskip

\noindent\textbf{Rotated square lattice.}
Consider the rotated square lattice $\bbL$ with vertex-set $\bbL_\bullet\cup\bbL_\circ$ and edges between nearest neighbours, that is, between vertices of Euclidean distance $1/\sqrt{2}$ (see the left side of Fig.~\ref{fig:rotated_lattice}).
Given $\Delta\subseteq\bbL$, we write $\Delta_\bullet=\Delta\cap\bbL_\bullet$ and $\Delta_\circ=\Delta\cap\bbL_\circ$.
The augmented graph $\overline{\bbL}$ has the same vertex-set as $\bbL$ and all edges of $\bbL,\,\bbL_\bullet$ and $\bbL_\circ$ (see the center of Fig.~\ref{fig:rotated_lattice}).
We restrict the notion of simple circuits in $\overline{\bbL}$ to those that do not traverse both $e$ and $e^*$ for any $e\in\bbE^\bullet$, so that they can be embedded in $\R^2$.
We identify such circuits with their planar embedding. 

\begin{figure}
\includegraphics[scale=0.45,page=1]{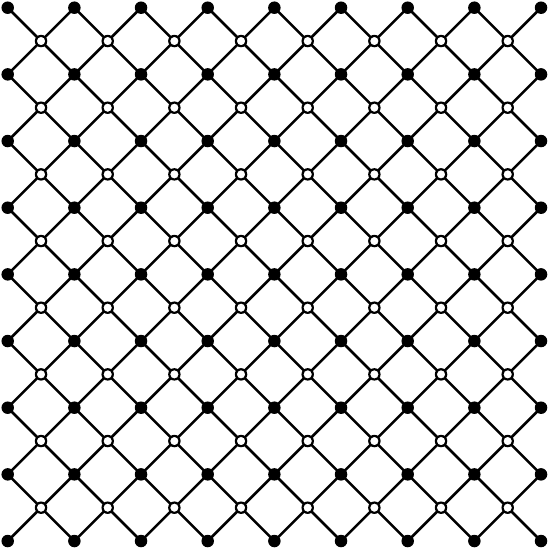}
\quad
\includegraphics[scale=0.45,page=2]{rotated_lattice_domains}
\quad
\includegraphics[scale=0.45,page=3]{rotated_lattice_domains}
\caption{Left: part of the rotated square lattice $\bbL$. Its faces are the tiles in $\bbL_\diamond$. Center: part of the augmented lattice $\overline{\bbL}$. Right: $\bbL$-domains given by the vertices strictly within simple circuits in $\overline{\bbL},\,\bbL_\bullet,\,\bbL_\circ$, respectively. The lower left and right $\bbL$-domains are even and odd, respectively.}
\label{fig:rotated_lattice}
\end{figure}

\medskip

\noindent\textbf{$\bbL$-domains.}
A finite subset $\mathcal{D}\subset\bbL$ is called an $\bbL$-domain if there exists a simple circuit $C$ in $\overline{\bbL}$ such that $\mathcal{D}$ is given by the vertices of $\bbL$ strictly within $C$, that is, in the bounded connected component of $\R^2\setminus C$ (see the right side of Fig.~\ref{fig:rotated_lattice}).
It is called \emph{even} (respectively, \emph{odd}) if $\partialex_{\bbL}\mathcal{D}\subset\bbL_\bullet$ (respectively, $\partialex_{\bbL}\mathcal{D}\subset\bbL_\circ$).
\medskip

\noindent\textbf{Weak $\bbL_\bullet$-domains.}
Fix an $\bbL$-domain $\mcal{D}$, and let $E$ be the set of edges $e\in\bbE^\bullet$ such that $e\cup e^*$ intersects $\mcal{D}$. Define $G_{\mcal{D}}$ to be the subgraph of $\bbL_\bullet$ with edge-set $E$ and vertex-set given by the endpoints of these edges. 
We call $G_{\mcal{D}}$ the \emph{weak $\bbL_\bullet$-domain} corresponding to $\mcal{D}$. Its $\mcal{D}$-boundary is defined by
\[
\partial_{\mcal{D}} G_{\mcal{D}}:=G_{\mcal{D}}\setminus \mcal{D}.
\]
Notice that there exist $\bbL$-domains $\mcal{D},\mcal{D}'$ such that $G_{\mcal{D}}=G_{\mcal{D}'}$ but $\partial_{\mcal{D}}G_{\mcal{D}}\neq\partial_{\mcal{D}'}G_{\mcal{D}'}$.

\subsubsection{Properties of the quasi-FK measures}
\label{sec:quasi-FK_prop}
In this section, we establish certain properties of the quasi-FK measures that we will need in the proof of Lemma~\ref{lem:proxi_lower_int}.

\subsubsection*{Positive association and stochastic domination.} 
The quasi-FK measures are positively associated if the cluster-weights are ordered in a specific way. 
\begin{lemma}\label{lem:qfk_fkg}
Let~\(G=(V,E)\) be a finite graph, and let~\(\mfr{b}_1,\mfr{b}_2\subset V\) be disjoint. For~\(p\in (0,1)\) and~\(1\leq q_1\leq q_2\leq q\), the measure~\(\qfk_{G;p,q}^{\mfr{b}_1,\mfr{b}_2;q_1,q_2}\) satisfies the strong FKG property.
\end{lemma}

\begin{proof}
Let~\(e=xy\in E\) and~\(\eta\in\{0,1\}^{E\setminus\{e\}}\), and let~\(\eta_e,\eta^e\in\{0,1\}^{E}\) coincide with~\(\eta\) on~\(E\setminus\{e\}\) while~\(\eta_e(e)=0\) and~\(\eta^e(e)=1\).
By the Holley criterion~\cite{Hol74} (see also~\cite[Chapter 2]{Gri06}), it suffices to show that
\begin{equation*}
\qfk_{G;p,q}^{\mfr{b}_1,\mfr{b}_2;q_1,q_2}(\omega(e)=1\given\omega=\eta\text{ off }e)
=\frac{1}{1+\tfrac{1-p}{p}\,q^{\clusters_\mrm{i}(\eta_e)-\clusters_\mrm{i}(\eta^e)}\,q_{1}^{\clusters_{\mfr{b}_1}(\eta_e)-\clusters_{\mfr{b}_1}(\eta^e)}\,q_{2}^{\clusters_{\mfr{b}_2}(\eta_e)-\clusters_{\mfr{b}_2}(\eta^e)}}
\end{equation*}
is increasing in~\(\eta\).
Let~\(C_x\) and~\(C_y\) be the clusters of~\(x\) and~\(y\) in~\(\eta\), respectively, and set~\(\mfr{b}=\mfr{b}_1\cup\mfr{b}_2\). Then, taking the reciprocal,
\begin{equation}\label{eq:prf:qfk_fkg}
\begin{multlined}
\frac{1}{\qfk_{G;p,q}^{\mfr{b}_1,\mfr{b}_2;q_1,q_2}(\omega(e)=1\given\omega=\eta\text{ off }e)}\\
=\begin{cases}
1+\tfrac{1-p}{p} &\text{if } C_x=C_y,\\
1+\tfrac{1-p}{p}q_1 &\text{if } C_x\neq C_y,\,C_x\cap\mfr{b}_1\neq\varnothing,\,C_y\cap\mfr{b}_1\neq\varnothing,\\
1+\tfrac{1-p}{p}q_2 &\text{if } C_x\neq C_y,\,C_x\cap\mfr{b}_1=\varnothing,\,C_x\cap\mfr{b}_2\neq\varnothing,\,C_y\cap\mfr{b}\neq\varnothing,\,\text{or vice versa},\\
1+\tfrac{1-p}{p}q &\text{if } C_x\neq C_y,\,C_x\cap\mfr{b}=\varnothing\text{ or }C_y\cap\mfr{b}=\varnothing,
\end{cases}
\end{multlined}
\end{equation}
which is decreasing in~\(\eta\) since~\(1\leq q_1\leq q_2\leq q\) by assumption.
\end{proof}

In the proof of Lemma~\ref{lem:proxi_lower_int}, we will encounter quasi-FK measures with boundary-weights~\(q_1=1\) and~\(q_2=\qbwired\).
The following stochastic domination statement will allow us to compare with the measure with homogeneous quasi-wired boundary-weight.
\begin{lemma}\label{lem:qfk_stoch_dom}
Let~\(G=(V,E)\) be a finite graph, and let~\(\mfr{b}_1,\mfr{b}_2\subset V\) be disjoint. For any~\(p\in (0,1)\) and~\(1\leq q_\mrm{b}\leq q\),
\begin{equation*}
\qfk_{G;p,q}^{\mfr{b}_1,\mfr{b}_2;1,q_\mrm{b}}\geq_\mrm{st}
\qfk_{G;p,q}^{\mfr{b}_1,\mfr{b}_2;q_\mrm{b},q_\mrm{b}}=
\qfk_{G;p,q}^{\mfr{b}_1\cup \mfr{b}_2;q_\mrm{b}}.
\end{equation*}
\end{lemma}
\begin{proof}
The equality is a simple consequence of the definition. 
It remains to prove the stochastic domination statement.
In the notation of the proof of Lemma~\ref{lem:qfk_fkg}, by the Holley criterion~\cite{Hol74} (see also~\cite[Chapter 2]{Gri06}), it suffices to show that
\begin{equation*}
\qfk_{G;p,q}^{\mfr{b}_1,\mfr{b}_2;q_1,q_2}(\omega(e)=1\given\omega=\eta\text{ off }e)
\end{equation*}
is increasing in~\(\eta\) and decreasing in~\(q_1\in [1,q_2]\). This follows directly from~\eqref{eq:prf:qfk_fkg}.

\end{proof}

\subsubsection{Exponential relaxation of quasi-wired measures.} 

\begin{lemma}\label{lem:qfk_relax}
Let $q>4$ and $p=p_c(q)$. There exist constants $C,c>0$ such that the following holds. Let $\mcal{D}$ be an $\bbL$-domain, and consider the corresponding weak $\bbL_\bullet$-domain $G_\mcal{D}=(V,E)$. Let~\(\mfr{b}=\partial_{\mcal{D}}G_{\mcal{D}}\) be the $\mcal{D}$-boundary of~\(G_\mcal{D}\). Then, for any $F\subseteq E$, 
\begin{equation*}
\rmd_{\mrm{TV}}\big(\qfk_{G_{\mcal{D}};p,q}^{\mfr{b};\qbwired}\vert _F,\fk_{p,q}^{\fkwired}\vert_F\big)\leq C\,\abs{\bbV_F}\,\big(\mrm{diam}(\bbV_F)+\rmd_\infty(\bbV_F,V^c)\big)^2\,e^{-c\,\rmd_\infty(\bbV_F,V^c)}.
\end{equation*}
\end{lemma}

\begin{proof}
The statement follow directly from the analogue~\cite[Proposition 5.8]{DobGlaOtt25} for the six-vertex height function measures and the fact that the quasi-FK measures are obtained from them by local operations. See Section~\ref{sec:coupling:sixv_to_fk} and Figure~\ref{fig:6V_to_oriented_loops} for the local sampling rule, and~\cite[Section 3]{GlaPel23} for a proof.
\end{proof}

\begin{remark}
An analogue of the statement for the measures~\(\qfk_{G_{\mcal{D}};p,q}^{\mfr{b}_1,\mfr{b}_2;1,\qbwired}\) with inhomogeneous boundary weights~\(1\) and~\(\qbwired\) on partitions~\(\mfr{b}_1,\mfr{b}_2\) of~\(\mfr{b}\) can be proved in exactly the same way. However, the coupling in~\cite[Section 3]{GlaPel23} is not formulated in this generality.  
\end{remark}

\subsubsection{Box crossing property at the boundary.}
Consider the quasi-wired FK measure on a weak $\bbL_\bullet$-domain with $q>4$ and $p=p_c$. As a consequence of the relaxation statement, Lemma~\ref{lem:qfk_relax}, we can find connections in the bulk of the domain with high probability. In Section~\ref{sec:proof_proxi_lower_int}, we will require this property also close to the boundary.

A \emph{rectangle of size}~\(k\) by~\(l\) is a set of the form $[a,b]\times[c,d]$ with $b-a=k$ and $d-c=l$.
We say that a rectangle $R:=[a,b]\times[c,d]$ in $\bbL$ is \emph{crossed horizontally} in $\omega\in\{0,1\}^{\bbE_R}$ if there exists a path in $\omega$ connecting 
\[
\mrm{L}_R:=\{a\}\times [c,d]\quad\text{and}\quad \mrm{R}_R:=\{b\}\times [c,d].
\] 
We denote this event by $\mcal{H}(R)$. Similarly, $R$ is \emph{crossed vertically} in $\eta\in\{0,1\}^{\bbE_R}$ if there exists a path in $\eta$ connecting
\[
\mrm{T}_R:=[a,b]\times\{d\}\quad\text{and}\quad \mrm{B}_R:=[a,b]\times\{c\},
\] 
and we denote this event by $\mcal{V}(R)$.
Recall the definition of $\bbL$-domains and weak $\bbL_\bullet$-domains in Section~\ref{sec:rotated_lattice} and the quasi-FK measures introduced in Section~\ref{sec:quasi_fk_def}.

\begin{lemma}\label{lem:box_crossing_qfk}
Let $q>4$ and $p=p_c(q)$. There exist $C,c>0$ such that the following holds for $\varepsilon>0$ sufficiently small. 
Let $\mcal{D}$ be an $\bbL$-domain, and consider the corresponding weak $\bbL_\bullet$-domain $G=G_\mcal{D}=(V,E)$ with its $\mcal{D}$-boundary $\mfr{b}:=\partial_{\mcal{D}}G_{\mcal{D}}$.
Let $k>0$, and let $R\subseteq V$ be a rectangle of size $\lfloor\varepsilon k\rfloor$ by $k$ such that $\mrm{L}_R\subset\mfr{b}$. Let $R'\subset R$ be the unique rectangle of size $\lfloor\varepsilon k/2\rfloor$ by $k$ such that $\mrm{L}_{R'}=\mrm{L}_{R}$. Then,  
\[
\qfk_{G;p,q}^{\mfr{b};\qbwired}[\mcal{H}(R')]\geq 1-Ck^4\,e^{-c\varepsilon k}.
\]
\end{lemma}

\begin{proof}
Let $0<\varepsilon<1/9$, and take $q,p,G,k,R,R'$ as in the statement. The constants $C,C',c,c'>0$ that appear throughout the proof depend on $q$ only and may change from line to line. 
To lighten notation, we omit the floor and ceiling brackets, and we omit $p$ and $q$ from the subscripts. We can assume without loss of generality that 
\[
R=[0,\varepsilon k]\times[0,k]\quad\text{so that}\quad R'=\left[0,\tfrac{\varepsilon k}{2}\right]\times[0,k].
\]
Define the rectangles (see Figure~\ref{fig:box_crossing})
\begin{align*}
R_1:=\left[0,\tfrac{\varepsilon k}{2}\right]&\times \left[\tfrac{\varepsilon k}{2},(1-\tfrac{\varepsilon}{2})k\right]\subset R',\\
R_2:=\left[\tfrac{\varepsilon k}{2},\varepsilon k\right]&\times \left[\tfrac{\varepsilon k}{2},(1-\tfrac{\varepsilon}{2})k\right]\subset R\setminus R'.
\end{align*}

\begin{figure}
\includegraphics[scale=0.5,page=1]{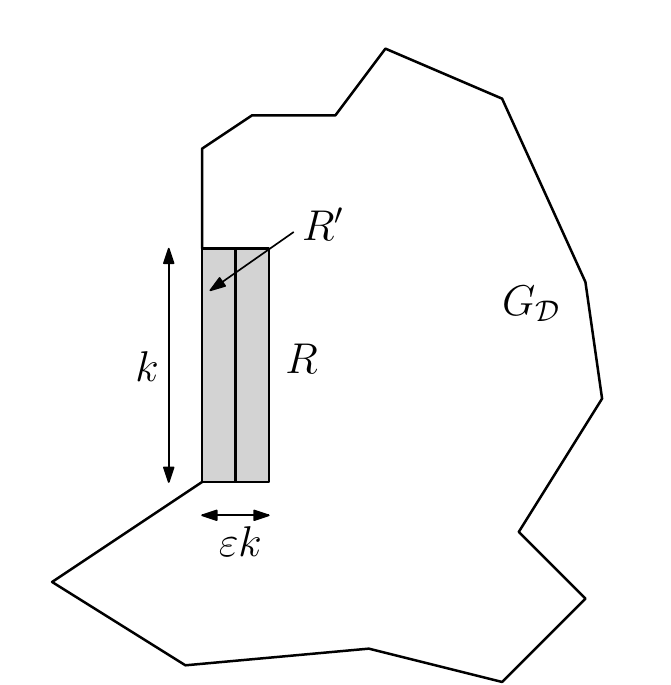}
\qquad
\includegraphics[scale=0.5,page=2]{box-crossing_areas}
\caption{Left: a weak \(\bbL_\bullet\)-domain~\(G_\mcal{D}\) and a rectangle~\(R\) (light grey) of size~\(\lfloor\varepsilon k\rfloor\) by~\(k\) with~\(\mrm{L}_R\subset\partial_\mcal{D} G_\mcal{D}\). Right: the rectangle~\(R\) (light grey) containing the rectangles~\(R_1\) and~\(R_2\), the annulus~\(Q\setminus Q'\) (dark grey rising lines) and the square~\(Q'\) (white) inside the square~\(Q\).}
\label{fig:box_crossing}
\end{figure}

Notice that $\mcal{H}(R_1)\subset\mcal{H}(R')$. Moreover, by Lemma~\ref{lem:qfk_relax}, we have
\begin{equation}\label{eq:box_cr_1}
\qfk_{G}^{\mfr{b};\qbwired}[\mcal{V}(R_2)^c]\leq C\varepsilon k^4 e^{-c\varepsilon k}+C'ke^{-c'\varepsilon k}.
\end{equation}
Hence, we can assume that $\mcal{V}(R_2)$ occurs.
Now, consider the event $\mcal{A}$ that all vertices on the top $\mrm{T}_{R_1\cup R_2}$ are connected to each other outside $R_1\cup R_2$ and that all vertices on the bottom $\mrm{B}_{R_1\cup R_2}$ are connected to each other outside $R_1\cup R_2$. By the finite energy property, there exists $c_\mrm{FE}>0$ depending on $q$ only such that
\begin{equation}\label{eq:box_cr_2}
\qfk_{G}^{\mfr{b};\qbwired}[\mcal{H}(R_1)^c\cap\mcal{V}(R_2)]\leq 
\qfk_{G}^{\mfr{b};\qbwired}\left[\mcal{H}(R_1)^c\cap\mcal{V}(R_2)\cap \mcal{A}\right]e^{c_\mrm{FE}\varepsilon k}.
\end{equation}
Conditioning on $\mcal{V}(R_2)\cap \mcal{A}$, positive association (Lemma~\ref{lem:qfk_fkg}) implies 
\begin{equation}\label{eq:box_cr_3}
\qfk_{G}^{\mfr{b};\qbwired}\left[\mcal{H}(R_1)^c\cap\mcal{V}(R_2)\cap \mcal{A}\right]\leq \qfk_{R_1\cup R_2}^{\mfr{b}_1,\mfr{b}_2;1,\qbwired}[\mcal{H}(R_1)^c]=:r,
\end{equation} 
where $\mfr{b}_1=(\partialin (R_1\cup R_2))\setminus \mrm{L}_{R_1}$ and $\mfr{b}_2=\mrm{L}_{R_1}$.
Now, consider the nested squares
\[
Q:=[0,(1-\varepsilon)k]\times \left[\tfrac{\varepsilon k}{2},(1-\tfrac{\varepsilon}{2})k\right]\quad\text{and}\quad
Q':=\left[\varepsilon k, (1-2\varepsilon)k\right]\times\left[\tfrac{3\varepsilon k}{2},(1-\tfrac{3\varepsilon}{2})k\right].
\]
Then, 
\begin{equation}\label{eq:box_cr_4}
r^4\leq\qfk_{Q}^{\qbwired}[Q'\text{ is surrounded by a dual circuit}]\leq Ck^4 e^{-c\varepsilon k},
\end{equation}
where we used positive association (Lemma~\ref{lem:qfk_fkg}) and symmetry (invariance of the measure~\(\qfk_{Q}^{\qbwired}\) under a rotation of~\(\pi/2\) about the centre of~\(Q\)) for the first inequality, and Lemma~\ref{lem:qfk_relax} for the second one. Combining equations~\eqref{eq:box_cr_1}--\eqref{eq:box_cr_4}, we obtain 
\[
\qfk_{G}^{\mfr{b};\qbwired}[\mcal{H}(R')^c]\leq 
C k^4 e^{-c \varepsilon k}+C'ke^{-c' k}e^{c_\mrm{FE}\varepsilon k}.
\]
Taking $\varepsilon<\min\{1/4,c'/c_\mrm{FE}\}$ finishes the proof.
\end{proof}

\subsubsection{Proof of Lemma~\ref{lem:proxi_lower_int}}
\label{sec:proof_proxi_lower_int}
Recall the setting of Section~\ref{sec:coupling_under_dobr} and the definitions of the interfaces in~\eqref{eq:interfaces_def}. Since \(q,p_c,\svc,\svcb\) as in~\eqref{eq:parameters_bulk}-\eqref{eq:parameters_bnd} are again fixed, we omit them from the subscripts of the measures.
As described at the beginning of Section~\ref{sec:proxi_int}, we aim to show that the lower FK interface~\(\Gamma_\fk^2\) is close to the Peierls contour between~\(\pm1\) in the six-vertex spin configuration~\(\sigma_\bullet\).

Consider the cluster~\(\mcal{C}_+\subset\bar{\msf{V}}\subset\bbL_\bullet\) of the boundary~\(\partialin\bar{\msf{V}}\) in the subgraph of~\(\msf{K}\) induced by the vertices~\(i\in\bar{\msf{V}}\) for which~\(\sigma_\bullet(i)=+1\).
Let~\(\mcal{E}'\subset *\partialedge_{\msf{K}}\mcal{C}_+\subset*\bar{\msf{E}}\subset\bbE^\circ\) be the path connecting~\(v'_L\) and~\(v'_R\).
First observe that~\(\sigma_\bullet(i)\neq\sigma_\bullet(j)\) for all~\(ij\in *\mcal{E}'\). Therefore, as~\(\sigma_\bullet\) is constant on edges in~\(\omega_{\tau\tau'}\) in the coupling~\(\psi^{1/1}_\msf{G}\), it holds that~\(\mcal{E}'\subset\omega_{\tau\tau'}^*\). This implies~\(\bbV_{\mcal{E}'}\subset\mcal{C}'_{v'_L}\), whence
\begin{equation}\label{eq:prf:proxi_lower_int1}
\rmd_\mrm{H}(\Gamma_\fk^2,\mcal{C}'_{v'_L})\leq
\rmd_\mrm{H}(\Gamma_\fk^2,\bbV_{\mcal{E}'}).
\end{equation}
Moreover, since~\(\sigma_\bullet\) is also constant on the FK random variable~\(\omega\), it also holds that~\(\mcal{E}'\subset\omega^*\), whence the lower FK interfaces~\(\Gamma_\fk^2\) lies below~\(\bbV_{\mcal{E}'}\). Therefore, it suffices to show that~\(\Gamma_\fk^2\) does not go far below~\(\mcal{E}'\).

\begin{figure}
\includegraphics[scale=0.35,page=1]{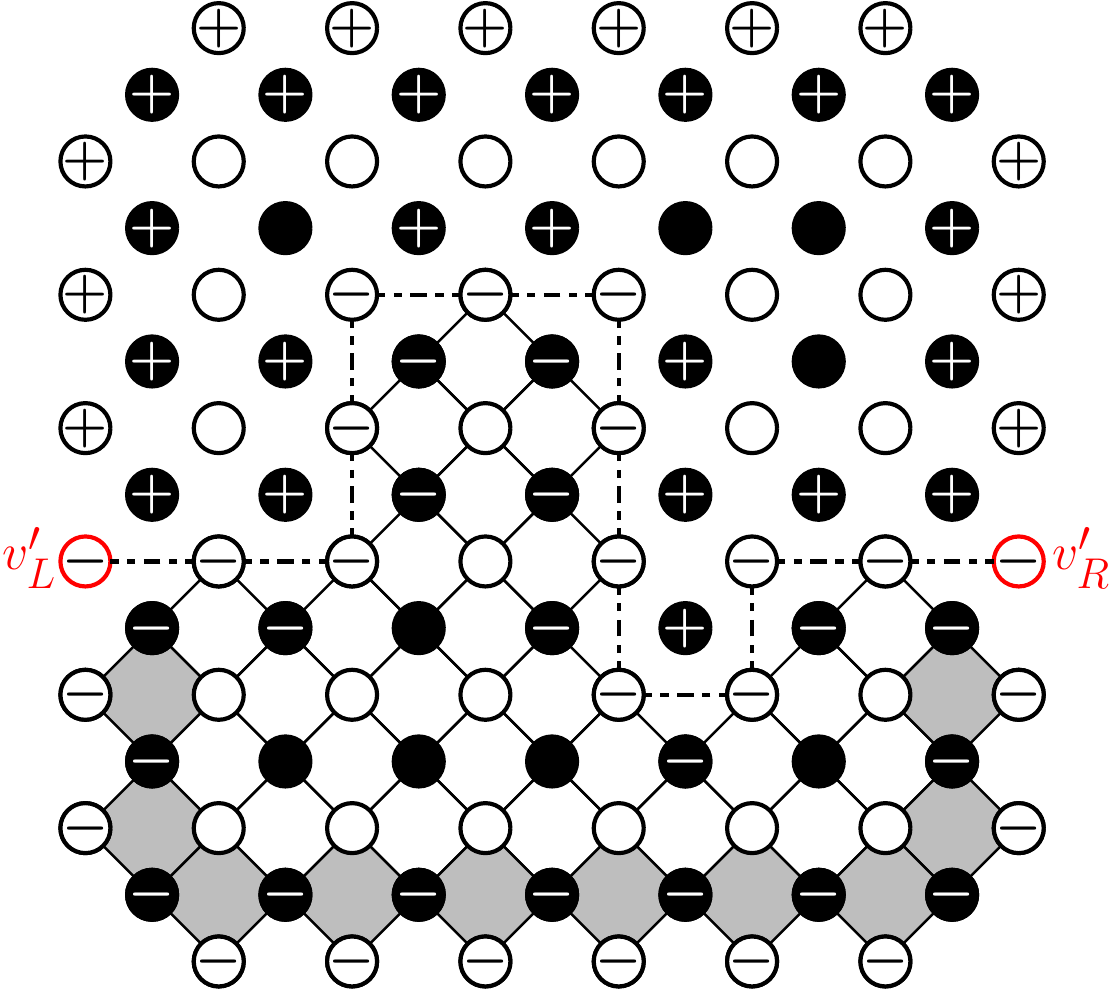}
\quad
\includegraphics[scale=0.35,page=2]{order-order_prox_lower_int_domain}
\caption{Left: A realisation~\(E'\) (dashed edges) of the Peierls contour~\(\mcal{E}'\) from~\(v'_L\) to~\(v'_R\) separating~\(\sigma_\bullet=+\) and~\(\sigma_\bullet=-\), and the tiles in~\(\msf{A}_{E'}^\rmi\) (white) and in~\(\msf{A}_{E'}^\rmb\) (grey). Right: The associated graph~\(\msf{G}_{E'}=(\msf{V}_{E'},\msf{E}_{E'})\) on which the quasi-FK random variable is defined. Boundary vertices in~\(\mfr{b}_1\) and~\(\mfr{b}_2\) are marked with their assigned boundary-weights~\(1\) and~\(\qbwired\), respectively.}
\label{fig:prox_lower_int_domain}
\end{figure}

We will now determine the conditional distributions of~\(\sigma\) and~\(\omega\) below~\(\mcal{E}'\). First, we need to define the relevant graphs and measures. Fix a realisation~\(E'\subset*\bar{\msf{E}}\) of~\(\mcal{E}'\). Let~\(\msf{D}_{E'}\) be the \(\bbL\)-domain given by the vertices in~\(\msf{D}\setminus\bbV_{E'\cup *E'}\) below~\(E'\) (see the left side of Figure~\ref{fig:prox_lower_int_domain}).
Moreover, let~\(\msf{A}_{E'},\msf{A}_{E'}^\rmi,\msf{A}_{E'}^\rmb\) be the tiles in~\(\msf{A},\msf{A}^\rmi,\msf{A}^\rmb\), respectively, that intersect~\(\msf{D}_{E'}\).
Let~\(\spin_{\msf{D}_{E'}}^{-,-}\) be the probability measure on~\(\sigma\in\{\pm1\}^{\bbL_\bullet}\times\{\pm1\}^{\bbL_\circ}\) defined by
\begin{equation*}
\spin_{\msf{D}_{E'}}^{-,-}(\sigma)
\propto
\svc^{|T_{5,6}(\sigma)\cap\msf{A}_{E'}^\rmi|}\,\svcb^{|T_{5,6}(\sigma)\cap\msf{A}_{E'}^\rmb|}\,\mathds{1}_{\sigma(x)=-1\,\forall x\in\bbL\setminus\msf{D}_{E'}}\,\mathds{1}_{\mathrm{ice}}(\sigma),
\end{equation*}
where~\(T_{5,6}(\sigma)\) is again the set of tiles of types 5-6 in~\(\sigma\).
Then, the spatial Markov property of the six-vertex spin measures implies the following.

\begin{claim}\label{claim:prf:proxi_lower_int1}
The conditional law of~\(\sigma\) below~\(E'\) is given by
\begin{equation*}
\Psi^{1/1}_\msf{G}\big(\sigma|_{\msf{D}_{E'}}\in\cdot\bgiven\mcal{E}'=E'\big)
=\spin_{\msf{D}_{E'}}^{-,-}(\sigma|_{\msf{D}_{E'}}\in\cdot\,).
\end{equation*}
\end{claim}

Now, let~\(\msf{E}_{E'}=\{e_t:t\in\msf{A}_{E'}^\rmi\}\) and~\(\msf{V}_{E'}=\bbV_{\msf{E}_{E'}}\) (see the right side of Figure~\ref{fig:prox_lower_int_domain}), and set~\(\msf{G}_{E'}=(\msf{V}_{E'},\msf{E}_{E'})\).
Moreover, let~\(\mfr{b}_1\) be the vertices in~\(\partialin_{\bbL_\bullet}\msf{V}_{E'}\) that belong to a tile in~\(\msf{A}_{E'}^\rmb\), and let~\(\mfr{b}_2\) be the vertices in~\(\partialin_{\bbL_\bullet}\msf{V}_{E'}\) that do not belong to a tile in~\(\msf{A}_{E'}^\rmb\) (see again the right side of Figure~\ref{fig:prox_lower_int_domain}). 
Using Claim~\ref{claim:prf:proxi_lower_int1}, the `reverse sampling procedure' of the BKW coupling (see Section~\ref{sec:coupling:sixv_to_fk}) implies the following.

\begin{claim}\label{claim:prf:proxi_lower_int2}
The conditional law of~\(\omega\) below~\(E'\) is given by
\begin{equation*}
\Psi^{1/1}_\msf{G}\big(\omega|_{\msf{E}_{E'}}\in\cdot\bgiven\mcal{E}'=E'\big)
=\qfk_{\msf{G}_{E'}}^{\mfr{b}_1,\mfr{b_2};1,\qbwired}.
\end{equation*}
\end{claim}

By Lemma~\ref{lem:qfk_stoch_dom}, the measure on the right side stochastically dominates that with homogeneous quasi-wired boundary weight: 
\begin{equation}\label{eq:prf:proxi_lower_int2}
\qfk_{\msf{G}_{E'}}^{\mfr{b}_1,\mfr{b_2};1,\qbwired}\geq_\mrm{st}\qfk_{\msf{G}_{E'}}^{\mfr{b}_1\cup\mfr{b_2};\qbwired}.
\end{equation}
Now, condition on~\(\mcal{E}'=E'\) and assume that
\begin{equation}\label{eq:prf:proxi_lower_int3}
\rmd_{\mrm{H}}(\Gamma_\fk^2,E')>k.
\end{equation}
We distinguish between the cases when~\(\Gamma_\fk^2\) enters the bulk of the domain or when it remains at the boundary.
Set $r:=\lfloor\varepsilon k/2\rfloor$, where $\varepsilon>0$ is as in Lemma~\ref{lem:box_crossing_qfk}.
\smallskip

\noindent\textbf{Case 1.} There exists $i\in\Z^2$ just below~$Gamma_{\fk}^2$ such that $i+\msf{B}_r\subseteq\msf{V}_{E'}$.

Then, we must have that $u\in i+\{(\tfrac{1}{2},\tfrac{1}{2}),(-\tfrac{1}{2},\tfrac{1}{2}),(\tfrac{1}{2},-\tfrac{1}{2}),(-\tfrac{1}{2},-\tfrac{1}{2})\}$ is connected to $u+\partialin\msf{B}_{r/2}$ in $\omega^*$ (see Figure~\ref{fig:prox_lower_int_cases}). Denote the event of the existence of such $u$ by $\mcal{A}$. As~\(\mcal{A}\) is a decreasing event, the stochastic domination statement~\eqref{eq:prf:proxi_lower_int2} and exponential relaxation (Lemma~\ref{lem:qfk_relax}) of the measure on the right side of~\eqref{eq:prf:proxi_lower_int2} imply
\begin{align*}
\qfk_{\msf{G}_{E'}}^{\mfr{b}_1,\mfr{b_2};1,\qbwired}(\mcal{A})\leq\qfk_{\msf{G}_{E'}}^{\mfr{b}_1\cup\mfr{b}_2;\qbwired}(\mcal{A})&\leq \rmd_{\mrm{TV}}\big(\qfk_{\msf{G}_{E'}}^{\mfr{b}_1\cup\mfr{b}_2;\qbwired}\vert _F,\fk^{\fkwired}\vert_F\big)+\fk^\fkwired(\mcal{A})\\
&\leq Cnm\max\{n,m\}^2e^{-c \frac{r}{2}}+C'nm e^{-c' \frac{r}{2}},
\end{align*}
where~\(F=\{e\in\msf{E}_{E'}:\rmd_\infty(e,\msf{V}_{E'}^c)\geq r/2\}\) is the support of~\(\mcal{A}\).

\begin{figure}
\includegraphics[scale=0.7]{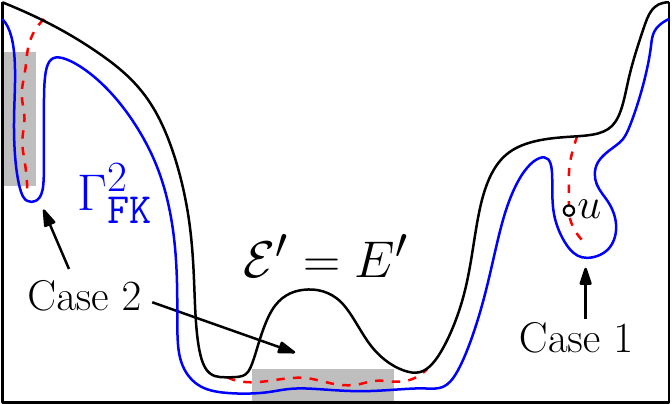}
\caption{A realisation~\(E'\) of~\(\mcal{E}'\) (that will not occur with high probability as it touches the bottom part of the boundary) and a realisation of~\(\Gamma_\fk^2\) (solid blue) for which the consequences of both Case 1 and 2 occur. Dual connections (\(\omega^*\)) are drawn in dashed red.}
\label{fig:prox_lower_int_cases}
\end{figure}

\smallskip

\noindent\textbf{Case 2.} For all $i\in\Z^2$ just below~$\Gamma_{\fk}^2$, we have~$(i+\msf{B}_r)\cap \msf{V}_{E'}^c\neq\varnothing$. 

Then, by assumption~\eqref{eq:prf:proxi_lower_int3} and the definition of the one-sided Hausdorff distance~\eqref{eq:hausdorff}, there must exist a rectangle of size $r$ by $k$ whose left or right side is contained in $\mfr{b}_1$ and that is crossed vertically in $\omega^*$, or there exists a rectangle of size $k$ by $r$ whose bottom side is contained in $\mfr{b}_1$ and that is crossed horizontally in $\omega^*$ (see again Figure~\ref{fig:prox_lower_int_cases}). By~\eqref{eq:prf:proxi_lower_int2} and Lemma~\ref{lem:box_crossing_qfk}, this happens with probability at most $Cnmk^4e^{-c\varepsilon k}$.

This completes the proof of Lemma~\ref{lem:proxi_lower_int}.

\section{Geometry of double-crossing}
\label{sec:DoubleBridge}

We now investigate the stochastic geometric properties of the two crossing induced by the conditioning in $\matrc_{\mathsf{K}}(\, \cdot \given v_R\in \calC_{v_L},\, v'_R\in \calC'_{v_L'})$. To lighten notations, let \(m_n\) be an increasing sequence of integers going to \(+\infty\), and set
\begin{equation*}
	\Phi_{n} \equiv \matrc_{\mathsf{K}} \equiv \matrc_{n,m_n},
	\quad
	\Phi \equiv \atrc
\end{equation*}for this section. Constraints on the sequence \(m_n\) will be explicitly given in the concerned results. The general strategy is similar to the one of~\cite[section 8]{DobGlaOtt25}: we first use mixing to relate the typical geometry of crossing clusters to the geometry of infinite volume ATRC long connections; then we use (again) mixing to couple the crossing clusters to infinite volume long clusters, which are themselves coupled to random walks using OZ theory. The main difference compared to~\cite[section 8]{DobGlaOtt25} is that we are studying two simultaneous crossings, and therefore need to study the interaction between them. This will be done using adaptation of the method introduced in~\cite{IofOttVelWac20} to study entropic repulsion of interfaces in the Potts model at \(\beta>\beta_c\). We will in particular import part of the analysis~\cite{DAl24}, which generalized the arguments of~\cite{IofOttVelWac20} to the case of a family of crossing high temperature FK clusters conditioned on non-intersection.

Also, to phrase various results and do some computations without having the size of displays becoming unbounded, it will be convenient to alternate between talking about probability of events under given measures, and the behaviour of random variables, with law given by the said measures, defined on a common space. We let \((\calX,\calF,P)\) be an abstract probability space, supposed large enough, on which we will define all the wanted random variables. Denote by \(E\) the expectation in that space.

\subsection{Cones and infinite volume OZ theory}
\label{subsec:DoubleBridge:Cones_OZ}

We start by introducing several standard definitions in the same way as in ~\cite[section 8]{DobGlaOtt25}. In particular, we consider cones of opening angle~$\pi/2$:
\begin{equation*}
	\fcone = \{x\in \R^2:\ x_1\geq |x_2|\},
	\quad
	\bcone = -\fcone.
\end{equation*}Then, for \(V\subset \Z^2\), say that \(v\in V\) is a \emph{cone-point} of \(V\) if
\begin{equation*}
	V\subset (v+\fcone)\cup (v+\bcone).
\end{equation*}Denote \(\CPts(V)\) the set of cone-points of \(V\). Also introduce the strip and the restriction of \(\CPts\) to a strip: for \(a\leq b\),
\begin{equation}
	\label{eq:def:slab_CP}
	\slab_{a,b} = [a,b]\times \R,
	\quad
	\slabCP_{a,b}(V) = \slab_{a,b}\cap \CPts(V).
\end{equation}

Recall that, see~\cite{DobGlaOtt25} for more details, \(\nu\) is the norm on \(\R^2\) such that \(\Phi(0\leftrightarrow x) = e^{-\nu(x) + o(|x|)}\). Let \(\rme_1 = (1,0)\), \(\rme_2=(0,1)\) be the canonical basis of \(\R^2\). Introduce also the shorthand
\begin{equation*}
	\nu_1 = \nu(\rme_1).
\end{equation*}Introduce the following objects.
\begin{itemize}
	\item A \emph{marked graph} is a pair \((\eta,v_*)\) with \(\eta\) a connected subgraph of \(\Z^2\) and \(v_*\) a vertex in \(\eta\).
	\item A connected graph \(\eta\) is called: \emph{forward confined} if there is \(u\in\eta\), denoted \(\fend(\eta)\), with \(\eta\subset u+\fcone\); \emph{backward confined} if there is \(v\in\eta\), denoted \(\bend(\eta)\), with \(\eta\subset v+\bcone\); \emph{diamond confined} if it is both forward and backward confined.
	\item \(\SetRootMarkBackCont\) is the set of marked graphs of the form \((\eta,0)\) with \(\eta\) backward confined.
	\item \(\SetRootMarkForwCont\) is the set of marked graphs of the form \((\eta,v_*)\) with \(\eta\) forward confined with \(\fend(\eta) = 0 \).
	\item \(\SetRootDiaCont\) is the set of diamond confined graphs \(\eta\) with \(\fend(\eta)=0\). It can be seen as a subset of \(\SetRootMarkBackCont\) via the marking of \(\fend(\eta) = 0\), or as a subset of \(\SetRootMarkForwCont\) via the marking of \(\bend(\eta)\).
	\item The \emph{displacement} \(\displace(\eta)\) is defined as follows: \(\displace(\eta) := \bend(\eta)\) if \((\eta,0)\in \SetRootMarkBackCont\); \(\displace(\eta) := v_*\) if \((\eta,v_*)\in \SetRootMarkForwCont\); \(\displace(\eta) := \bend(\eta)-\fend(\eta)\) if \(\eta\) is diamond confined.
	\item The \emph{concatenation} \(\concatenate\) of \((\eta,0)\in \SetRootMarkBackCont\) and \((\eta',v_*)\in \SetRootMarkForwCont\) is
	\begin{equation*}
		\eta\concatenate\eta' = \eta \cup (\displace(\eta) + \eta')
	\end{equation*}with marked vertex \(0\).
\end{itemize}
To lighten notations, we will usually omit mentioning the marked points. Also, to prevent displays from reaching uncontrolled width, we sill use the following notations for sequences of backward/forward/diamond-confined graphs: for \(m\geq 0\), \(\gamma_0\in \SetRootMarkBackCont\), \(\gamma_{m+1}\in \SetRootMarkForwCont\), \(\gamma_1,\dots,\gamma_{m}\in \SetRootDiaCont\), define
\begin{equation}
	\label{eq:def:sequences_concatenation_short}
	\gamma_{i}^j = (\gamma_i,\gamma_{i+1},\dots,\gamma_j),
	\quad
	\bar{\gamma}_{i}^j = \gamma_i\concatenate\gamma_{i+1}\concatenate\dots\concatenate\gamma_j,
	\
	0\leq i \leq j\leq m+1.
\end{equation}

We will then use the same ``infinite volume OZ theory'' of ATRC clusters as in~\cite[Section 8]{DobGlaOtt25}. Let us recall the statement that we will use. This is a particular case of~\cite[Theorem~7.1]{DobGlaOtt25}, see~\cite[Section 8]{DobGlaOtt25}.

\begin{theorem}[Infinite volume OZ decomposition]
	\label{thm:OZ_atrc_inf_vol}
	Let \(0<J<U\) satisfy \(\sinh 2J=e^{-2U}\).
	 Then, there exist \(n_0\geq 1,c>0,c'>0,\epsilon>0,\rmC>0\), and probability measures \(p,p_L,p_R\) on \(\SetRootDiaCont,\SetRootMarkBackCont,\SetRootMarkForwCont\) respectively such that for any \(x=(x_1,x_2)\in \Z^2\) with \(x_1 \geq n_0\), and \(|x_2|\leq \epsilon x_1\), and any function \(f\) on subsets of~$\Z^2$
	\begin{multline*}
		\Big| \sum_{m\geq 1}\sum_{\gamma_0\in \SetRootMarkBackCont}\sum_{\gamma_{m+1}\in \SetRootMarkForwCont} \sum_{\gamma_1,\dots,\gamma_m\in \SetRootDiaCont} \OZDec(m;\gamma_0^{m+1}) f\big(\bar{\gamma}_0^{m+1}\big) \mathds{1}_{\displace(\bar{\gamma}_0^{m+1}) = x}
		\\
		- e^{\nu_1 x_1}\Phi\big[f(\calC_0) \mathds{1}_{x\in \calC_0}\big] \Big|
		\leq \norm{f}_{\infty} e^{-c x_1},
	\end{multline*}where \(\calC_0\) is the connected component of \(0\) in \(\omega_{\tau}\), and
	\begin{equation}
		\label{eq:def:OZDec}
		\OZDec(m;\gamma_0^{m+1}) \coloneqq \rmC p_L(\gamma_L)p_R(\gamma_R)\prod_{i=1}^{m}p(\gamma_i).
	\end{equation}Moreover, \(p\) inherits the invariance under reflection through \(\R\times \{0\}\) from \(\Phi\) and, for any \(r>0\) and \(\gamma'\in \SetRootDiaCont\) such that \(\CPts(\gamma') = \{0,\displace(\gamma')\}\),
	\begin{equation}
		\label{eq:OZ:fin_ene_exp_dec}
		\sum_{\gamma:\norm{\displace(\gamma)}_{\infty} \geq r} p(\gamma) \leq e^{-cr}
		\quad
		\text{and}
		\quad
		p(\gamma') \geq e^{-c'|\gamma'|}.
	\end{equation}
	Similar statements hold also for \(p_L,p_R\) and \(\gamma'\in \SetRootMarkBackCont, \SetRootMarkForwCont\) respectively.
\end{theorem}
We will also use the probability measures
\begin{equation}
	\label{eq:def:OZ_OZwalk}
	\OZ \coloneqq p^{\otimes \N},
	\quad
	\OZwalk \coloneqq (p\circ \displace^{-1})^{\otimes \N}.
\end{equation}
For \((X_1,X_2,\dots) \sim \OZwalk\), and \(v\in \Z^2\), consider the {\em hitting time}
\begin{equation}
	\label{eq:def:hitting_times}
	T_v = \min\Big\{k\geq 1:\ \sum_{i=1}^k X_i = v \Big\},
\end{equation}where the min of an empty set is \(+\infty\) by convention. Note that these hitting times are also well-defined under \(\OZ\) by replacing \(X_i\) by \(\displace(\gamma_i)\). Define the bridge measures
\begin{equation}
	\label{eq:def:bridge_measures}
	\OZ_v = \OZ( \gamma_{1}^{T_v}\in \cdot \given T_v<\infty),
	\quad
	\OZwalk_v = \OZ( (X_1,\dots,X_{T_v})\in \cdot \given T_v<\infty).
\end{equation}
For \(\gamma_1^M\sim \OZ_v\), there is a naturally associated cluster: \(\OZRVcluster^v = \gamma_1\concatenate \dots\concatenate \gamma_M\).

Finally, for a sequence \(x_0,x_1,\dots,x_m\) with \(x_{k}\in x_{k-1}+\fcone\), define the \emph{diamond} and \emph{diamond envelop}:
\begin{equation}
	\label{eq:def:diamond_envelop}
	\diam(u,v) = (u+\fcone)\cap (v+\bcone),
	\quad
	\DiaEnv(x_0,x_1,\dots,x_m) \coloneqq \bigcup_{k=1}^{m}\diam(x_{k-1},x_k).
\end{equation}
The part of the boundary of \(\DiaEnv(x_0,x_1,\dots,x_m)\) staying above (resp. below) the linear interpolation between \(x_0,x_1,\dots,x_m\) is called the upper (resp. lower) boundary of the envelop and denoted by \(\partial_+\DiaEnv(x_0,x_1,\dots,x_m)\) (resp. \(\partial_-\DiaEnv(x_0,x_1,\dots,x_m)\)).

\subsection{Definitions and preliminaries on the percolation measure}
\label{subsec:DoubleBridge:Def_Prelim}

We start by introducing the notations and objects that we will use throughout this section. Our goal will be to investigate the joint geometry of the clusters of \(v_L\) in \(\omega_{\tau}\) and of \(v_L'\) in \(\omega_{\tau\tau'}^*\), with \((\omega_{\tau},\omega_{\tau\tau'})\sim \Phi_n(\, \cdot \given v_L\xleftrightarrow{\omega_\tau} v_R,\, v'_L\xleftrightarrow{\omega_{\tau\tau'}^*} v'_R)\).

\subsubsection*{Crossing clusters and paths}

 Introduce the following objects and notations. The general idea is that objects with a \('\) refer to the crossing from \(v_L'\) to \(v_R'\).
\begin{itemize}
	\item \(\calC\) is the cluster of \(v_L\) in \(\omega_{\tau}\), and is \(\calC'\) the cluster of \(v_L'\) in \(\omega_{\tau\tau'}^*\). 
	\item \(\Gamma_+ \equiv \Gamma_+(\calC) \subset \calC\) the top most path connecting \(v_L\) to \(v_R\) in \(\omega_{\tau}\) when it exists, and \(\varnothing\) otherwise.
	\item \(\Gamma_-\equiv \Gamma_-(\calC')\subset \calC'\) the bottom most path connecting \(v_L'\) to \(v_R'\) in \(\omega_{\tau\tau'}^*\) when it exists and \(\varnothing\) otherwise.
	\item For a simple path \(\gamma\) formed of edges in \(\FVEdges_n\) from \(v_L\) to \(v_R\), denote \({\gamma_\shortuparrow}\), \({\gamma_\shortdownarrow}\) the edges in $\FVEdges_n\setminus \gamma$ that are above and below $\gamma$, respectively.
	\item For a simple path \(\gamma'\) formed of edges in \(*\FVEdges_n\) from \(v_L'\) to \(v_R'\), denote \({\gamma_\shortuparrow'}\), \({\gamma_\shortdownarrow'}\) the edges in $*\FVEdges_n\setminus \gamma'$ that are above and below $\gamma'$, respectively.
\end{itemize}

\medskip

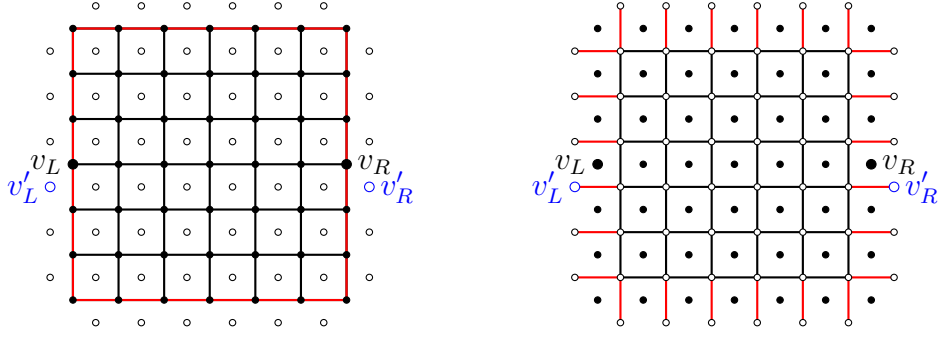
\begin{figure}
	\begin{tikzpicture}[scale=0.6]
		\foreach \j in {-3,...,3}{
			\draw[thick, black] (-3,\j)--(3,\j);
			\draw[thick, black] (\j,3)--(\j,-3);
		}
		
		\draw[red, thick] (-3,-3)--(3,-3);
		\draw[red, thick] (-3,-3)--(-3,3);
		\draw[red, thick] (3,3)--(3,-3);
		\draw[red, thick] (3,3)--(-3,3);
		
		\foreach \i in {-3,...,3}{
			\foreach \j in {-3,...,3}{
				\filldraw[fill=black] (\i,\j) circle(2pt);
			}
		}
		\foreach \i in {-4,...,3}{
			\foreach \j in {-3,...,2}{
				\filldraw[fill=white] ({\i+0.5},{\j+0.5}) circle(2pt);
			}
		}
		\foreach \i in {-3,...,2}{
			\foreach \j in {-4,3}{
				\filldraw[fill=white] ({\i+0.5},{\j+0.5}) circle(2pt);
			}
		}
		\filldraw[fill=black] (-3,0) circle(3pt) node[left]{$v_L$};
		\filldraw[fill=black] (3,0) circle(3pt) node[right]{$v_R$};
		\filldraw[fill=white, draw=blue] (-3.5,-0.5) circle(3pt) node[blue, left]{$v_L'$};
		\filldraw[fill=white, draw=blue] (3.5,-0.5) circle(3pt) node[blue, right]{$v_R'$};
	\end{tikzpicture}
	\hspace{1cm}
	\begin{tikzpicture}[scale=0.6]
		\foreach \j in {-3,...,2}{
			\draw[thick, black] (-2.5,{\j+0.5})--(2.5,{\j+0.5});
			\draw[thick, black] ({\j+0.5},2.5)--({\j+0.5},-2.5);
		}
		
		\foreach \j in {-3,...,2}{
			\draw[thick, red] (-3.5,{\j+0.5})--(-2.5,{\j+0.5});
			\draw[thick, red] (2.5,{\j+0.5})--(3.5,{\j+0.5});
			\draw[thick, red] ({\j+0.5},-3.5)--({\j+0.5}, -2.5);
			\draw[thick, red] ({\j+0.5},3.5)--({\j+0.5}, 2.5);
		}
		
		\foreach \i in {-3,...,3}{
			\foreach \j in {-3,...,3}{
				\filldraw[fill=black] (\i,\j) circle(2pt);
			}
		}
		\foreach \i in {-4,...,3}{
			\foreach \j in {-3,...,2}{
				\filldraw[fill=white] ({\i+0.5},{\j+0.5}) circle(2pt);
			}
		}
		\foreach \i in {-3,...,2}{
			\foreach \j in {-4,3}{
				\filldraw[fill=white] ({\i+0.5},{\j+0.5}) circle(2pt);
			}
		}
		\filldraw[fill=black] (-3,0) circle(3pt) node[left]{$v_L$};
		\filldraw[fill=black] (3,0) circle(3pt) node[right]{$v_R$};
		\filldraw[fill=white, draw=blue] (-3.5,-0.5) circle(3pt) node[blue, left]{$v_L'$};
		\filldraw[fill=white, draw=blue] (3.5,-0.5) circle(3pt) node[blue, right]{$v_R'$};
	\end{tikzpicture}
	\caption{\(n= m_n = 2\). Left: the edges of \(\mathsf{E}\) in black, the edges of \(\FVEdges_n^{\rmb} = \bar{\mathsf{E}}_{n,m_n}\setminus \mathsf{E}_{n,m_n}\) in red. Right: the edges of \(*\FVEdges_n^{\rmi}\) in black, the edges of \(*\FVEdges_n^{\rmb}\) in red.}
	\label{fig:domain_mATRC}
\end{figure}

\subsubsection*{Measures}

Let \(0<J<U\) satisfy \(\sinh 2J=e^{-2U}\). Recall that the measure \(\Phi_n\) is supported on configurations \((a,b), a,b\in \{0,1\}^{\FVEdges_n}\) with \(a\leq b\) and is given by
\begin{equation*}
		\Phi_n(a,b)
		\propto
		\mathds{1}_{a\subseteq b}\,\mathds{1}_{b\setminus a\subseteq \FVEdges_n}\,
		2^{|a\cap\FVEdges_n|}\,\big(\tfrac{2}{\svcb-1}\big)^{|a\cap\FVEdges_n^\rmb|}\,(\svc-2)^{|b\setminus a|}\,2^{\clusters_{\mathsf{K}}(a)+\clusters_{\mathsf{K}^1}(b)},
	\end{equation*}where~\(\clusters_{\mathsf{K}}(a)\) and~\(\clusters_{\mathsf{K}^1}(\omega_{\tau\tau'})\) are the numbers of clusters of~\(a\) and~\(b\) viewed as spanning subgraphs of~\(\mathsf{K}\) and~\(\mathsf{K}^1\), respectively; recall~$\svc, \svcb$ defined by~\eqref{eq:parameters_bulk}-\eqref{eq:parameters_bnd}.
	
The measure \(\Phi_n\) satisfies the following properties:
\begin{enumerate}
	\item \label{property:FKG} {\em FKG lattice} condition on \(\{0,1\}^{\FVEdges_n}\times\{0,1\}^{\FVEdges_n}\) where \(\FVEdges_n = \bbE_{n+1,m_n+1}\), see Figure~\ref{fig:domain_mATRC} and Lemma~\ref{lem:fkg_lattice_mod_atrc}.
	\item \label{property:FinEne} {\em Finite Energy:} there is \(\cstFinEne>0\) such that for any \(e\in \FVEdges_n\), one has that \(\Phi_n\)-a.s.:
	\begin{gather*}
		\Phi_n\big(\omega_{\tau}(e) = \omega_{\tau\tau'}(e) = 1 \given \calF_{\FVEdges_n\setminus e}\big)
		\geq
		\cstFinEne,
		\\
		\Phi_n\big(\omega_{\tau}(e) = \omega_{\tau\tau'}(e) = 0 \given \calF_{\FVEdges_n\setminus e}\big)
		\geq
		\cstFinEne,
	\end{gather*}where \(\calF_{F}\) is the sigma-algebra generated by the states of the edges in \(F\).
	\item \label{property:condATRC} {\em Conditionally ATRC:}
	\begin{gather*}
		\Phi_n\big(\cdot \given (\omega_{\tau},\omega_{\tau\tau'})|_{\mathsf{E}\setminus \bbE_{n,m_n}} = (1,1)\big) = \atrc_{ \bbE_{n,m_n}}^{1,1},
		\\
		\Phi_n\big(\cdot \given (\omega_{\tau},\omega_{\tau\tau'})|_{\mathsf{E}\setminus \bbE_{n,m_n}} = (0,0)\big) = \atrc_{ \bbE_{n,m_n}}^{0,0}.
	\end{gather*}
\end{enumerate}
In particular, \(\Phi_n\) satisfies the Hypotheses (I) to (V) of~\cite[Section 8]{DobGlaOtt25} and we can apply~\cite[Theorem 8.1]{DobGlaOtt25} to obtain a fine control over single crossings.

\begin{corollary}[Finite volume OZ for a single crossing]
	\label{cor:fin_vol_cluster_bridge_coupling}
	Let \(0<J<U\) satisfy \(\sinh 2J=e^{-2U}\). There are \(c_0,c>0\) such that for any sequence \(\frac{\sqrt{n}}{c_0}\geq k_n\geq c_0\ln(n)\), there exists \(n_0\geq 1\) such that for \(n\geq n_0\), \(m_n \geq c_0n\), there exists a probability measure \(\bndMeas_n\) on \(\SetRootMarkBackCont \times \SetRootMarkForwCont\) satisfying
	\begin{itemize}
		\item \(\bndMeas_n\) is supported on pairs \((\eta_L,\eta_R)\) with \(\norm{\displace(\eta_L)}_{\infty},\norm{\displace(\eta_R)}_{\infty} \leq k_n^2\),
		\item defining \(v(\eta_L,\eta_R) = v_R-v_L-\displace(\eta_L)-\displace(\eta_R)\), for any \(f\) function of \(\calC\),
		\begin{multline*}
			\Big|\sum_{m\geq 0}\sum_{\eta_L,\eta_R} \bndMeas_n(\eta_L,\eta_R) \OZ_{v(\eta_L,\eta_R)}\big[f(v_L + \eta_L\concatenate\bar{\eta}_1^M \concatenate \eta_R)\big]
			\\
			-\Phi_{n}\big[f(\calC)\bgiven v_R\in \calC\big]\Big|
			\leq
			\norm{f}_{\infty}e^{-ck_n},
		\end{multline*}where \(\eta_1^M \sim \OZ_{v(\eta_L,\eta_R)}\).
	\end{itemize}
	The same result holds for \(\calC', v_L',v_R'\).
\end{corollary}

Note that cone-points of \(\calC\) are cone-points of every path from \(v_L\) to \(v_R\). Moreover, for a realisation \(C\ni v_L,v_R\) of \(\calC\), and \(v,w\in \CPts(C)\), one has that \(\Gamma_+(C)\cap \diam(v,w)\) is the highest path in \(\calC\cap \diam(v,w)\) linking \(v\) to \(w\). Similar considerations hold for \(\Gamma_-,\calC'\). Letting \(\tilde{p}_{\pm}\) be the pushforward of \(p\) by the mapping assigning to a diamond-confined graph the highest/lowest path linking its endpoints, and setting
\begin{equation*}
	\tilde{\OZ}^{\pm} = (\tilde{p}_{\pm})^{\otimes \N},
	\quad
	\tilde{\OZ}^{\pm}_v = \tilde{\OZ}^{\pm}( (\gamma_1,\dots,\gamma_{T_v})\in \cdot \given T_v <\infty),
\end{equation*}one has the following consequence of Corollary~\ref{cor:fin_vol_cluster_bridge_coupling}.
\begin{corollary}[Finite volume OZ for a single cluster boundary]
	\label{cor:fin_vol_paths_bridge_coupling}
	Let \(0<J<U\) satisfy \(\sinh 2J=e^{-2U}\). There exist \(c_0,c>0\) such that, for any sequence \(\frac{\sqrt{n}}{c_0}\geq k_n\geq c_0\ln(n)\), there exists\(n_0\geq 1\) such that for any \(n\geq n_0\), \(m_n \geq c_0n\), there exists a probability measure \(\tilde{\bndMeas}_n\) on \(\SetRootMarkBackCont \times \SetRootMarkForwCont\) satisfying
	\begin{itemize}
		\item \(\tilde{\bndMeas}_n\) is supported on pairs of simple paths \((\eta_L,\eta_R)\) with \(\norm{\displace(\eta_L)}_{\infty},\norm{\displace(\eta_R)}_{\infty} \leq k_n^2\),
		\item defining \(v(\eta_L,\eta_R) = v_R-v_L-\displace(\eta_L)-\displace(\eta_R)\), for any \(f\) function of \(\Gamma_+\),
		\begin{multline*}
			\Big|\sum_{m\geq 0}\sum_{\eta_L,\eta_R} \tilde{\bndMeas}_n(\eta_L,\eta_R) \tilde{\OZ}_{v(\eta_L,\eta_R)}^+\big[f(v_L + \eta_L\concatenate\bar{\eta}_1^M \concatenate \eta_R)\big]
			\\
			-\Phi_{n}\big[f(\Gamma_+)\bgiven v_R\in \calC\big]\Big|
			\leq
			\norm{f}_{\infty}e^{-ck_n},
		\end{multline*}where \(\eta_1^M \sim \tilde{\OZ}_{v(\eta_L,\eta_R)}^+\).
	\end{itemize}
	The same result holds for \(\calC',\Gamma_-, v_L',v_R'\).
\end{corollary}

Recall $\gamma_{\shortdownarrow}$ and $\gamma_{\shortuparrow}$ defined in Corollary~\ref{cor:repulsiveness}.
For realizations \(\gamma,\gamma'\) of \(\Gamma_+,\Gamma_-\), define
\begin{equation*}
	\Phi_n^{\gamma} \equiv \Phi_n(\cdot \given \Gamma_+= \gamma)|_{\gamma_{\shortdownarrow}},
	\quad
	\Phi_n^{\gamma'} \equiv \Phi_n(\cdot \given \Gamma_-= \gamma')|_{\gamma'_{\shortuparrow}},
\end{equation*}where \(\gamma_{\shortdownarrow}, \gamma'_{\shortuparrow}\) are the edges of \(\FVEdges_n\) below \(\gamma\), above \(\gamma'\) respectively. Note that there is no ambiguity in the notation: \(\gamma\) is a path on primal edges, and \(\gamma'\) on dual edges.

From the results of Section~\ref{sec:fkg-repulsion}, one has that for any (admissible) path \(\gamma\) from \(v_L\to v_R\), and dual path \(\gamma'\) from \(v_L'\) to \(v_R'\),
\begin{equation}
	\label{eq:pathwise_stoch_dom}
	\Phi_n^{\gamma} \succcurlyeq \Phi_{n},
	\qquad
	\Phi_n^{\gamma'} \preccurlyeq \Phi_{n},
\end{equation}seen as, respectively, the law of \((\omega_{\tau},\omega_{\tau\tau'})|_{\gamma_{\shortdownarrow}}\), and \((\omega_{\tau},\omega_{\tau\tau'})|_{\gamma_{\shortuparrow}'}\).

\subsubsection*{Mixing}

Finally, we gather the relaxation properties of \(\Phi_n\) and of various conditional versions of it that we shall need later.

\begin{lemma}[Mixing]
	\label{lem:DoubleBridge:relax_mATRC}
	Let \(0<J<U\) satisfy \(\sinh 2J=e^{-2U}\). There exist \(c,C\in (0,+\infty)\) such that for any \(n,m_n\geq 1\), \(F\subset F'\subset \bbE_{n-1,m_n-1}\), and any events \(A\) supported on edges in \(F\), and \(B\) supported on the edges of \((F')^c\),
	\begin{equation*}
		\Big|\frac{\Phi_n(A\given B)}{\Phi(A)} -1 \Big|\leq C \sum_{e\in F} e^{-c\rmd(e,(F')^c)},
	\end{equation*}as soon as the R.H.S. is less than \(0.1\).
	In particular, for any dual simple path \(\gamma^*\) from \(v_L'\) to \(v_R'\) using only edges in \(*\FVEdges_n\), any \(F\subset\gamma'_{\shortuparrow}\cap \bbE_{n-1,m_n-1}\), and any event \(A\) supported on edges in \(F\),
	\begin{equation*}
		\Big|\frac{\Phi_n^{\gamma'}(A)}{\Phi(A)} -1 \Big|\leq C \sum_{e\in F} e^{-c\rmd(e,(\bbE_{n,m_n}\setminus (\gamma'\cup \gamma'_{\shortdownarrow}))^c)},
	\end{equation*}as soon as the R.H.S. is less than \(0.1\). The equivalent claim holds for \(\Phi_n^{\gamma}\) with a path \(\gamma\) from \(v_L\) to \(v_R\).
\end{lemma}
\begin{proof}
	The proof is identical to the one of~\cite[Lemma 8.3]{DobGlaOtt25}: the two measures are ``stochastically sandwiched'' between two ATRC measures with extremal boundary conditions on \(F'\), and the ATRC model has exponential ratio mixing.
\end{proof}

\subsection{Main result: coupling of a pair of clusters with a pair of non-intersecting random walks bridges}
\label{subsec:DoubleBridge:coupling_avoiding_bridges}

We can now state the main result of the section, Theorem~\ref{thm:coupling_with_avoiding_bridges}. The remainder of the section is devoted to its proof. The statement is a total variation bound between the law of \(\calC,\calC'\) under \(\Phi_n(\ \given v_R\in \calC,\, v_R'\in \calC')\), and a pair of \emph{independent} chains of microscopic clusters, as encountered in Theorem~\ref{thm:OZ_atrc_inf_vol}, conditioned on non-intersection. In particular, see Corollary~\ref{cor:coupling_with_avoiding_bridges}, it implies that there exists a coupling between \(\calC,\calC'\) and a pair of random walk bridges conditioned on non-intersection such that the following holds: the geometry of \(\calC,\calC'\) is entirely controlled by the geometry of the walks with probability \(1-o_n(1)\). 
This is Step 2 in the proof of Theorem~\ref{thm:invariance_princ_FK}, see Section~\ref{subsec:proof-thms}.

The ideas behind the proof of Theorem~\ref{thm:coupling_with_avoiding_bridges} are the following. First, the lattice FKG property implies that \(\calC,\calC'\) are not ``attracted'' to one another, that is them being close brings no energetic gain. Thus, they experience \emph{entropic repulsion}: for entropic reasons, it is better for them to be far apart one another. Second, when they are far apart, mixing implies that they locally look like two independent infinite volume long clusters, which geometry is perfectly understood through Theorem~\ref{thm:OZ_atrc_inf_vol}. 
The rigorous implementation of these ideas is strongly inspired by the treatment of a similar problem in~\cite{IofOttVelWac20}: the main conceptual difference is that we are dealing with two clusters instead of a single one and a static object; and the main technical difference is that we are dealing with the ATRC model, which is less Markovian than FK percolation, so the arguments are more involved. 

As it is the main scale to be considered in most of this section, we introduce the shorthand
\begin{equation*}
	\scale_n = \lceil\ln^2(n) \rceil.
\end{equation*}
In the next Theorem, \(111\) and \(22\) are just constants that work. They are by no mean optimal: we anyway lose powers of log everywhere in the proof, so we did not try to optimize much.

\begin{theorem}[Total variation bound between two clusters and two bridges]
	\label{thm:coupling_with_avoiding_bridges}
	Let \(0<J<U\) satisfy \(\sinh 2J=e^{-2U}\). There exist \(n_0\geq 1, c_0\geq 0, c>0\) such that for any \(n\geq n_0\), \(m_n\geq c_0n\), there exists a probability measure \(\MixMeas_n\) on \((\SetRootMarkBackCont\times \SetRootMarkForwCont)^2\) such that the following holds.
	\begin{itemize}
		\item \(\MixMeas_n\) is supported on quadruplets \((\eta_L,\eta_R,\eta_L',\eta_R')\) with
		\begin{equation*}
			\max(\norm{\displace(\eta_L)}_{\infty}, \norm{\displace(\eta_L')}_{\infty}, \norm{\displace(\eta_R)}_{\infty}, \norm{\displace(\eta_R')}_{\infty} ) \leq \ln^{111}(n),
		\end{equation*}and with \(\displace(\eta_L) \cdot \rme_2 - \displace(\eta_L')\cdot \rme_2\geq \scale_n\), \(\displace(\eta_L')\cdot \rme_2 - \displace(\eta_R) \cdot \rme_2 \geq \scale_n\).
		\item Let \((\OZRVchain^v)_{v\in \Z^2}, (\widetilde{\OZRVchain^v})_{v\in \Z^2}\) be independent and identically distributed with law \(\otimes_{v\in \Z^2} \OZ_v\).
		Denote by \((\OZRVcluster^v)_{v\in \Z^2},(\widetilde{\OZRVcluster^v})_{v\in \Z^2}\) the associated clusters.
		Sample \linebreak \((\Upsilon_L,\Upsilon_L',\Upsilon_R,\Upsilon_R')\) from~\(\MixMeas_n\) independently from \((\OZRVchain^v)_{v\in \Z^2}\) and \((\widetilde{\OZRVchain^v})_{v\in \Z^2}\).
		Introduce random variables \(V := v_R-v_L-\displace(\Upsilon_L)-\displace(\Upsilon_R)\) and \(V' := v_R'-v_L'-\displace(\Upsilon_L')-\displace(\Upsilon_R')\).
		Then, for any function \(f\) of \(\calC,\calC'\),
		\begin{multline*}
			\Big| E\Big[ f(v_L+ \Upsilon_L\concatenate \OZRVcluster^{V} \concatenate \Upsilon_R, v_L' + \Upsilon_L'\concatenate \widetilde{\OZRVcluster}^{V'} \concatenate \Upsilon_R') \bgiven \mathrm{Ord} \Big]
			\\
			- \Phi_n\big[ f(\calC,\calC') \bgiven v_R\in \calC,\, v_R'\in \calC'\big]\Big|
			\leq
			\frac{1}{\ln^{22}(n)}\norm{f}_{\infty}
		\end{multline*}where \(\mathrm{Ord}\) is the event
		\begin{equation*}
			\mathrm{Ord} = \big\{ \rmd_{\infty}\big(v_L+\displace(\Upsilon_L) + \OZRVcluster^{V}, v_L'+\displace(\Upsilon_L') + \widetilde{\OZRVcluster}^{V'}\big)\geq \scale_n \big\}.
		\end{equation*}
	\end{itemize}
\end{theorem}

As a corollary of Theorem~\ref{thm:coupling_with_avoiding_bridges}, there exists a coupling of \((\calC,\calC')\) under \(\Phi_n(\cdot \given v_R\in \calC,\, v_{R}'\in \calC')\) with a pair of random walks conditioned on some form of mutual avoidance.
\begin{corollary}[Coupling two clusters with two bridges.]
	\label{cor:coupling_with_avoiding_bridges}
	There exist \(n_0\geq 1, c\in (0,+\infty)\), such that for any \(n\geq n_0\), there are random variables \(T,S_0, S_1,\dots, S_{T},V,\Upsilon\) and \(T', S_0',S_1',\dots,S_{T'}',V',\Upsilon'\) defined on a common space satisfying
	\begin{enumerate}
		\item \((\Upsilon,\Upsilon')\) has the law of \((\calC,\calC')\) under \(\Phi_n(\cdot \given v_R\in \calC,\, v_R'\in \calC')\);
		\item the start and end are ordered, \(S_0\cdot \rme_2 - S_0'\cdot \rme_2 \geq \scale_n \), \(V\cdot \rme_2- V' \cdot \rme_2 \geq \scale_n \), and close to \(v_L,v_R,v_L',v_R'\): with probability \(1\),
		\begin{equation*}
			\max(\norm{v_L-S_0}_{\infty},\norm{v_L'-S_0'}_{\infty},\norm{v_R-V}_{\infty},\norm{v_R'-V'}_{\infty}) \leq \ln^{111}(n);
		\end{equation*}
		\item conditionally on \((S_0,S_0',V,V')\), \((S_0,\dots,S_{T})\) and \((S_0',\dots, S_{T'}')\) have law of two independent bridges, \(\OZwalk^{S_0,V}\otimes \OZwalk^{S_0',V'}\), conditioned on non-intersection of their diamond-envelop, defined in~\eqref{eq:def:diamond_envelop}: \(\DiaEnv(S_0,\dots,S_T)\cap \DiaEnv(S_{0}',\dots,S_{T'}') = \varnothing\);
		\item with probability \(1-o_n(1)\), \(\Upsilon\subset (S_0+\bcone)\cup (\DiaEnv(S_0,\dots,S_T))\cup (V+\fcone)\), and \(\Upsilon'\subset (S_0'+\bcone)\cup (\DiaEnv(S_0',\dots,S_{T'}'))\cup (V'+\fcone)\).
	\end{enumerate}
\end{corollary}

We now focus on proving Theorem~\ref{thm:coupling_with_avoiding_bridges}.

\subsection{Results on system of directed Random Walks}
\label{subsec:DoubleBridge:RW_system}

We will need some results about random walks conditioned on mutual avoidance for the proof of the invariance principle and of entropic repulsion. We gather them here in the form we will use them. There are two kinds of results that are needed: up to constants (or even, up-to-log) bounds, and sharp asymptotics for the probability of certain events. Only up to constants/log bounds are needed for the proof of Theorem~\ref{thm:coupling_with_avoiding_bridges}.
Whilst they are mostly small variations of known arguments, we try to be relatively self-contained for these estimates. Indeed, the reduction from the study of the order-order Potts interface to a random walk problem is the main contribution of this work, the invariance principle for a pair of walks conditioned on non-intersection having already been studied in~\cite{DAl24}, based on results of~\cite{DurWac20}. We still allow ourselves to only sketch certain classical arguments or to refer to similar computations in other works to keep the proofs short.

\subsubsection*{Pair of directed random walks}

For this section, let \(X_1,X_2,\dots, Y_1,Y_2,\dots \) be family of i.i.d. random variables with law \(p\circ \displace^{-1}\). From Theorem~\ref{thm:OZ_atrc_inf_vol}, they satisfy:
\begin{itemize}
	\item \(X_1\) is supported on \(\fcone \cap \Z^2\setminus\{0\}\), and has positive density on that set;
	\item \(X_1\) has exponential tails;
	\item \(p\) inherits lattice symmetries from \(\atrc\), so \(X_1\cdot \rme_2\) is symmetric around \(0\) conditionally on \(X_1\cdot \rme_1\).
\end{itemize}In particular, our increments fit in the setup of~\cite[Section 5]{DAl24}. Let \(S_0,S_0'\) be independent standard Gaussian random variables on \(\R^2\). Let \(S_k = S_0 + \sum_{i=1}^k X_i\), \(S_{k}'= S_0' + \sum_{i=1}^k Y_i\) be the associated random walks. For \(x\in \R^2\), define
\begin{equation}
	T_x = \inf\Big\{k\geq 1:\ \sum_{i=1}^k X_i = x\Big\},
	\quad
	T_x' = \inf\Big\{k\geq 1:\ \sum_{i=1}^k Y_i = x\Big\}
\end{equation}where \(\inf \varnothing = +\infty\) by convention. In particular, $T_x = T'_x =\infty$ a.s. for all~$x\not \in \Z^2$.
We will use the shorthands
\begin{equation*}
	X_{i,1} \equiv X_{i}\cdot \rme_1,\quad X_{i,2} \equiv X_{i}\cdot \rme_2,\quad i\geq 1,
\end{equation*}and similarly for the \(Y_i\)'s, \(S_i\)'s, and \(S_i'\)'s.

For \(x,x',y,y'\in \R^2\), introduce the events
\begin{equation}
	\label{eq:def:hit_events_pair_of_walks}
	\HitEvent_{y,y'} = \{T_{y-S_0}<\infty\}\cap \{T_{y'-S_0'}' <\infty\},
\end{equation}and the probability measures
\begin{equation*}
	P_{x,x'} \coloneqq P(\cdot \given S_0=x,\, S_0'=x'),
	\quad
	P_{x,x'}^{y,y'} \coloneqq P_{x,x'}(\cdot \given \HitEvent_{y,y'}).
\end{equation*}

\subsubsection*{Synchronized pair of directed random walks}

We recall the setup of \emph{walk synchronization}, and refer to~\cite{DAl24,OttVel19} for more details. Define the \emph{synchronized times}: say that \((k,k')\) is a \emph{synchronized times pair} if
\begin{equation*}
	S_{k}\cdot \rme_1 = S_{k'}'\cdot \rme_1.
\end{equation*}
Since the walks are directed, these pairs of times are strictly ordered. Denote by \(\{(\syncTime_0,\syncTime_0'), (\syncTime_1,\syncTime_1'),\dots \}\) the sequence of synchronized times, and
\begin{equation}
	\label{eq:def:synchro_walks}
	\big((\mathbf{T}_0,\mathbf{S}_0, \mathbf{S}_0'),(\mathbf{T}_1,\mathbf{S}_1, \mathbf{S}_1'),\dots \big)
	\coloneqq
	\big((S_{\syncTime_0,1},S_{\syncTime_0,2},S_{\syncTime_0',2}'),(S_{\syncTime_1,1},S_{\syncTime_1,2},S_{\syncTime_1',2}'),\dots\big)
\end{equation}the walks seen at synchronized times.
They satisfy the following properties, which we will use throughout the text.
\begin{enumerate}
	\item The increments \((\mathbf{T}_{i}-\mathbf{T}_{i-1},\mathbf{S}_i-\mathbf{S}_{i-1}, \mathbf{S}_i'-\mathbf{S}_{i-1}')\), \(i\geq 1\), are i.i.d. under \(P_{x,x'}\) for any \(x,x'\in \Z^2\).
	\item \(|\mathbf{S}_{i+1} - \mathbf{S}_i|,|\mathbf{S}_{i+1}' - \mathbf{S}_i'| \leq |\mathbf{T}_{i+1}- \mathbf{T}_i|\) as the \(X_i,Y_i\)'s are supported on \(\fcone\).
	\item For any \(j>i\), \(\mathbf{T}_{j}- \mathbf{T}_i\geq j-i\) a.s., as \(X_{1}\cdot \rme_1\geq 1\) a.s..
	\item For \(x,x'\in \Z^2\), conditionally on \(S_0=x,S_0'=x'\), \(\mathbf{T}_0-\max(x_1,x_1')\) has exponential tails. This follows from the exponential tails of \(X_i\), and the fact that \(P(X_1\cdot \rme_1 = 1) >0\) (irreducibility suffices). See for example~\cite[Section 2.3.2]{OttVel19}.
	\item In the same fashion, the ``time'' increments \(\mathbf{T}_{i}-\mathbf{T}_{i-1}\) have exponential tails. By the cone constraint, this yields exponential tails on the other coordinates as well.
\end{enumerate}
We also stress that \textbf{\(\mathbf{S},\mathbf{S}'\) are \(\Z\)-valued walks, and \(S,S'\) are \(\Z^2\)-valued}. Forgetting this can easily lead to notational confusion.

Define
\begin{equation*}
	\IntersectionTime \coloneqq \inf\big\{k\geq 0:\ \mathbf{S}_k \leq \mathbf{S}_k'\big\},
\end{equation*}the first synchronized time at which \(S'\) goes above \(S\).

For \(i,i',j,j'\in \Z\), \(n\geq 0\), define
\begin{equation}
\label{eq:def:sync_hit_times}
\begin{aligned}
	\syncHitTime_{n,j,j'} &= \inf\big\{k\geq 0:\ \mathbf{T}_k=n,\, \mathbf{S}_k = j,\, \mathbf{S}_k'=j' \big\},
	\\
	\syncHitEvent_{n,j,j'} &= \big\{\syncHitTime_{n,j,j'} <\infty\big\},
\end{aligned}
\end{equation}where \(\inf \varnothing = +\infty\) by convention.
Define also the measures
\begin{equation}
	\label{eq:def:sync_measures}
	Q_{i,i'} = P\big( \cdot \bgiven \mathbf{T}_0 = 0,\, \mathbf{S}_0 = i,\, \mathbf{S}_0' = i' \big),
\end{equation}
and the transition kernels
\begin{equation}
\label{eq:def:sync_trans_kernels}
\begin{gathered}
	q_{i,i'}(n,j,j') = Q_{i,i'}\big( \syncHitEvent_{n,j,j'} \big),
	\quad
	q_{i,i'}^+(n) = Q_{i,i'}\big( \mathbf{T}_{\IntersectionTime} > n \big)
	\\
	q_{i,i'}^+(n,j,j') = Q_{i,i'}\big( \syncHitEvent_{n,j,j'},\, \IntersectionTime > \syncHitTime_{n,j,j'} \big).
\end{gathered}
\end{equation}
Define \(\overleftarrow{Q}_{i,i'}\), \(\overleftarrow{q}_{i,i'}(n,j,j')\), \(\overleftarrow{q}_{i,i'}^+(n,j,j')\) to be the same quantities for the time reversed walk, which is the walk with time increments \(\mathbf{T}_1-\mathbf{T}_0\), and space increments \((\mathbf{S}_{0}-\mathbf{S}_{1}, \mathbf{S}_0'-\mathbf{S}_1')\).

The point of synchronized random walks, is that \(\big((\mathbf{S}_0,\mathbf{S}_0'),(\mathbf{S}_1,\mathbf{S}_1'),\dots\big)\) conditioned on \(\IntersectionTime = +\infty\) is a \(\Z^2\)-valued random walk in the cone \(\{(x_1,x_2)\in \R^2:\ x_1> x_2\}\), to which one can apply the analysis of~\cite{DenWac15}. If we were only looking at events measurable in terms of the synchronized walk, we would be in the setup of~\cite{DenWac15}. We therefore need to do some work to reduce estimates for the pair of directed walks we are interested in to estimates on events depending only on the synchronized walk. This kind of reduction is performed for example in~\cite{IofOttVelWac20}, and used extensively in~\cite{DAl24}.

\subsubsection*{Estimates for synchronized walks}

\begin{lemma}
	\label{lem:DoubleBridge:sync_trans_kernels:free_walk}
	1) There exist \(C,c\in (0,+\infty)\) such that for any \(n\geq 1\), \(i,i',j,j'\in \Z\),
	\begin{equation*}
		q_{i,i'}(n,j,j')
		\leq
		\tfrac{C}{n}e^{-c\frac{(i-j)^2 +(i'-j')^2}{n}}.
	\end{equation*}
	2) Let \(K\geq 1\). There exists \(c\in (0,+\infty)\) such that for any \(n\geq 1\), \(i,i',j,j'\in \Z\), with \(|i-j|\leq K\sqrt{n}\), and \(|i'-j'|\leq K\sqrt{n}\),
	\begin{equation*}
		q_{i,i'}(n,j,j')
		\geq
		\tfrac{1}{cn}.
	\end{equation*}
	Both statements also hold for \(\overleftarrow{q}\).
\end{lemma}
\begin{proof}
	The lower bound follows from \cite{AouOttVel24}, see Theorem~\ref{thm:OZ_asymp_from_LLT} in the appendix for details.
	The upper bound when both \(|i-j|,|i'-j'|\leq n^{7/12}\) follows again from the same statement. Here, \(\frac{7}{12}\) is any number in \((\frac{1}{2},\frac{2}{3})\).
	
	The upper bound for either \(|i-j|\geq n^{\frac{7}{12}}\) or \(|i'-j'|\geq n^{\frac{7}{12}}\) follows from Doob's submartingale inequality. Indeed, by symmetry we can assume that \(|i-j|\geq |i'-j'|\). Then, as \(\mathbf{T}_n- \mathbf{T}_0 \geq n\) a.s.
	\begin{equation*}
		q_{i,i'}(n,j,j')
		\leq
		Q_{0,0}\big( \max_{k=1,\dots,n} |\mathbf{S}_{k}| \geq |i-j|\big)
		\leq
		2Q_{0,0}\big( \max_{k=1,\dots,n} \mathbf{S}_{k} \geq |i-j|\big)
	\end{equation*}where we used a union bound and symmetry between \(\mathbf{S}\) and \(-\mathbf{S}\) in the last line.
	Now, for \(\lambda>0\), \(e^{\lambda \mathbf{S}_{k}}, k\geq 0\) is a submartingale. So, by Doob's inequality, for any \(R>0\),
	\begin{multline*}
		Q_{0,0}\big( \max_{k=1,\dots,n} \mathbf{S}_{k} \geq R\big)
		=
		Q_{0,0}\big( \max_{k=1,\dots,n} e^{\lambda\mathbf{S}_{k}} \geq e^{\lambda R}\big)
		\\
		\leq
		e^{-\lambda R}E_{Q_{0,0}}[e^{\lambda\mathbf{S}_{n}}]
		=
		e^{-\lambda R}E_{Q_{0,0}}[e^{\lambda\mathbf{S}_{1}}]^n.
	\end{multline*}
	We then have that, as \(\mathbf{S}_{1}\) has exponential tails under \(Q_{0,0}\), for \(\lambda\) small enough
	\begin{equation*}
		\ln\big(E_{Q_{0,0}}[e^{\lambda \mathbf{S}_{1}}]\big)
		\leq
		\tfrac{\lambda^2}{2}\mathrm{Var}_{Q_{0,0}}(\mathbf{S}_{1}) + C\lambda^3
		\leq
		c'\lambda^2,
	\end{equation*}with \(C,c'\in (0,+\infty)\).
	Thus, for \(0\leq R \leq n\), taking \(\lambda = \frac{\delta R}{n}\) for \(\delta >0\) small, one gets
	\begin{equation*}
		Q_{0,0}\big( \max_{k=1,\dots,n} \mathbf{S}_{k} \geq R\big)
		\leq
		\exp(-\tfrac{\delta R^2}{n} + c' \tfrac{\delta^2R^2}{n}  )
		\leq
		e^{-c'' \frac{R^2}{n}}.
	\end{equation*}
	Thus, when \(n^{7/12}\leq |i-j|\leq n\),
	\begin{equation*}
		q_{i,i'}(n,j,j')
		\leq
		2e^{-c'' \frac{|i-j|^2}{n}},
	\end{equation*}
	which implies the claimed bound.
	When \(|i-j|> n\), \(q_{i,i'}(n,j,j') = 0\) by the cone confinement property.
\end{proof}

\begin{lemma}
	\label{lem:DoubleBridge:sync_trans_kernels:asymp_ordering}
	There is \(c\in (0,+\infty)\) such that for any \(n\geq 1\), \(i,i'\in \Z\) with \(i> i'\),
	\begin{equation*}
		\tfrac{\min(i-i',\sqrt{n})}{c\sqrt{n}}
		\leq
		q_{i,i'}^+(n)
		\leq
		\tfrac{c\min(i-i',\sqrt{n})}{\sqrt{n}}.
	\end{equation*}
	The same holds for \(\overleftarrow{q}^+\).
\end{lemma}
\begin{proof}
	Note that the difference walk \(\mathbf{S}- \mathbf{S}'\) is a random walk on \(\Z\) with symmetric increments having exponential tails. Moreover, \(\IntersectionTime\) is the hitting time of \((-\infty,0]\) for this walk. It is well known that, see for example~\cite[Lemmata 3.6 and 3.8]{OttVel25}\footnote{These bounds are folklore and can already be deduced from the material presented in~\cite[Chapter XII]{Feller91}, but it is not so easy to find them stated this way in the literature, as most texts are focused on deriving a version of \(\lim_{k\to\infty} \frac{Q_{i,i'}( \IntersectionTime > k )}{Q_{i,i-1}( \IntersectionTime > k )}\), and sharp asymptotics for the denominator in the previous limit.}, that there is \(c\in (0,+\infty)\) such that for every \(k\geq 1\),
	\begin{equation*}
		\tfrac{\min(i-i',\sqrt{k})}{c\sqrt{k}}
		\leq
		Q_{i,i'}( \IntersectionTime > k )
		\leq
		\tfrac{c\min(i-i',\sqrt{k})}{\sqrt{k}}.
	\end{equation*}
	Now, for integers \(k> l\), \(\mathbf{T}_{k}-\mathbf{T}_{l}\geq k-l\) a.s.. So,
	\begin{equation*}
		q_{i,i'}^+(n) \geq Q_{i,i'}( \IntersectionTime > n ),
	\end{equation*}which gives the wanted lower bound. Then, there are \(a,c'>0\) such that \(Q_{i,i'}\big( \exists k\leq an:\ \mathbf{T}_k \geq n \big) \leq e^{-c'n}\), so
	\begin{equation*}
		q_{i,i'}^+(n) \leq e^{-c' n} + Q_{i,i'}( \IntersectionTime > an ),
	\end{equation*}which gives the wanted upper bound. The proof for \(\overleftarrow{q}^+\) is the same.
\end{proof}

\begin{lemma}
	\label{lem:DoubleBridge:sync_trans_kernels:ord_walkUB}
	There are \(C\in (0,+\infty)\) such that for any \(n\geq 1\), \(i,i',j,j'\in \Z\) with \(i> i'\), \(j>j'\),
	\begin{equation*}
		q_{i,i'}^+(n,j,j')
		\leq
		\tfrac{C\min(i-i',\sqrt{n})\min(j-j',\sqrt{n})}{n^2}.
	\end{equation*}
	The same holds for \(\overleftarrow{q}^+\).
\end{lemma}
\begin{proof}
	As \(q_{i,i'}^+(n,j,j')\leq 1\), it is sufficient to consider \(n\) large which we implicitly do every time it is needed. We consider \(q^+\), the proof for \(\overleftarrow{q}^+\) is identical.
	We start with some considerations, and then divide the proof into several cases. For notational convenience, suppose \(n\) is a multiple of \(5\), and let \(L= \frac{n}{5}\). Then, define
	\begin{equation*}
		\tau_- = \min\{k\geq 0:\ \mathbf{T}_k \geq L\},
		\quad
		\tau_+ = \max\{k\geq 0:\ \mathbf{T}_k \leq 4L\}.
	\end{equation*}By exponential tails, the probability that \(\tau_-\geq 2L\) or that \(\tau_+\leq 3L\) is exponentially small in \(n\). Thus, using Markov's property,
	\begin{multline*}
		q_{i,i'}^+(n,i,i')
		=
		\mathrm{err}_n
		+
		\sum_{k=L}^{2L} \sum_{k'=3L}^{4L} \sum_{u<u'}\sum_{v<v'}
		\\
		Q_{i,i'}\big(\tau_-= k,\, \syncHitEvent_{k,u,u'},\, \tau_+ = k',\, \syncHitEvent_{k',v,v'},\, \mathbf{T}_{\IntersectionTime} >n,\, \syncHitEvent_{n,j,j'}  \big)
	\end{multline*}with \(\mathrm{err}_n \leq e^{-c'n}\). Now, recall \(\overleftarrow{Q}_{i,i'}\), \(\overleftarrow{\tau}_-\), and \(\overleftarrow{\syncHitEvent}_{k,i,i'}\) are defined as \(Q_{i,i'}\), \(\tau_-\) and~\(\syncHitEvent_{k,i,i'}\) but for the time reversed walk, so the Markov property gives
	\begin{multline*}
		Q_{i,i'}\big(\tau_-= k,\, \syncHitEvent_{k,u,u'},\, \tau_+ = k',\, \syncHitEvent_{k',v,v'},\, \mathbf{T}_{\IntersectionTime} >n,\, \syncHitEvent_{n,j,j'} \big)
		=
		\\
		Q_{i,i'}\big(\mathbf{T}_{\IntersectionTime} > \tau_-= k,\, \syncHitEvent_{k,u,u'} \big)
		q_{u,u'}^+(k'-k,v,v')
		\overleftarrow{Q}_{j,j'}\big(\mathbf{T}_{\IntersectionTime} > \overleftarrow{\tau}_-= n-k',\, \overleftarrow{\syncHitEvent}_{n-k',v,v'}\big).
	\end{multline*}
	As \(k'-k \geq L=n/5\),
	using the above considerations and the upper bound \(q_{u,u'}^+(k'-k,v,v')\leq q_{u,u'}(k'-k,v,v')\leq \frac{c}{n}\) from Lemma~\ref{lem:DoubleBridge:sync_trans_kernels:free_walk}, one gets
	\begin{align*}
		q_{i,i'}^+(n,j,j')
		&\leq
		e^{-c'n}
		+
		\tfrac{c}{n} Q_{i,i'}\big(\mathbf{T}_{\IntersectionTime} > \tau_-\in [L,2L] \big)
		\overleftarrow{Q}_{j,j'}\big(\mathbf{T}_{\IntersectionTime} > \overleftarrow{\tau}_-\in [L,2L]\big)
		\\
		&\leq
		e^{-c'n}
		+
		\tfrac{c}{n} q_{i,i'}^+(L)
		\overleftarrow{q}_{j,j'}^+(L),
	\end{align*}which gives the wanted bound using Lemma~\ref{lem:DoubleBridge:sync_trans_kernels:asymp_ordering}.
\end{proof}

\begin{lemma}
	\label{lem:DoubleBridge:sync_trans_kernels:ord_bridgeLB}
	There exist \(c>0, n_0\geq 1\) such that for \(n\geq n_0\) and \(|i|,|i'|,|j|,|j'|\leq \sqrt{n}\) with \(i<i'\), \(j<j'\),
	\begin{equation*}
		q^+_{i,i'}(n,j,j')
		\geq
		\tfrac{c(i-i')(j-j')}{n^2}.
	\end{equation*}
	The same holds for \(\overleftarrow{q}^+\).
\end{lemma}
\begin{proof}
	First, we have the following bounds. Let \(K>0\) be large enough, fixed. Then, there is \(c>0\) such that for any \(n\) large enough,
	\begin{itemize}
		\item for any \(i,i',j,j'\in \Z\) with \(|i|,|i'|,|j|,|j'|\leq 2K\sqrt{n}\), and \(i-i'\geq \sqrt{n}/K\), \(j-j'\geq \sqrt{n}/K \),
		\begin{equation}
			\label{eq:prf_of_lem:DoubleBridge:sync_trans_kernels:ord_bridgeLB:eq1}
			q_{i,i'}^+(n,j,j') \geq cq_{i,i'}(n,j,j'),
		\end{equation}
		\item for any \(\sqrt{n}\geq i-i'>0\),
		\begin{equation}
			\label{eq:prf_of_lem:DoubleBridge:sync_trans_kernels:ord_bridgeLB:eq2}
			Q_{i,i'}\big(\exists k\geq 1:\ \mathbf{T}_k =n,\ \tfrac{\sqrt{n}}{K}\leq \textbf{S}_k-\textbf{S}_k' \leq K\sqrt{n},\, |\textbf{S}_k- i|\leq K\sqrt{n}\bgiven \textbf{T}_{\calT} > n\big) \geq c.
		\end{equation}
	\end{itemize}
	The same holds for \(\overleftarrow{Q}\). We do not prove these statements, as they are direct adaptations of well known arguments for random walks on \(\Z\):~\eqref{eq:prf_of_lem:DoubleBridge:sync_trans_kernels:ord_bridgeLB:eq1} follows from a small ball estimate for the walk bridge, which can be proved using either the invariance principle or martingales methods; \eqref{eq:prf_of_lem:DoubleBridge:sync_trans_kernels:ord_bridgeLB:eq2} can be proved using a combination of moments bounds obtained by sub-martingale inequality (for the \(\leq K\sqrt{n}\) part), and the CLT (for the \(\geq \sqrt{n}/K\) part). For details see, for example,~\cite[proof of Lemma 6.1]{OttVel25}.
	
	We then combine these two inequalities with Lemma~\ref{lem:DoubleBridge:sync_trans_kernels:asymp_ordering} to obtain that for \(n\) large enough and \(i,i',j,j'\) as in the statement (assuming \(n\) is a multiple of \(3\) for notational convenience)
	\begin{align*}
		q^+_{i,i'}(n,j,j')
		&\geq
		\sum_{\substack{K\sqrt{n} \geq u-u'\geq \sqrt{n}/K\\ |u-i|\leq K\sqrt{n}}} \sum_{\substack{K\sqrt{n} \geq v-v'\geq \sqrt{n}/K\\ |v-j|\leq K\sqrt{n}}} q^+_{i,i'}(\tfrac{n}{3},u,u')q^+_{u,u'}(\tfrac{n}{3},v,v')q^+_{v,v'}(\tfrac{n}{3},j,j')
		\\
		&\geq
		\frac{c}{n}\sum_{\substack{K\sqrt{n} \geq u-u'\geq \sqrt{n}/K\\ |u-i|\leq K\sqrt{n}}} \sum_{\substack{K\sqrt{n} \geq v-v'\geq \sqrt{n}/K\\ |v-j|\leq K\sqrt{n}}} q^+_{i,i'}(\tfrac{n}{3},u,u')q^+_{v,v'}(\tfrac{n}{3},j,j')
		\\
		&\geq
		\frac{c}{n} q^+_{i,i'}(\tfrac{n}{3})\overleftarrow{q}^+_{j,j'}(\tfrac{n}{3})
		\geq
		\frac{c(i-i')(j-j')}{n^2},
	\end{align*}where we used~\eqref{eq:prf_of_lem:DoubleBridge:sync_trans_kernels:ord_bridgeLB:eq1} and Lemma~\ref{lem:DoubleBridge:sync_trans_kernels:free_walk} in the second line,~\eqref{eq:prf_of_lem:DoubleBridge:sync_trans_kernels:ord_bridgeLB:eq2} in the third line, and Lemma~\ref{lem:DoubleBridge:sync_trans_kernels:asymp_ordering} in the last inequality.
\end{proof}

\subsubsection*{Up-to-constants estimates: directed walks pair}

\begin{lemma}
	\label{lem:DoubleBridge:RW:single_OZ_asymp}
	There exist \(n_0\geq 0, c\in (0,+\infty)\) such that, with the notations above, for any \(x\in \Z^2\cap \fcone\) with \(x_1\geq \min(n_0, |x_2|^{2})\),
	\begin{equation*}
		\tfrac{1}{c \sqrt{x_1}}
		\leq
		P(T_x <\infty \given S_0 =0)
		\leq
		\tfrac{c}{ \sqrt{x_1}}.
	\end{equation*}
\end{lemma}
\begin{proof}
	This is a direct consequence of Theorem~\ref{thm:OZ_asymp_from_LLT}.
\end{proof}

In the next Lemma, \(10\) is an arbitrarily fixed constant that must be greater than \(2\). Recall that \(\calS_{a,b}= [a,b]\times \R\), see~\eqref{eq:def:slab_CP}.
\begin{lemma}
	\label{lem:DoubleBridge:RW:LB_diam_avoid}
	Let \(K>0\). There exist \(n_0\geq 0, K_0\geq 0, c\in (0,+\infty)\) such that, with the notations above, for any \(n\geq n_0\), \(x,x'\in \R^2\), \(y\in x+\Z^2\), \(y'\in x'+\Z^2\) with \(y_1-x_1= n\), \(|x_2|,|x_2'|,|y_2|,|y_2'|\leq \sqrt{n}\), \(x_2-x_2' \geq K_0\), \(y_2-y_2' \geq K_0\), \(|x_1-x_1'|\leq 10\), \(|y_1-y_1'|\leq 10\)
	\begin{equation*}
		P_{x,x'}\Big( \HitEvent_{y,y'},\, \bigcap_{x_1\leq k\leq y_1-1}\big\{\rmd(\DiaEnv\cap \slab_{k,k+1},  \DiaEnv'\cap \slab_{k,k+1}) > K+ \rmd(k,\{x_1,y_1\})^{1/6}\big\}\Big)
		\geq
		\frac{c}{n^2},
	\end{equation*}where \(\DiaEnv = \DiaEnv(S_0,\dots, S_{T_y})\), \(\DiaEnv'=\DiaEnv(S_0',\dots, S_{T_{y'}'}')\), and \(\rmd(\varnothing, A) = +\infty\) by convention.
\end{lemma}

\begin{proof}
	First, up to a shift of \(z\in \R^2\) with \(\norm{z}_{\infty}\leq 1\), one can assume that \(x,x'\in \Z^2\) at the cost af changing the value of \(K_0\).
	Then, using that the \(X_i,Y_i\)'s have full support in \(\fcone\cap \Z^2\), we can lower bound the wanted probability by a positive constant times
	\begin{equation*}
		P\big( \HitEvent_{\tilde{y},\tilde{y}'},\, \cap_{\tilde{x}_1\leq k\leq \tilde{y}_1-1}\{\dots \} \bgiven S_0 = \tilde{x},\, S_0' = \tilde{x}'\big),
	\end{equation*}where \(\tilde{x},\tilde{x}',\tilde{y},\tilde{y}'\) are at distance at most \(10\) from their non-tilde versions, and \(\tilde{x}_1 = \tilde{x}_1'\), \(\tilde{y}_1 = \tilde{y}_1'\).
	Now, one can look at the synchronized walk instead: \(\DiaEnv, \DiaEnv'\) are included in the diamond envelops of the synchronized walks, denoted \(\tilde{\DiaEnv}, \tilde{\DiaEnv}'\), so it is sufficient to lower bound
	\begin{equation*}
		P\Big( \syncHitEvent_{\tilde{n},\tilde{y}_2,\tilde{y}'_2},\, \cap_{0\leq k\leq \tilde{n}-1} A_k \Bgiven \mathbf{T}_0 = 0,\, \mathbf{S}_0 = \tilde{x}_2,\, \mathbf{S}_0' = \tilde{x}'_2\Big)
	\end{equation*}where \(n-20 \leq \tilde{n} \leq n+20\), we translated everything by \((-\tilde{x}_1,0)\), and we introduced
	\begin{equation*}
		A_k = \big\{\rmd(\tilde{\DiaEnv}\cap \slab_{k,k+1},  \tilde{\DiaEnv}'\cap \slab_{k,k+1}) > \min(k,\tilde{n}-k)^{1/6} + K\big\}.
	\end{equation*}
	Now, by the cone property of the increments, note that the union of the \(A_k^c\)'s is contained in the event \(\cup_{k=0}^{\tilde{n}-1} B_k\), see Figure~\ref{fig:diamond_far_to_walk_far}, with
	\begin{equation*}
		B_k = \big\{\exists t\geq 0:\ \mathbf{T}_t =k,\, \mathbf{S}_t - \mathbf{S}_{t}' < (\min(k, \tilde{n}-k) )^{1/6} + K + 3(\mathbf{T}_{t+1}-\mathbf{T}_{t})\big\}.
	\end{equation*}
	\begin{figure}
		\includegraphics[scale=0.7]{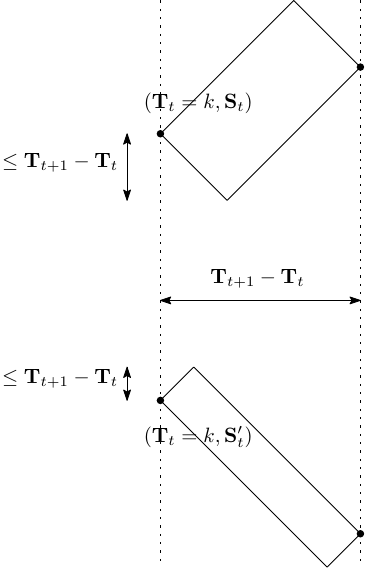}
		\caption{.}
		\label{fig:diamond_far_to_walk_far}
	\end{figure}
	
	Thus, by inclusion of events, a union bound, and Lemma~\ref{lem:DoubleBridge:sync_trans_kernels:ord_bridgeLB}
	\begin{align}
	\label{eq:prf_of_lem:DoubleBridge:RW:LB_diam_avoid:eq1}
		\nonumber &Q_{\tilde{x}_2,\tilde{x}_2'}\big( \syncHitEvent_{\tilde{n},\tilde{y}_2,\tilde{y}'_2},\, \cap_{0\leq k\leq \tilde{n}-1} A_k \big)
		\geq
		q_{\tilde{x}_2,\tilde{x}_2'}^+(\tilde{n},\tilde{y}_2,\tilde{y}'_2 ) - \sum_{k=0}^{\tilde{n}-1} Q_{\tilde{x}_2,\tilde{x}_2'}\big( \syncHitEvent_{\tilde{n},\tilde{y}_2,\tilde{y}'_2} ,\, B_k,\, \mathbf{T}_{\calT}> \tilde{n} \big)
		\\
		&\quad\geq
		\frac{C(\tilde{x}_2-\tilde{x}_2')(\tilde{y}_2-\tilde{y}_2')}{n^2} - \sum_{k=0}^{\tilde{n}-1} Q_{\tilde{x}_2,\tilde{x}_2'}\big( \syncHitEvent_{\tilde{n},\tilde{y}_2,\tilde{y}'_2} ,\, B_k,\, \mathbf{T}_{\calT}> \tilde{n} \big),
	\end{align}with \(C>0\) uniform over \(|\tilde{x}_2|,|\tilde{x}_2'|,|\tilde{y}_2|,|\tilde{y}_2'|\leq \sqrt{n}\).
	
	Let \(k_0\geq K^6\) large, and \(\epsilon>0\) small (\(\frac{1}{100}\) works). We then split the sum over \(k\) as sums over \(k\leq \tilde{n}/2\) and sum over \(k\geq \tilde{n}/2\).
	The latter sum is bounded similarly as the fist using the time reversed walk, so we only bound the first sum. We will further split it into a sum over \(k\leq k_0\) and \(k>k_0\). Partitioning over the values of \(t\) such that \(\mathbf{T}_t=k\), and over the values of \(\mathbf{S}_t'\), \(\mathbf{S}_t\), and \(\mathbf{T}_{t+1}\), we get
	\begin{align}
		\label{eq:prf_of_lem:DoubleBridge:RW:LB_diam_avoid:eq2}
		\nonumber Q_{\tilde{x}_2,\tilde{x}_2'}\big( &\syncHitEvent_{\tilde{n},\tilde{y}_2,\tilde{y}'_2} ,\, B_k^c,\, \mathbf{T}_{\calT}> \tilde{n} \big)
		=
		\sum_{t\geq 0} \sum_{i>i'}\sum_{l\geq 1}\mathds{1}_{k^{1/6} + K + 3l > i-i'} \cdot
		\\
		&
		Q_{\tilde{x}_2,\tilde{x}_2'}\big( \textbf{T}_t = k,\, \textbf{S}_t =i,\, \textbf{S}_t'=i',\, \textbf{T}_{t+1} = k+l,\, \syncHitEvent_{\tilde{n},\tilde{y}_2,\tilde{y}'_2} ,\, \mathbf{T}_{\calT}> \tilde{n} \big).
	\end{align}By Markov's property, this last expression is equal to
	\begin{equation*}
		\sum_{t\geq 0} \sum_{l\geq 1}\sum_{k^{1/6} + K + 3l>i-i'>0}
		q_{\tilde{x}_2,\tilde{x}_2'}^+(k,i,i')
		Q_{i,i'}\big( \textbf{T}_1 = l,\, \syncHitEvent_{\tilde{n}-k,\tilde{y}_2,\tilde{y}'_2} ,\, \mathbf{T}_{\calT}> \tilde{n}-k \big)
	\end{equation*}which, using the fact that \(\textbf{T}_{1}\) has exponential tails under \(Q_{i,i'}\) and the cone constraint, is less than
	\begin{align}
		\label{eq:prf_of_lem:DoubleBridge:RW:LB_diam_avoid:eq3}
		\nonumber &\sum_{l\geq 1}\sum_{k^{1/6} + K + 3l>i-i'>0}
		q_{\tilde{x}_2,\tilde{x}_2'}^+(k,i,i') \sum_{\substack{|j-i|,|j'-i'|\leq l\\ j> j'}}
		e^{-cl} q_{j,j'}^+( \tilde{n}-k-l,\tilde{y}_2,\tilde{y}'_2 )
		\\
		&\leq
		e^{-cn} + \sum_{l=1}^{\epsilon n} \sum_{k^{1/6} + K + 3l>i-i'>0}
		q_{\tilde{x}_2,\tilde{x}_2'}^+(k,i,i') \sum_{\substack{|j-i|,|j'-i'|\leq l\\ j> j'}}
		e^{-cl} q_{j,j'}^+( \tilde{n}-k-l,\tilde{y}_2,\tilde{y}'_2 )
		,
	\end{align}where \(c>0\). Now, by the cone constraint and the fact that \(\tilde{x}_2-\tilde{x}_2'\geq K_0\), \(q_{\tilde{x}_2,\tilde{x}_2'}^+(k,i,i') = 0\) whenever \(i-i' < K_0 - 2k\). This will fix our choice of \(k_0\) and the value of \(K_0\). Taking \(k_0 = \frac{K_0}{4}\) and \(K_0 \geq 8K\) large enough multiple of \(4\), one has that for \(k\leq k_0\),~\eqref{eq:prf_of_lem:DoubleBridge:RW:LB_diam_avoid:eq3} is less than
	\begin{align}
		\label{eq:prf_of_lem:DoubleBridge:RW:LB_diam_avoid:eq4}
		\nonumber &e^{-cn} + \tfrac{C(\tilde{y}_2-\tilde{y}'_2)}{n^2}\sum_{l=1}^{\epsilon n}e^{-cl}\sum_{\frac{K_0}{4} + 3l>i-i'\geq \frac{K_0}{2}}
		q_{\tilde{x}_2,\tilde{x}_2'}^+(k,i,i') \sum_{\substack{|j-i|,|j'-i'|\leq l\\ j> j'}} (j-j')
		\\
		\nonumber &\leq
		e^{-cn} + \tfrac{C(\tilde{y}_2-\tilde{y}'_2)}{n^2}\sum_{l= \frac{K_0}{12}}^{\epsilon n}l^2e^{-cl} \sum_{\substack{\frac{K_0}{4} + 3l>i-i'\geq \frac{K_0}{2}\\ |i-\tilde{x}_2|\leq k}} q_{\tilde{x}_2,\tilde{x}_2'}^+(k,i,i')(\tfrac{K_0}{4}+5l)
		\\
		\nonumber &\leq
		e^{-cn} + \tfrac{C(\tilde{y}_2-\tilde{y}'_2)(\tilde{x}_2-\tilde{x}'_2)}{n^2(k+1)^2} \sum_{l= \frac{K_0}{12}}^{\epsilon n}l^2e^{-cl} (\tfrac{K_0}{4}+5l)\sum_{\substack{\frac{K_0}{4} + 3l>i-i'\geq \frac{K_0}{2}\\ |i-\tilde{x}_2|\leq k}} (i-i')
		\\
		&\leq
		e^{-cn} + \tfrac{C(\tilde{y}_2-\tilde{y}'_2)(\tilde{x}_2-\tilde{x}'_2)}{n^2(k+1)} \sum_{l= \frac{K_0}{12}}^{\epsilon n}l^5 e^{-cl}
		\leq
		e^{-cn} + \tfrac{C(\tilde{y}_2-\tilde{y}'_2)(\tilde{x}_2-\tilde{x}'_2)}{n^2} e^{-cK_0}
		,
	\end{align}where \(C,c>0\) are independent of \(K_0,K\) (and change from line to line), and we used Lemma~\ref{lem:DoubleBridge:sync_trans_kernels:ord_walkUB}, and Lemma~\ref{lem:DoubleBridge:sync_trans_kernels:asymp_ordering}. Thus, from~\eqref{eq:prf_of_lem:DoubleBridge:RW:LB_diam_avoid:eq2},~\eqref{eq:prf_of_lem:DoubleBridge:RW:LB_diam_avoid:eq3},~\eqref{eq:prf_of_lem:DoubleBridge:RW:LB_diam_avoid:eq4}, one gets
	\begin{equation}
		\label{eq:prf_of_lem:DoubleBridge:RW:LB_diam_avoid:eq5}
		\sum_{k=0}^{k_0} Q_{\tilde{x}_2,\tilde{x}_2'}\big( \syncHitEvent_{\tilde{n},\tilde{y}_2,\tilde{y}'_2} ,\, B_k,\, \mathbf{T}_{\calT}> \tilde{n} \big)
		\leq
		CK_0e^{-cn} + \tfrac{C(\tilde{y}_2-\tilde{y}'_2)(\tilde{x}_2-\tilde{x}'_2)}{n^2} e^{-cK_0}.
	\end{equation}We then need to bound the sum in~\eqref{eq:prf_of_lem:DoubleBridge:RW:LB_diam_avoid:eq1} over \(k_0<k\leq \tilde{n}/2\). We can always assume \(K_0\) large enough, in particular, we can assume that \(k_0 = \frac{K_0}{4} \geq K^4\). There, proceeding similarly than in~\eqref{eq:prf_of_lem:DoubleBridge:RW:LB_diam_avoid:eq4},~\eqref{eq:prf_of_lem:DoubleBridge:RW:LB_diam_avoid:eq3} is less than
	\begin{equation}
		\label{eq:prf_of_lem:DoubleBridge:RW:LB_diam_avoid:eq6}
		e^{-cn} + \tfrac{C (\tilde{y}_2-\tilde{y}'_2) }{n^2}\sum_{l=1}^{\epsilon n} \sum_{2k^{1/6} + 3l>i-i'>0}
		q_{\tilde{x}_2,\tilde{x}_2'}^+(k,i,i')
		l^2 e^{-cl} (2k^{1/6} + 5l).
	\end{equation}Now, split the sum over \(l\) into \(l\leq k^{1/6}\) and \(l>k^{1/6}\). On the latter, we have the upper bound (assuming \(K_0\) large enough)
	\begin{multline}
		\label{eq:prf_of_lem:DoubleBridge:RW:LB_diam_avoid:eq7_largel}
		\sum_{l=k^{1/6}}^{\epsilon n} \sum_{2k^{1/6} + 3l>i-i'>0}
		q_{\tilde{x}_2,\tilde{x}_2'}^+(k,i,i')
		l^2 e^{-cl} (2k^{1/6} + 5l)
		\\
		\leq
		q_{\tilde{x}_2,\tilde{x}_2'}^+(k) \sum_{l=k^{1/6}}^{\epsilon n}
		l^2 e^{-cl} (2k^{1/6} + 5l)
		\leq
		e^{-ck^{1/6}}.
	\end{multline}
	For the sum over \(l\leq k^{1/6}\), we have the upper bound
	\begin{equation}
		\label{eq:prf_of_lem:DoubleBridge:RW:LB_diam_avoid:eq7_smalll}
		7k^{1/6} \sum_{l=1}^{k^{1/6}} \sum_{5k^{1/6}>i-i'>0}
		q_{\tilde{x}_2,\tilde{x}_2'}^+(k,i,i')
		l^2 e^{-cl}
		\leq
		C k^{1/6} \sum_{5k^{1/6}>i-i'>0}
		q_{\tilde{x}_2,\tilde{x}_2'}^+(k,i,i').
	\end{equation}Splitting the sum over \(i\) into sums over \(|i-\tilde{x}_2|\leq k^{7/12}\) and over \(|i-\tilde{x}_2|> k^{7/12}\), we have that~\eqref{eq:prf_of_lem:DoubleBridge:RW:LB_diam_avoid:eq7_smalll} is upper bounded by
	\begin{equation*}
		C k^{1/6}e^{-ck^{1/6}} + C k^{1/6} \sum_{\substack{5k^{1/6}>i-i'>0\\ |i-\tilde{x}_2|\leq k^{7/12}}}
		\tfrac{(i-i')(\tilde{x}_2-\tilde{x}_2')}{k^2}
		\leq
		C k^{1/6}e^{-ck^{1/6}} + \tfrac{C(\tilde{x}_2-\tilde{x}_2')}{k^{13/12}}
	\end{equation*}as the probability under \(Q_{\tilde{x}_2,\tilde{x}_2'}\) that there is \(t\geq 0\) such that \(\textbf{T}_t = k\) and \(|\textbf{S}_k-\tilde{x}_2| > k^{7/12}\) is less than \(e^{-ck^{1/6}}\) by large deviation estimates, and where we used Lemma~\ref{lem:DoubleBridge:sync_trans_kernels:ord_walkUB}. Using this bound on~\eqref{eq:prf_of_lem:DoubleBridge:RW:LB_diam_avoid:eq7_smalll}, combining the resulting inequality with~\eqref{eq:prf_of_lem:DoubleBridge:RW:LB_diam_avoid:eq7_largel} to bound~\eqref{eq:prf_of_lem:DoubleBridge:RW:LB_diam_avoid:eq6}, and thus~\eqref{eq:prf_of_lem:DoubleBridge:RW:LB_diam_avoid:eq3}, and~\eqref{eq:prf_of_lem:DoubleBridge:RW:LB_diam_avoid:eq2}, we get that for \(k_0<k\leq \tilde{n}\),
	\begin{align}
		\label{eq:prf_of_lem:DoubleBridge:RW:LB_diam_avoid:eq8}
		\nonumber Q_{\tilde{x}_2,\tilde{x}_2'}\big(\syncHitEvent_{\tilde{n},\tilde{y}_2,\tilde{y}'_2} ,\, B_k^c,\, \mathbf{T}_{\calT}> \tilde{n} \big)
		&\leq
		e^{-cn} + \tfrac{C (\tilde{y}_2-\tilde{y}'_2) }{n^2} \Big(C k^{1/6}e^{-ck^{1/6}} + \tfrac{C(\tilde{x}_2-\tilde{x}_2')}{k^{-13/12}} + e^{-ck^{1/6}}\Big)
		\\
		&\leq
		e^{-cn} + \tfrac{C (\tilde{y}_2-\tilde{y}'_2)(\tilde{x}_2-\tilde{x}_2') }{n^2 k^{13/12}}
	\end{align}where the last bound holds for \(K_0\) (and thus \(k_0\)) large enough. Thus, combining with~\eqref{eq:prf_of_lem:DoubleBridge:RW:LB_diam_avoid:eq5}, and mimicking the bounds for \(k>\tilde{n}/2\), we get
	\begin{align*}
		&\sum_{k=0}^{\tilde{n}-1} Q_{\tilde{x}_2,\tilde{x}_2'}\big( \syncHitEvent_{\tilde{n},\tilde{y}_2,\tilde{y}'_2} ,\, B_k,\, \mathbf{T}_{\calT}> \tilde{n} \big)
		\\
		&\qquad \leq
		2(CK_0+n)e^{-cn} +  \tfrac{C (\tilde{y}_2-\tilde{y}'_2)(\tilde{x}_2-\tilde{x}_2') }{n^2}\Big(e^{-cK_0} + \sum_{k\geq k_0}\frac{1}{k^{13/12}} \Big)
		\\
		&\qquad \leq
		e^{-cn} + \tfrac{C (\tilde{y}_2-\tilde{y}'_2)(\tilde{x}_2-\tilde{x}_2') }{n^2}\big(e^{-cK_0} +\tfrac{C}{K_0^{1/12}}\big),
	\end{align*}for \(n\) large. Using this in~\eqref{eq:prf_of_lem:DoubleBridge:RW:LB_diam_avoid:eq1} for \(K_0\) large enough gives the wanted bound.
\end{proof}

Recall \(\scale_n = \lceil \ln^2(n) \rceil\). The last Lemma will be used in the proof of Lemma~\ref{lem:DoubleBridge:entropic_repulsion} which proves a priori entropic repulsion between \(\calC\) and \(\calC'\).
\begin{lemma}
	\label{lem:DoubleBridge:RW:UB_avoid_but_close}
	There exist \(n_0\geq 1, c>0\) such that for any \(n\geq n_0\), \(x,x'\in \R^2\), \(y\in x+\Z^2\), \(y'\in x'+\Z^2\) with \(\norm{v_L-x} \leq 2\scale_n^2\), \(\norm{v_L'-x'} \leq 2\scale_n^2\), \(\norm{v_R-y} \leq 2\scale_n^2\), and \(\norm{v_R'-y'} \leq 2\scale_n^2\),
	\begin{equation*}
		P_{x,x'}^{y,y'}\Big(\big\{\forall k\ S_k \text{ is above } \partial_-\DiaEnv' \big\}
		\cap \big\{\exists k \text{ s.t. } S_k\in \slab_{-a_n,a_n},\ \rmd(S_k,\DiaEnv')\leq \scale_n \big\} \Big)
		\leq
		\frac{c}{\scale_n^{16} n},
	\end{equation*}where \(\DiaEnv\equiv \DiaEnv(S_0,\dots,S_{T_{y}})\), \(\DiaEnv'\equiv \DiaEnv(S_0',\dots,S_{T_{y'}'}')\), and \(a_n = n-\scale_n^{50}\).
\end{lemma}
\begin{proof}
	We always assume that \(n\) is large enough. First, observe that we can shift \(x,x'\) so that they belong to \(\Z^2\) if we consider instead \(\rmd(S_k,\DiaEnv')\leq 2\scale_n\) in the second part of the event. Now, for \(x,x'\in \Z^2\) the synchronized walk is well-defined, and the event \(A'=\big\{\forall k\ S_k \text{ is above } \partial_-\DiaEnv' \big\}\) is included in the event
	\begin{equation*}
		A = \big\{\forall k \ \mathbf{S}_k\geq \mathbf{S}_k'\big\}.
	\end{equation*}
	Now, introduce the last synchronized time of the bridges:
	\begin{equation*}
		\tau = \max\{k\geq 0:\ \mathbf{T}_{k}\leq \min(y_1,y_1')\}.
	\end{equation*}
	By the exponential tails of the synchronized walk increments, up to an error of order \(e^{-c\scale_n}\), we can assume that
	\begin{itemize}
		\item the maximal step of the synchronized walk is of sup-norm \(\scale_n\),
		\item \(|\mathbf{T}_{0}+n|\leq 3\scale_n^2\), and \(|\mathbf{T}_{\tau}-n|\leq 3\scale_n^2\).
	\end{itemize}
	Denote this event \(\mathrm{SmallSteps}\). Under \(\mathrm{SmallSteps}\), the event \(B'=\big\{\exists k \text{ s.t. } S_k\in \slab_{-a_n,a_n},\ \rmd(S_k,\DiaEnv')\leq \scale_n \big\}\) is included in the event
	\begin{equation*}
		B = \big\{ \exists k\geq 0:\  \mathbf{T}_k\in [-a_n,a_n] \text{ and } \mathbf{S}_k \leq \mathbf{S}_k' + 3\scale_n\big\}.
	\end{equation*}
	So,
	\begin{equation*}
		P_{x,x'}^{y,y'}(A'\cap B') \leq e^{-c\scale_n} + P_{x,x'}^{y,y'}(A\cap B\cap \mathrm{SmallSteps}).
	\end{equation*}
	Note that by our conditions on \(x,x',y,y'\), and the cone confinement property, under the event \(\mathrm{SmallSteps}\), we have that the values of \(\mathbf{S}_0,\mathbf{S}_0',\mathbf{S}_{\tau},\mathbf{S}_{\tau}'\) a.s. satisfy \(3\scale_n^2 \geq \mathbf{S}_0 > \mathbf{S}_0' \geq - 3\scale_n^2\), and \(3\scale_n^2 \geq \mathbf{S}_{\tau} > \mathbf{S}_{\tau}' \geq - 3\scale_n^2\).
	Let then \(u,u',v,v'\in \Z\) be such that \(3\scale_n^2 \geq u \geq u' \geq - 3\scale_n^2\), and \(3\scale_n^2 \geq v \geq v' \geq - 3\scale_n^2\), and \(\tilde{n}\in [2n-6\scale_n^2,2n]\). The claim will follow if we can show that for \(n\) large enough,
	\begin{equation}
		\label{eq:prf_lem:DoubleBridge:RW:UB_avoid_but_close:eq1}
		Q_{u,u'}\Big( \big\{ \IntersectionTime > \syncHitTime_{\tilde{n},v,v'}\big\} \cap \bigcup_{k=\scale_n^{49}}^{\tilde{n}-\scale_n^{49}} R_k \Bgiven \syncHitEvent_{\tilde{n},v,v'}\Big)
		\leq
		\frac{c}{\scale_n^{16} n},
	\end{equation}with \(c\) uniform over \(u,u',v,v',\tilde{n}\) as above, and
	\begin{equation*}
		R_k = \big\{\exists l\geq 0:\ \mathbf{T}_{l} = k,\, \mathbf{S}_l\leq \mathbf{S}_l'+\tfrac{1}{2}\scale_n\big\}.
	\end{equation*}
	By a union bound and the Markov property, the probability in~\eqref{eq:prf_lem:DoubleBridge:RW:UB_avoid_but_close:eq1} is less or equal to
	\begin{equation}
		\label{eq:prf_lem:DoubleBridge:RW:UB_avoid_but_close:eq2}
		\frac{1}{q_{u,u'}(\tilde{n},v,v')} \sum_{k= \scale_n^{49}}^{\tilde{n}-\scale_n^{49}}\sum_{w'\in \Z}\sum_{w= w'+1}^{w'+\scale_n} q_{u,u'}^+(k,w,w')\overleftarrow{q}_{v,v'}^+(\tilde{n}-k,w,w').
	\end{equation}
	Now, Lemma~\ref{lem:DoubleBridge:sync_trans_kernels:free_walk}
	implies that uniformly over \(u,u',v,v'\) as above,
	\begin{equation}
		\label{eq:prf_lem:DoubleBridge:RW:UB_avoid_but_close:eq3}
		q_{u,u'}(\tilde{n},v,v') \geq \frac{c}{n}.
	\end{equation}
	Then, by Lemma~\ref{lem:DoubleBridge:sync_trans_kernels:ord_walkUB}, and the fact that \(q^+,\overleftarrow{q}^+,q,\overleftarrow{q}\) are non-negative kernels with total mass less than \(1\), for any \( \scale_n^{49}\leq k\leq \tilde{n}/2\),
	\begin{equation}
	\label{eq:prf_lem:DoubleBridge:RW:UB_avoid_but_close:eq4}
	\begin{aligned}
		&\sum_{w'\in \Z}\sum_{w= w'+1}^{w'+\scale_n} q_{u,u'}^+(k,w,w')\overleftarrow{q}_{v,v'}^+(\tilde{n}-k,w,w')
		\\
		&\quad\leq
		\sum_{w':|w'-u'|\geq \sqrt{k}\scale_n}\sum_{w\in \Z} q_{u,u'}^+(k,w,w')
		+
		\sum_{w':|w'-u'|< \sqrt{k}\scale_n}\sum_{w= w'+1}^{w'+\scale_n} \frac{C6\scale_n^3}{k^2} \frac{C6\scale_n^3}{(\tilde{n}-k)^2}
		\\
		&\quad\leq
		Q_{u,u'}\big(\exists l\geq 1:\ \mathbf{T}_l = k,\, |\mathbf{S}'_l-u'|\geq \sqrt{k} \scale_n\big)
		+
		\frac{C \scale_n^7}{n^2}\sum_{w':|w'-u'|\leq \sqrt{k} \scale_n} \frac{1}{k^2}
		\\
		&\quad\leq
		Q_{u,u'}\big(\exists l\geq 1:\ \mathbf{T}_l = k,\, |\mathbf{S}'_l-u'|\geq \sqrt{k} \scale_n\big)
		+
		\frac{C \scale_n^8}{n^2} \frac{1}{k^{3/2}},
	\end{aligned}
	\end{equation}where we used Lemma~\ref{lem:DoubleBridge:sync_trans_kernels:ord_walkUB} together with \(w-w'\leq \scale_n \leq \sqrt{k}\), and \(u-u',v-v'\leq 6\scale_n^2\leq \sqrt{k}\) in the second line.
	We will use this in~\eqref{eq:prf_lem:DoubleBridge:RW:UB_avoid_but_close:eq2} to bound the sum over \(\scale_n^{49}\leq k\leq \tilde{n}/2\), and use that, by direct large deviation estimates, there is \(c>0\) such that
	\begin{equation}
		\label{eq:prf_lem:DoubleBridge:RW:UB_avoid_but_close:eq5}
		\sum_{k= \scale_n^{49}}^{n} Q_{u,u'}\big(\exists l\geq 1:\ \mathbf{T}_l = k,\, |\mathbf{S}'_l-u'|\geq \sqrt{k} \scale_n\big)
		\leq
		e^{-c\scale_n}.
	\end{equation}The cases \( \tilde{n}-\scale_n^{49}\geq k\geq \tilde{n}/2\) are bounded in the same way. Thus,  using~\eqref{eq:prf_lem:DoubleBridge:RW:UB_avoid_but_close:eq4},~\eqref{eq:prf_lem:DoubleBridge:RW:UB_avoid_but_close:eq5}, and~\eqref{eq:prf_lem:DoubleBridge:RW:UB_avoid_but_close:eq3} in~\eqref{eq:prf_lem:DoubleBridge:RW:UB_avoid_but_close:eq2}, we get that the probability in~\eqref{eq:prf_lem:DoubleBridge:RW:UB_avoid_but_close:eq1} is upper bounded by
	\begin{equation*}
		e^{-c\scale_n} + \frac{C\scale_n^8}{n}\sum_{k\geq \scale_n^{49}} \frac{1}{k^{3/2}}
		\leq
		e^{-c\scale_n} + \frac{C\scale_n^8}{n} \frac{1}{\scale_n^{49/2}}
		\leq
		e^{-c\scale_n} + \frac{C}{n \scale_n^{16}}.
	\end{equation*}This gives~\eqref{eq:prf_lem:DoubleBridge:RW:UB_avoid_but_close:eq1} as wanted.
\end{proof}

\subsection{Good clusters and proof strategy for Theorem~\ref{thm:coupling_with_avoiding_bridges}.}
\label{subsec:DoubleBridge:good_clusters}

The general idea of the proof is similar to the one of~\cite{DAl24}, which is itself similar to the one of~\cite[Theorem 4.3]{IofOttVelWac20}. We first prove that the typical geometry of \(\calC,\calC'\) is ``nice'' with high probability: the clusters have many cone points, and stay far away from one another. Then, we use mixing properties of ATRC measures to couple the clusters under \(\Phi_n(\cdot \given v_R\in \calC,v_R'\in \calC')\) to two \emph{independent} infinite volume clusters conditioned on (a strong version of) non-intersection. The main novelty of the proof compared to~\cite[Section 8]{DobGlaOtt25} is that we have to deal with the interaction between \(\calC,\calC'\) and we need to prove some form of a priori entropic repulsion. We proceed by adapting the method of~\cite{IofOttVelWac20} based on monotonicity (FKG) for FK percolation to the setup of ATRC measures.

For \(n\geq 1\), define the sets of good pairs of clusters, \(\goodCl_n\). A realization \((C,C')\) of \((\calC,\calC')\) under \(\Phi_n(\cdot \given v_R\in \calC,v_L\in \calC')\) is in \(\goodCl_n\) if all of the following are satisfied. For some \(\rho>0\) depending only on \(U,J\),
\begin{itemize}
	\item For every \(i\in [\scale_n -n,\dots, n-2\scale_n]\), both \(|\slabCP_{i,i+\scale_n}(C)| \geq \rho \scale_n\), and \(|\slabCP_{i,i+\scale_n}(C')| \geq \rho \scale_n\),
	\item \(\rmd_{\infty}(C\cap \calS_{\mathrm{in}},C'\cap \calS_{\mathrm{in}}) \geq \scale_n^2\), where \(\calS_{\mathrm{in}} = \slab_{\scale_n^{50} -n,\dots, n- \scale_n^{50}}\),
	\item \(C,C'\subset \R\times [-n,n]\).
\end{itemize}

\subsection{Clusters are good with high probability}
\label{subsec:DoubleBridge:good_clusters_high_proba}

\subsubsection*{Preliminaries}

We start with an a priori lower bound on \(\Phi_n( v_R\in \calC,\, v_R'\in \calC' )\). This will be fundamental to upper bound \(\Phi_n( A \given v_R\in \calC,\, v_R'\in \calC' )\) for suitable events \(A\). This Lemma plays the same role as Lemma~\cite[Lemma 8.4]{DobGlaOtt25}.

\begin{lemma}
	\label{lem:DoubleBridge:apriori_lower_bound}
	Let \(0<J<U\) satisfy \(\sinh 2J=e^{-2U}\). Then, there exist \(c>0, n_0\geq 1\) such that for any \(n\geq n_0\),
	\begin{equation*}
		\Phi_n\big( v_R\in \calC,\, v_R'\in \calC' \big)
		\geq
		\frac{c}{n^2} e^{-4\nu_1 n}.
	\end{equation*}
\end{lemma}
\begin{proof}
	Let \(K\) be a large enough integer. Then, let \(w_L=v_L+(K,K)\), \(w_R= v_R + (-K,K)\), \(w_L'= v_L'+(K,-K)\), \(w_R' = v_R' + (-K,-K)\).
	Define the sets of clusters/dual clusters
	\begin{gather*}
		\mathrm{DiaCo} = \big\{C\ni w_L,w_R \text{ connected}:\ C \subset \diam(w_L,w_R) \big\},
		\\
		\mathrm{DiaCo}' = \big\{C'\ni w_L',w_R'\text{ connected}:\ C' \subset \diam(w_L',w_R') \big\}.
	\end{gather*}Now, by inclusion of events and finite energy, we have
	\begin{align*}
		\Phi_n\big( v_R\in \calC,\, v_R'\in \calC' \big)
		&\geq
		\Phi_n\big( v_R\xleftrightarrow{\omega_{\tau}} w_R,\, v_L\xleftrightarrow{\omega_{\tau}} w_L,\, v_R'\xleftrightarrow{\omega_{\tau\tau'}^*} w_R',\, v_L'\xleftrightarrow{\omega_{\tau\tau'}^*} w_L',
		\\
		&\qquad\qquad \calC_{w_L}(\omega_{\tau}) \in \mathrm{DiaCo},\, \calC_{w_L'}(\omega_{\tau\tau'}^*) \in \mathrm{DiaCo}' \big)
		\\
		&\geq
		(\cstFinEne)^{8K} \Phi_n\big( \calC_{w_L}(\omega_{\tau}) \in \mathrm{DiaCo},\, \calC_{w_L'}(\omega_{\tau\tau'}^*) \in \mathrm{DiaCo}' \big).
	\end{align*}
	
	Then, by Lemma~\ref{lem:DoubleBridge:relax_mATRC}, for \(K\) large enough,
	\begin{equation*}
		\frac{\Phi_n\big( \calC_{w_L}(\omega_{\tau}) \in \mathrm{DiaCo},\, \calC_{w_L'}(\omega_{\tau\tau'}^*) \in \mathrm{DiaCo}' \big)}{\Phi\big( \calC_{w_L}(\omega_{\tau}) \in \mathrm{DiaCo},\, \calC_{w_L'}(\omega_{\tau\tau'}^*) \in \mathrm{DiaCo}' \big)} \geq \frac{1}{2}.
	\end{equation*}Introduce then the set of pairs of clusters for \(K'<K\) large enough,
	\begin{multline*}
		\mathrm{SepCl}
		=
		\big\{ (C,C')\in \mathrm{DiaCo}\times \mathrm{DiaCo}':\\
		\forall x\in C,\, \rmd_{1}(x,C') \geq K' + (\min(n+K'+x_1,n-K'-x_1))^{1/6}\big\},
	\end{multline*}
	where~$\rmd_{1}$ denotes the~$L^1$ distance.
	Then, for any \((C,C')\in \mathrm{SepCl}\), the ratio mixing property of the ATRC model, see~\cite[Theorem~3.3]{DobGlaOtt25}, implies that for \(K'\) large enough,
	\begin{equation*}
		\Phi\big( \calC_{w_L}(\omega_{\tau}) = C,\, \calC_{w_L'}(\omega_{\tau\tau'}^*) = C' \big)
		\geq
		\tfrac{1}{2}\Phi\big( \calC_{w_L}(\omega_{\tau}) = C \big)\Phi\big( \calC_{w_L'}(\omega_{\tau\tau'}^*) = C' \big).
	\end{equation*}In particular, letting \(\calK \) be a random variable with law \(\Phi( \calC_{w_L}(\omega_{\tau}) \in \cdot )\), and \(\calK'\) be a random variable \emph{independent of \(\calK\)} with law \(\Phi( \calC_{w_L'}(\omega_{\tau\tau'}^*) \in \cdot )\), both defined on some probability space with measure \(P\),
	\begin{equation*}
		\Phi\big( (\calC_{w_L}(\omega_{\tau}), \calC_{w_L'}(\omega_{\tau\tau'}^*)) \in \mathrm{SepCl} \big)
		\geq
		\tfrac{1}{2} P\big((\calK,\calK') \in \mathrm{SepCl}\big).
	\end{equation*}
	Now, by Theorem~\ref{thm:OZ_atrc_inf_vol} and the fact that a cluster is included into the diamond envelop of any subset of its cone-points,
	\begin{multline*}
		e^{4\nu_1 n}P\big((\calK,\calK') \in \mathrm{SepCl}\big)
		\\
		\geq
		-e^{-cn} + cP_{w_L,w_L'}\big(\HitEvent_{w_R,w_R'},\ \cap_{k= w_L\cdot \rme_1}^{w_R\cdot \rme_1}\{\rmd(\DiaEnv\cap \slab_{k,k+1},  \DiaEnv'\cap \slab_{k,k+1}) > d_k\} \big)
	\end{multline*}where \(c>0\), \(d_k \equiv K' + \min(k-w_L\cdot \rme_1, w_R\cdot \rme_1 -k)^{1/6}\), and we used the notations of Section~\ref{subsec:DoubleBridge:RW_system}, and where \(\DiaEnv = \DiaEnv(S_0,\dots, S_{T_{w_R}})\), \(\DiaEnv'=\DiaEnv(S_0',\dots, S_{T_{w_R'}'}')\). Now, Lemma~\ref{lem:DoubleBridge:RW:LB_diam_avoid} provides the wanted bound on the last probability once \(K\) is taken large enough for any fixed \(K'\).
\end{proof}

\begin{lemma}
	\label{lem:DoubleBridge:apriori_upper_bound_single_crossing}
	Let \(0<J<U\) satisfy \(\sinh 2J=e^{-2U}\). Then, there exist \(c>0, c_0, n_0\geq 1\) such that for any \(n\geq n_0\), \(m_n\geq c_0n\), 
	\begin{equation*}
		\max\Big(\Phi_n\big( v_R\in \calC \big),\, \Phi_n\big(  v_R' \in \calC' \big) \Big)
		\leq
		\frac{c\ln^2(n)}{\sqrt{n}} e^{-2\nu_1 n}.
	\end{equation*}
\end{lemma}
\begin{proof}
	By~\cite[Lemma 8.2]{DobGlaOtt25}, there are \(C,c>0\) such that for \(n\) large enough,
	\begin{multline*}
		\Phi_n\big( v_R\in \calC \big)
		\\
		\leq
		\Phi_n\big( \exists x,y \in \{-k_n,\dots,k_n\}:\ (k_n-n,x) \xleftrightarrow{\mathsf{B}_{n-k_n,c_0n/2} } (n-k_n,y) \big) + \frac{c}{n}e^{-2\nu_1 n}
	\end{multline*}where \(k_n = \lfloor C\ln(n) \rfloor\), and \(\mathsf{B}_{a,b} = \{-a,\dots, a\}\times \{-b,\dots, b\}\), and \(\xleftrightarrow{\mathsf{B} }\) means that the connection must take place in the configuration restricted to the edges with both endpoints in \(\mathsf{B}\). Now, Lemma~\ref{lem:DoubleBridge:relax_mATRC}, a union bound, and~\cite[Theorem~3.6]{DobGlaOtt25} give the Claim. The bound for \(\calC'\) is identical.
\end{proof}

We finish the preliminaries with three key observations: recall that \(\Gamma_+\) is the top-most path in \(\calC\) going from \(v_L\) to \(v_R\), and \(\Gamma_-\) is the lowest path in \(\calC'\) going from \(v_L'\) to \(v_R'\).
\begin{enumerate}[label= \roman*), ref=\theenumii \roman*)]
	\item\label{obs:DoubleBridge:monotone_cond} First, by Corollary~\ref{cor:repulsiveness}, for any realization \(\gamma\) of \(\Gamma_+\), and \(\gamma^*\) of \(\Gamma_-\),
	\begin{equation*}
		\Phi_n^{\gamma} \succcurlyeq \Phi_n,
		\qquad
		\Phi_n^{\gamma^*} \preccurlyeq \Phi_n.
	\end{equation*}
	\item\label{obs:DoubleBridge:cones_points_are_decreasing} As was pointed out in~\cite{IofOttVelWac20}, if \(A\) is an event decreasing in \(\slabCP_{-n,n}(\calC)\), then \(\mathds{1}_A\mathds{1}_{v_R\leftrightarrow v_L}\) is an increasing function. Indeed, the function \(\mathds{1}_{v_R\leftrightarrow v_L}\) is non-zero only if there is a connection from \(v_L\) to \(v_R\), which has to cross \(\slab_{-n,n}\). Now, adding edges to any existing connection can only remove cone-points in \(\slab_{-n,n}\).
	\item\label{obs:DoubleBridge:sup_norm_bound} Finally, by Lemma~\ref{lem:DoubleBridge:apriori_upper_bound_single_crossing}, and the monotonicity noted in the first point, the function going from the image of \(\Gamma_-\) to \(\R\) given by \(\gamma^*\mapsto \Phi_n^{\gamma^*}(v_R\in \calC)\) is upper bounded by \(\frac{c\ln^2(n)}{\sqrt{n}} e^{-2\nu_1 n}\).
\end{enumerate}

\subsubsection*{Crossings contain many cone-points}

We start with the easy part of the job: showing that \(\calC,\calC'\) contain many well-distributed cone-points.
\begin{lemma}
	\label{lem:DoubleBridge:cone_points}
	Let \(0<J<U\) satisfy \(\sinh 2J=e^{-2U}\). Then, there exist \(\rho >0, c>0, c_0\geq 0, n_0\geq 1\) such that for any \(n\geq n_0\), \(m_n\geq c_0n\), and any \(n-\scale_n \leq i\leq n-2\scale_n\),
	\begin{equation*}
		\Phi_n\big( \slabCP_{i,i+\scale_n}(\calC) < \rho \scale_n,\, v_R\in \calC,\, v_R'\in \calC' \big)
		\leq
		e^{-c\scale_n-4\nu_1n}.
	\end{equation*}The same holds for \(\calC'\).
\end{lemma}
\begin{proof}
	Let \(A = \{\slabCP_{i,i+\scale_n}(\calC) < \rho \scale_n\}\). Then, \(\mathds{1}_A\mathds{1}_{v_R\in \calC}\) is an increasing function as observed in Observation~\ref{obs:DoubleBridge:cones_points_are_decreasing}. So, by Observation~\ref{obs:DoubleBridge:monotone_cond},
	\begin{multline*}
		\Phi_n\big( A,\, v_R\in \calC,\, v_R'\in \calC' \big)
		=
		\sum_{\gamma^*} \Phi_n\big(\Gamma_- = \gamma^* \big)\Phi_n^{\gamma^*}(A,\, v_R\in \calC)
		\\
		\leq
		\Phi_n(A,\, v_R\in \calC)\sum_{\gamma^*} \Phi_n\big(\Gamma_- = \gamma^* \big)
		=
		\Phi_n(A,\, v_R\in \calC)\Phi_n(v_R'\in \calC').
	\end{multline*}where the sums are over simple paths \(\gamma^*:v_L'\to v_R'\). Now, by~\cite[Lemma~8.6]{DobGlaOtt25}, \(\Phi_n(A,\, v_R\in \calC) \leq e^{-2\nu_1 n -c\scale_n}\) for some \(c>0\) as soon as \(\rho >0\) is small enough, and \(n_0\) is large enough. Also, Lemma~\ref{lem:DoubleBridge:apriori_upper_bound_single_crossing} gives that \(\Phi_n(v_R'\in \calC')\leq Ce^{-2\nu_1 n}\).
\end{proof}

\subsubsection*{Entropic repulsion}

We are now left with the difficult part of the job: showing that \(\calC,\calC'\) stay ``far away'' from one another. We follow the same general idea as~\cite{IofOttVelWac20}. For \(C,C'\) subsets of \(\R^2\), introduce the ``modified distance''
\begin{equation*}
	\mathrm{Mdist}_n(C,C') = \rmd_{\infty}\big(C\cap \slab_{\scale_n^{50}-n,n-\scale_n^{50}},\, C'\cap \slab_{\scale_n^{50}-n,n-\scale_n^{50}}\big).
\end{equation*}
Here, \(50\) is an arbitrarily fixed constant, taken large enough so that everything works out, and the \(13\) in the next lemma is the consequence of this choice of constant.
\begin{lemma}
	\label{lem:DoubleBridge:entropic_repulsion}[Entropic repulsion.]
	Let \(0<J<U\) satisfy \(\sinh 2J=e^{-2U}\). Then, there exist \(c>0, c_0\geq 0, n_0\geq 1\) such that for any \(n\geq n_0\), \(m_n\geq c_0n\),
	\begin{equation*}
		\Phi_n\big( \mathrm{Mdist}_n(\calC,\calC')\leq \scale_n^2,\, v_R\in \calC,\, v_R'\in \calC' \big)
		\leq
		\frac{1}{n^2\scale_n^{13}}e^{-4\nu_1 n}.
	\end{equation*}
\end{lemma}
\begin{proof}
	Let \(h:\mathrm{Image}(\calC)\times \mathrm{Image}(\mathrm{\Gamma_-})\to\{0,1\}\) be the function
	\begin{equation*}
		h(C,\gamma^*) = \begin{cases}
			1 & \text{ if } \mathrm{Mdist}_n(C,\gamma^*)\leq 2\scale_n^2,
			\\
			0 & \text{ else}.
		\end{cases}
	\end{equation*}In particular, \(C\mapsto h(C,\gamma^*)\) is non-decreasing.
	Note that by Lemma~\ref{lem:DoubleBridge:cone_points} and the cone confinement property, it is sufficient to prove, see Figure~\ref{fig:distance_clusters_path}, that
	\begin{equation*}
		\Phi_n\big( h(\calC,\Gamma_-) = 1,\, v_R\in \calC,\, v_R'\in \calC' \big)
		\leq
		\frac{C}{n^2 \scale_n^{13}}e^{-4\nu_1 n}.
	\end{equation*}
	
	\begin{figure}
		\includegraphics[scale=0.6]{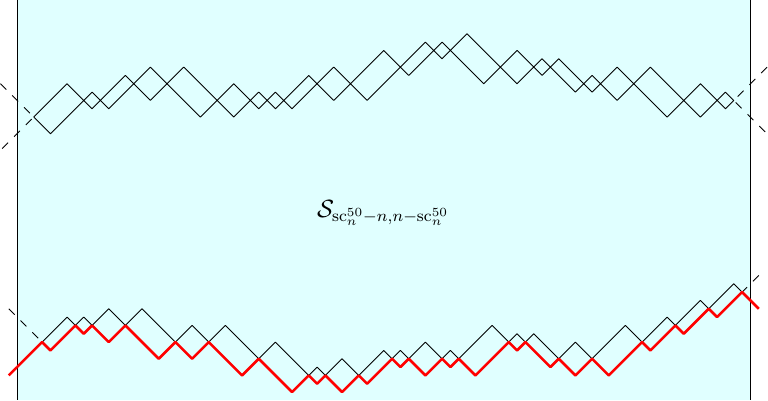}
		\caption{If the distance between two cone points is less than \(\scale_n\), the distance between any point of the lower cluster and the lower boundary of its diamond envelop (in red) is at most \(\scale_n\). Thus, controlling the distance between the top cluster and the lowest path in the lower cluster (which is above the red curve) is sufficient to control the distance between the two clusters.}
		\label{fig:distance_clusters_path}
	\end{figure}
	
	Let \(\bndMeas_n,\tilde{\bndMeas}_n\) be the measures on \(\SetRootMarkBackCont\times \SetRootMarkForwCont\) given by Corollaries~\ref{cor:fin_vol_cluster_bridge_coupling} and~\ref{cor:fin_vol_paths_bridge_coupling} with \(k_n = \scale_n\) respectively. Let \((\tilde{\OZRVchain^v})_{v\in \Z^2}\) be an independent copy of \((\OZRVchain^v)_{v\in \Z^2}\). Let \((\calK_L,\calK_R)\sim \bndMeas_n\), \((\calK_L',\calK_R')\sim \tilde{\bndMeas}_n\) be independent and independent of \((\tilde{\OZRVchain^v})_{v\in \Z^2}\) and \((\OZRVchain^v)_{v\in \Z^2}\). Introduce
	\begin{gather*}
		V = v_R-v_L-\displace(\calK_L) - \displace(\calK_R),
		\quad
		V' = v_R'-v_L'-\displace(\calK'_L) - \displace(\calK'_R),
		\\
		\calK = v_L + \calK_L\concatenate \OZRVcluster^V \concatenate\calK_R,
		\quad
		\calP = v_L' + \Gamma_-\big(\calK_L'\concatenate \widetilde{\OZRVcluster}^{V'} \concatenate\calK_R').
	\end{gather*}Then, using Corollary~\ref{cor:fin_vol_paths_bridge_coupling} to couple \(\Gamma_-(\calC')\) with \(\calP\), and Observation~\ref{obs:DoubleBridge:sup_norm_bound} to control the sup-norm in the error term of Corollary~\ref{cor:fin_vol_paths_bridge_coupling}, we get that
	\begin{multline}
		\label{eq:prf:lem_entropic_rep:eq1}
		\Phi_n\big( h(\calC,\Gamma_-) = 1,\, v_R\in \calC,\, v_R'\in \calC' \big)
		\\
		\leq
		\Phi_n(v_R'\in \calC')e^{-2\nu_1 n}e^{-c\scale_n}
		+ \Phi_n(v_R'\in \calC')E\Big[ \Phi_n^{\calP}\big[\mathds{1}_{v_R\in \calC} h(\calC,\calP) \big]\Big].
	\end{multline}
	For any \(\gamma'\) realisation of \(\calP\), introduce
	\begin{multline*}
		A(\gamma')
		=
		\big\{C\ni v_L,v_R\text{ connected}:\, h(C,\gamma')=1 ,
		\\
		\text{ and } \exists \gamma\subset C\cap \gamma'_{\shortuparrow} \text{ simple path from } v_L \text{ to } v_R\big\}.
	\end{multline*}
	This is an increasing event. Now, by Observation~\ref{obs:DoubleBridge:monotone_cond}, for any \(\gamma'\) realisation of \(\calP\),
	\begin{equation*}
		\Phi_n^{\gamma'}\big[\mathds{1}_{v_R\in \calC} h(\calC,\gamma') \big]
		=
		\Phi_n^{\gamma'}\big[\mathds{1}_{v_R\in \calC} \mathds{1}_{A(\gamma')}(\calC) \big]
		\leq
		\Phi_n\big(v_R\in \calC \in A(\gamma') \big).
	\end{equation*}Then, by Corollary~\ref{cor:fin_vol_cluster_bridge_coupling},
	\begin{equation*}
		\Phi_n\big(v_R\in \calC \in A(\gamma^*) \big)
		\leq
		\Phi_n(v_R\in \calC)P\big( \calK\in A(\gamma^*) \big) + \Phi_n(v_R\in \calC)e^{-c\scale_n}.
	\end{equation*}
	In particular, by these observations and Lemma~\ref{lem:DoubleBridge:apriori_upper_bound_single_crossing}, the R.H.S. of~\eqref{eq:prf:lem_entropic_rep:eq1} is less than
	\begin{equation*}
		e^{-4\nu_1 n- c\scale_n} + \Phi_n(v_R'\in \calC')\Phi_n(v_R\in \calC)  P\big( \calK\in A(\calP) \big).
	\end{equation*}By Lemma~\ref{lem:DoubleBridge:apriori_upper_bound_single_crossing}, and the value of \(\scale_n\), it remains to show that \(P\big( \calK\in A(\calP) \big) \leq \frac{C}{\scale_n^{15} n}\).
	Let \(\OZRVwalk^v, \tilde{\OZRVwalk^v}\) be the walks sampled according to \(\OZRVchain^v\) and~\(\tilde{\OZRVchain^v}\) respectively.
	Introduce
	\begin{equation*}
		S_0 = v_L+\displace(\calK_L),\quad \tilde{S}_0 = v_L'+\displace(\calK_L'),
		\quad
		S_i = S_0 + \displace(\OZRVwalk^V_i),\quad \tilde{S}_i = \tilde{S}_0 + \displace(\widetilde{\OZRVwalk}^V_i).
	\end{equation*}
	Now, from the exponential tails of the step distribution of \(\OZwalk\), we can assume that, up to an error of order \(e^{-c\scale_n}\), the maximal step size of \(S,\tilde{S}\) under \(\OZRVchain^V, \tilde{\OZRVchain}^{V'}\) is of size \(\frac{1}{2}\scale_n\) (a.s.-uniformly over the values of \(V,V'\)). Moreover, as \(\calK,\calP\) are included in the diamond envelop of the corresponding walks, the event \(\calK\in A(\calP)\) intersected with the event that all step sizes are less than \(\frac{1}{2}\scale_n\) is included in the event
	\begin{multline*}
		\big\{\forall k\, S_k \text{ is above } \partial_-\DiaEnv(\tilde{S}_0,\dots,\tilde{S}_{T_{V'}}) \big\}
		\\
		\cap \big\{\exists k \text{ such that } S_k\in \slab_{\scale_n^{50}-n,n-\scale_n^{50}},\ \rmd(S_k,\DiaEnv(\tilde{S}_0,\dots,\tilde{S}_{T_{V'}}))\leq \scale_n \big\}.
	\end{multline*}
	Note that \(\bndMeas_n,\tilde{\bndMeas}_n\) are supported on graphs/paths with displacement sup-norm at most \(\scale_n^2\).
	This allows to use Lemma~\ref{lem:DoubleBridge:RW:UB_avoid_but_close} and the claim follows.
\end{proof}

\subsubsection*{Clusters are good with high probability}

Combining all the previous Lemmas, we obtain the wanted control over the probability to see good pairs of clusters.
\begin{lemma}
	\label{lem:DoubleBridge:good_cluster_have_proba_one}
	Let \(0<J<U\) satisfy \(\sinh 2J=e^{-2U}\). Then, there exist \(c_0\geq 0, n_0\geq 1\) such that for any \(n\geq n_0\), \(m_n\geq c_0n\),
	\begin{equation*}
		\Phi_n\big( (\calC,\calC')\in \goodCl_n \bgiven v_R\in \calC,\, v_R'\in \calC' \big)
		\geq
		1-\frac{1}{\scale_n^{12}}.
	\end{equation*}
\end{lemma}
\begin{proof}
	We will upper bound \(\Phi_n\big( (\calC,\calC')\notin \goodCl_n \bgiven v_R\in \calC,\, v_R'\in \calC' \big)\). First, by Lemma~\ref{lem:DoubleBridge:apriori_lower_bound},
	\begin{equation*}
		\Phi_n\big( (\calC,\calC')\notin \goodCl_n\bgiven v_R\in \calC,\, v_R'\in \calC' \big)
		\leq
		cn^2e^{4\nu_1 n}\Phi_n\big( (\calC,\calC')\notin \goodCl_n,\, v_R\in \calC,\, v_R'\in \calC' \big).
	\end{equation*}Then, by Lemmas~\ref{lem:DoubleBridge:cone_points}, and~\ref{lem:DoubleBridge:entropic_repulsion}, the definition of \(\goodCl_n\), and a union bound,
	\begin{equation*}
		\Phi_n\big( (\calC,\calC')\notin \goodCl_n,\, v_R\in \calC,\, v_R'\in \calC' \big)
		\leq
		e^{-4\nu_1 n}\Big(2n e^{-c\scale_n} + \frac{C}{n^{2} \scale_n^{13}}\Big),
	\end{equation*}which implies the claim.
\end{proof}

\subsection{Good cluster decomposition and density swapping}
\label{subsec:DoubleBridge:density_swap}

Before turning to the proof of Theorem~\ref{thm:coupling_with_avoiding_bridges}, we define a suitable decomposition of the clusters \(\calC,\calC'\) under the event that they are good.
We then prove a suitable bound on relative densities of the ``inside parts'' of the clusters law in finite and infinite volume. This part is similar to~\cite[Section 8.6]{DobGlaOtt25}.
To shorten notations, let
\begin{equation}
	\label{eq:DoubleBridge:def:a_n_connection_events}
	a_n = 2\scale_n^{50},
	\quad
	\connection = \{v_R\in \calC\},
	\quad
	\connection' = \{v_R'\in \calC'\}.
\end{equation}
For a connected subgraph \(C\) of \(\Z^2\) or its dual, define the following.
\begin{itemize}
	\item \(W_L(C)\) is the leftmost cone-point of \(C\) in \(\slab_{a_n -n, a_n+\scale_n -n}\) if it exists, and \(W_L(C) = \dagger\) otherwise.
	\item \(W_R(C)\) is the rightmost cone-point of \(C\) in \(\slab_{n-a_n -\scale_n, n-a_n}\) if it exists, and \(W_R(C) = \dagger\) otherwise.
\end{itemize}
Note that under \((\calC,\calC')\in \goodCl_n\), one has that \(W_L(\calC),W_L(\calC')\in \slab_{a_n -n, a_n -n + \scale_n}\), and \(W_R(\calC),W_R(\calC')\in \slab_{n-a_n -\scale_n, n-a_n}\).
Under the event \((\calC,\calC')\in \goodCl_n\), define
\begin{itemize}
	\item \(\calK_L = \calC\cap (W_L(\calC)+\bcone)\), \(\calK_R = \calC\cap (W_R(\calC)+\fcone)\), \(\calK = \calC\cap \diam(W_L(\calC),W_R(\calC))\),
	\item \(\tilde{\calK}_L = \calK_L-v_L\), \(\tilde{\calK}_R = \calK_R-W_R(\calC)\), \(\tilde{\calK} = \calK - W_L(\calC)\),
\end{itemize}
so that
\begin{equation*}
	\calC = \calK_L\cup \calK \cup \calK_R = v_L + \tilde{\calK}_L\concatenate \tilde{\calK} \concatenate \tilde{\calK}_R,
\end{equation*}and define \(\calK_L^*,\calK_R^*,\calK^*,\tilde{\calK}_L^*,\tilde{\calK}_R^*,\tilde{\calK}^*\) similarly using \(\calC'\) in place of \(\calC\) and \(v_L'\) in place of \(v_L\). If \(W_{L/R}(\calC/\calC')\) are not defined, set all of those to be equal to \(\dagger\).

\begin{lemma}
	\label{lem:DoubleBridge:density_swap}
	Let \(0<J<U\) satisfy \(\sinh 2J=e^{-2U}\). There exists \(n_0\geq 1, c,c_0\in (0,+\infty)\) such that for any \(n\geq n_0\), \(m_n\geq c_0n\), any \(\eta_L, \eta, \eta_R,\eta_L', \eta', \eta_R'\) such that
	\begin{itemize}
		\item \((v_L+\eta_L\concatenate \eta\concatenate \eta_R, v_L'+\eta_L'\concatenate \eta'\concatenate \eta_R') \in \goodCl_n\),
		\item \(W_L(v_L+\eta_L\concatenate \eta \concatenate \eta_R) = v_L + \displace(\eta_L)\), \(W_R(v_L+\eta_L\concatenate \eta \concatenate \eta_R) = v_L + \displace(\eta_L\concatenate \eta)\),
		\item \(W_L(v_L'+\eta_L'\concatenate \eta' \concatenate \eta_R') = v_L' + \displace(\eta_L')\), \(W_R(v_L'+\eta_L'\concatenate \eta' \concatenate \eta_R') = v_L' + \displace(\eta_L'\concatenate \eta')\),
	\end{itemize}
	\begin{equation*}
		\Big|\frac{\Phi_n\big( \tilde{\calK} = \eta,\, \tilde{\calK}^* = \eta' \bgiven \tilde{\calK}_L = \eta_L,\, \tilde{\calK}_R = \eta_R,\, \tilde{\calK}_L^* = \eta_L',\, \tilde{\calK}_R^* = \eta_R',\, \connection,\, \connection' \big)}{\Phi\big( \tilde{\calK} = \eta \bgiven \tilde{\calK}_L = \eta_L,\, \tilde{\calK}_R = \eta_R,\, \connection \big)\Phi\big( \tilde{\calK}^* = \eta' \bgiven \tilde{\calK}_L^* = \eta_L',\, \tilde{\calK}_R^* = \eta_R',\, \connection' \big)} -1 \Big|
		\leq
		e^{-c\scale_n}.
	\end{equation*}
\end{lemma}
\begin{proof}
	\begin{figure}
		\includegraphics[scale=0.6]{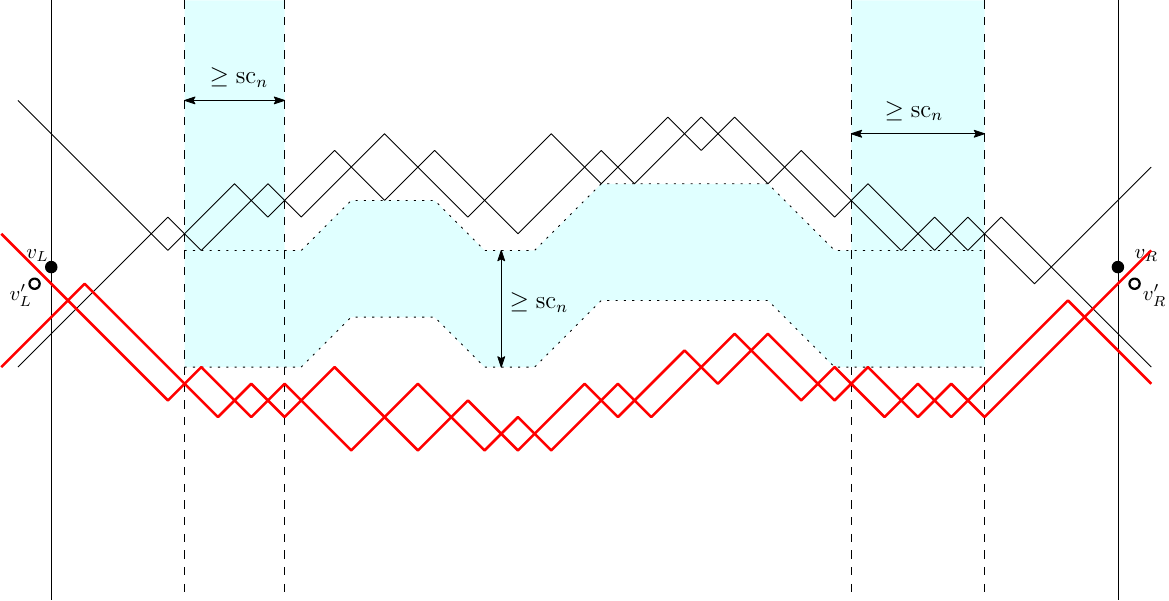}
		\caption{The lower dual cluster (envelop in red) stays well separated from the top primal cluster. This allows for the presence of a wide decoupling region (in light blue), containing many (\(\geq \rho \scale_n\)) cone-points of the top cluster.}
		\label{fig:Density_swapping}
	\end{figure}
	The proof is a straightforward adaptation of the one of~\cite[Lemma 8.7]{DobGlaOtt25}: first, one proves that
	\begin{equation*}
		\frac{\Phi_n\big( \tilde{\calK} = \eta \bgiven \tilde{\calK}_L = \eta_L,\, \tilde{\calK}_R = \eta_R,\, \tilde{\calK}_L^* = \eta_L',\, \tilde{\calK}_R^* = \eta_R',\, \tilde{\calK}^* = \eta',\, \connection,\, \connection' \big)}{\Phi\big( \tilde{\calK} = \eta \bgiven \tilde{\calK}_L = \eta_L,\, \tilde{\calK}_R = \eta_R,\, \connection \big)}
	\end{equation*}is at distance at most \(e^{-c\scale_n}\) from \(1\). Then, one shows that
	\begin{equation*}
		\frac{\Phi_n\big( \tilde{\calK}^* = \eta' \bgiven \tilde{\calK}_L = \eta_L,\, \tilde{\calK}_R = \eta_R,\, \tilde{\calK}_L^* = \eta_L',\, \tilde{\calK}_R^* = \eta_R',\, \connection,\, \connection' \big)}{\Phi\big( \tilde{\calK}^* = \eta' \bgiven \tilde{\calK}_L^* = \eta_L',\, \tilde{\calK}_R^* = \eta_R',\, \connection' \big)}
	\end{equation*}is at distance at most \(e^{-c\scale_n}\) from \(1\). The arguments are carried away exactly as in the proof of~\cite[Lemma 8.7]{DobGlaOtt25}. The only change in the proof compare to~\cite[Lemma 8.7]{DobGlaOtt25} is than one has to use a cluster-dependent annuli, see Figure~\ref{fig:Density_swapping}, instead of a rectangular annuli.
\end{proof}

\subsection{Concluding the proof of Theorem~\ref{thm:coupling_with_avoiding_bridges}}
\label{subsec:DoubleBridge:proof_of_main_thm}

As in~\cite[Section 8.7]{DobGlaOtt25}, we will proceed by successively comparing different sampling procedure for \(\calC,\calC'\). The last Lemma of the previous section motivates the definition of intermediate measures on the primal and dual clusters:
\begin{equation}
	\label{eq:DoubleBridge:def:factoClusterMeas}
	\begin{gathered}
		\condPrimalClusterMeas_{\eta_L}^{\eta_R}(A)
		=
		\Phi\big( \calC\in A \bgiven \tilde{\calK}_L = \eta_L,\, \tilde{\calK}_R = \eta_R,\, \connection \big),
		\\
		\condDualClusterMeas_{\eta_L}^{\eta_R}(A)
		=
		\Phi\big(\calC'\in A \bgiven \tilde{\calK}_L^* = \eta_L,\, \tilde{\calK}_R^* = \eta_R,\, \connection' \big).
	\end{gathered}
\end{equation}
Also introduce the ``factorized/OZ'' version of these measures:
\begin{equation}
	\label{eq:DoubleBridge:def:OZDecCond}
	\begin{gathered}
		\OZDec_{\eta_L,\eta_R}(B)
		\coloneqq 
		\frac{\int d\OZDec(M;\gamma_0^{M+1}) \mathds{1}_{B}(v_L+\bar{\gamma}) \mathds{1}_{\displace(\bar{\gamma}) = v_R-v_L}\mathds{1}_{\bar{\gamma}\sim \eta_L} \mathds{1}_{\bar{\gamma}\sim \eta_R} }{\int d\OZDec(M;\gamma_0^{M+1}) \mathds{1}_{\displace(\bar{\gamma}) = v_R-v_L}\mathds{1}_{\bar{\gamma}\sim \eta_L} \mathds{1}_{\bar{\gamma}\sim \eta_R}},
		\\
		\OZDec_{\eta_L,\eta_R}^*(B)
		\coloneqq 
		\frac{\int d\OZDec(M;\gamma_0^{M+1}) \mathds{1}_{B}(v_L'+\bar{\gamma}) \mathds{1}_{\displace(\bar{\gamma}) = v_R'-v_L'}\mathds{1}_{\bar{\gamma}\sim^* \eta_L} \mathds{1}_{\bar{\gamma}\sim^* \eta_R} }{\int d\OZDec(M;\gamma_0^{M+1}) \mathds{1}_{\displace(\bar{\gamma}) = v_R'-v_L'}\mathds{1}_{\bar{\gamma}\sim^* \eta_L} \mathds{1}_{\bar{\gamma}\sim^* \eta_R}},
	\end{gathered}
\end{equation}
where \(\bar{\gamma} = \gamma_0\concatenate \dots \concatenate \gamma_{M+1}\). We also defined the following events:
\begin{itemize}
	\item \(\bar{\gamma}\sim \eta_L\) if \(\displace(\eta_L)\in \CPts(\bar{\gamma})\) and \(\eta_L = \bar{\gamma} \cap (\displace(\eta_L) + \bcone)\),
	\item \(\bar{\gamma}\sim \eta_R\) if \(v_R-\displace(\eta_R)\in \CPts(v_L+\bar{\gamma})\) and \(v_R-\displace(\eta_R) + \eta_R = (v_L+\bar{\gamma}) \cap (v_R-\displace(\eta_R) + \fcone)\),
	\item \(\bar{\gamma}\sim^* \eta_L\), \(\bar{\gamma}\sim^* \eta_R\) are defined similarly with \(v_L',v_R'\) in place of \(v_L,v_R\).
\end{itemize}
Note that whilst the chains in the support of \(\OZDec\) are rooted at \(0\), the ones in the support of \(\OZDec_{\eta_L,\eta_R}\) are rooted at \(v_L\).

The next Lemma is a  particular case of~\cite[Lemma 8.9]{DobGlaOtt25}.
\begin{lemma}
	\label{lem:DoubleBridge:condClusters_to_condOZ}
	Let \(0<J<U\) satisfy \(\sinh 2J=e^{-2U}\). There exist \(n_0\geq 1, c,c_0\in (0,+\infty)\) such that for any \(n\geq n_0\), \(m_n\geq c_0n\), any \(\eta_L\in \SetRootMarkBackCont, \eta_R \in \SetRootMarkForwCont\) satisfying \(\norm{\displace(\eta_L)}_{\infty},\norm{\displace(\eta_R)}_{\infty} \leq n^{1/4}\),
	one has
	\begin{equation*}
		\tvd\big(\condPrimalClusterMeas_{\eta_L}^{\eta_R}, \OZDec_{\eta_L,\eta_R} \big) \leq e^{-cn},
		\quad
		\tvd\big(\condDualClusterMeas_{\eta_L}^{\eta_R}, \OZDec_{\eta_L,\eta_R}^* \big) \leq e^{-cn}.
	\end{equation*}
\end{lemma}

We will need the following lower bounds on the probabilities of seeing good clusters under different sampling procedures. Introduce the product measures
\begin{gather*}
	\condPrimalDualClusterMeas_{\eta_L,\eta_L'}^{\eta_R,\eta_R'} \equiv \condPrimalClusterMeas_{\eta_L}^{\eta_R}\otimes \condDualClusterMeas_{\eta_L'}^{\eta_R'},
	\quad
	\OZDec^2_{\eta_L,\eta_R,\eta_L',\eta_R'} \equiv \OZDec_{\eta_L,\eta_R}\otimes \OZDec_{\eta_L',\eta_R'}^*.
\end{gather*}From properties of the total variation distance, see Lemma~\ref{lem:app:tot_var_dist_prop}, for any \(\eta_L,\eta_R,\eta_L',\eta_R'\) as in the statement of Lemma~\ref{lem:DoubleBridge:condClusters_to_condOZ},
\begin{equation}
	\label{eq:DoubleBridge:tot_var_dist_product_meas}
	\tvd\big(\condPrimalDualClusterMeas_{\eta_L,\eta_L'}^{\eta_R,\eta_R'}, \OZDec^2_{\eta_L,\eta_R,\eta_L',\eta_R'}\big)
	\leq
	e^{-cn}.
\end{equation}
\begin{lemma}
	\label{lem:DoubleBridge:LB_GCl_product_meas}
	Let \(0<J<U\) satisfy \(\sinh 2J=e^{-2U}\). There exist \(n_0\geq 1, c,c_0\in (0,+\infty)\) such that for any \(n\geq n_0\), \(m_n\geq c_0n\), any \(\eta_L,\eta_L'\in \SetRootMarkBackCont, \eta_R,\eta_R' \in \SetRootMarkForwCont\) satisfying
	\begin{equation*}
		\Phi_n\big(\tilde{\calK}_L = \eta_L,\, \tilde{\calK}_R = \eta_R,\, \tilde{\calK}_L^* = \eta_L',\, \tilde{\calK}_R^* = \eta_R' \bgiven (\calC,\calC')\in \goodCl_n \big)
		>0,
	\end{equation*}
	one has
	\begin{equation*}
		\condPrimalDualClusterMeas_{\eta_L,\eta_L'}^{\eta_R,\eta_R'}\big(\goodCl_n \big)
		\geq
		e^{-c\scale_n^{100}},
		\quad
		\OZDec^2_{\eta_L,\eta_R,\eta_L',\eta_R'}\big(\goodCl_n\big)
		\geq
		e^{-c\scale_n^{100}}.
	\end{equation*}
\end{lemma}

\begin{proof}
	Note that by~\eqref{eq:DoubleBridge:tot_var_dist_product_meas}, it is sufficient to prove any of the two bounds. We prove the second. Let
	\begin{equation*}
		v = v_R-v_L,\quad v' = v_R'-v_L',
		\quad
		\tilde{v} = v_R-v_L -\displace(\eta_L)-\displace(\eta_R),\quad \tilde{v}' = v_R'-v_L'-\displace(\eta_L')-\displace(\eta_R').
	\end{equation*}First, by definition, for any event \(B\),
	\begin{multline*}
		\OZDec_{\eta_L,\eta_R,\eta_L',\eta_R'}^2\big(B\big)
		=
		\\
		\frac{\OZDec^{\otimes 2}\big(B,\, \{\displace(\bar{\gamma}) = v,\, \bar{\gamma}\sim \eta_L,\, \bar{\gamma}\sim \eta_R\}\times \{\displace(\bar{\gamma}) = v',\, \bar{\gamma}\sim^* \eta_L',\, \bar{\gamma}\sim^* \eta_R'\}\big)}{\OZDec\big(\displace(\bar{\gamma}) = v,\, \bar{\gamma}\sim \eta_L,\, \bar{\gamma}\sim \eta_R \big)\OZDec\big(\displace(\bar{\gamma}) = v',\, \bar{\gamma}\sim^* \eta_L',\, \bar{\gamma}\sim^* \eta_R' \big)}
	\end{multline*}where it is good to remember that \(\OZDec\) is a measure but not a probability measure. Then,
	\begin{equation*}
		\OZDec\big(\displace(\bar{\gamma}) = v,\, \bar{\gamma}\sim \eta_L,\, \bar{\gamma}\sim \eta_R \big)
		\leq
		\rmC,
	\end{equation*}with \(\rmC\) the constant given by Theorem~\ref{thm:OZ_atrc_inf_vol}, and similarly for \(\OZDec\big(\displace(\bar{\gamma}) = v',\, \bar{\gamma}\sim^* \eta_L',\, \bar{\gamma}\sim^* \eta_R' \big)\).
	Now, by the finite energy property of \(p_L,p_R,p\), see Theorem~\ref{thm:OZ_atrc_inf_vol}, and the constraint on \(\eta_L,\eta_R,\eta_L',\eta_R'\),
	\begin{align*}
		&\frac{1}{\rmC^2}\OZDec^{\otimes 2}\big(\goodCl_n,\, \{\displace(\bar{\gamma}) = v,\, \bar{\gamma}\sim \eta_L,\, \bar{\gamma}\sim \eta_R\}\times \{\displace(\bar{\gamma}) = v',\, \bar{\gamma}\sim^* \eta_L',\, \bar{\gamma}\sim^* \eta_R'\}\big)
		\\
		&\geq
		e^{-c(|\eta_L|+ |\eta_L'| + |\eta_R| + |\eta_R'|)}
		\\
		&\quad \cdot \OZwalk^{\otimes 2}\big(T_{\tilde{v}}<\infty,\, T_{\tilde{v}'}'<\infty,\, \rmd(\DiaEnv(v_L+\displace(\eta_L)+S),\DiaEnv(v_L'+\displace(\eta_L')+S')) \geq \scale_n\big)
		\\
		&\geq
		e^{-c\scale_n^{100}}\OZwalk^{\otimes 2}\big(\dots\big),
	\end{align*}for some \(c>0\), where \(S,T\) refer to the first marginal of the measure \(\OZwalk^{\otimes 2}\), \(S',T'\) refer to the second, and we used that \(\norm{\displace(\eta_L)}_{\infty},\norm{\displace(\eta_R)}_{\infty},\norm{\displace(\eta_L')}_{\infty},\norm{\displace(\eta_R')}_{\infty}\leq 2\scale_n^{50}\) by assumption. Now, proceeding as in the proof of Lemma~\ref{lem:DoubleBridge:RW:UB_avoid_but_close}, one gets
	\begin{equation*}
		\OZwalk^{\otimes 2}\big(T_{\tilde{v}}<\infty,\, T_{\tilde{v}'}'<\infty,\, \rmd(\DiaEnv(v_L+\displace(\eta_L)+S),\DiaEnv(v_L'+\displace(\eta_L')+S')) \geq \scale_n\big)
		\geq
		\frac{c}{n^2},
	\end{equation*}for some \(c>0\) (note that we need a way worse bound than this). This concludes the proof.
\end{proof}
From~\eqref{eq:DoubleBridge:tot_var_dist_product_meas}, Lemma~\ref{lem:DoubleBridge:LB_GCl_product_meas}, and properties of the total variation distance, see again Lemma~\ref{lem:app:tot_var_dist_prop}, for any \(\eta_L,\eta_R,\eta_L',\eta_R'\) as in the statement of Lemma~\ref{lem:DoubleBridge:LB_GCl_product_meas},
\begin{equation}
	\label{eq:DoubleBridge:tot_var_dist_product_meas_cond_on_good}
	\tvd\big(\condPrimalDualClusterMeas_{\eta_L,\eta_L'}^{\eta_R,\eta_R'}(\cdot \given \goodCl_n), \OZDec^2_{\eta_L,\eta_R,\eta_L',\eta_R'}(\cdot \given \goodCl_n)\big)
	\leq
	e^{-cn}.
\end{equation}

Now, for \(\eta_L\in \SetRootMarkBackCont,\eta_R\in \SetRootMarkForwCont\), and \(\gamma_0^{M+1}\sim \OZDec_{\eta_L,\eta_R}\), define the random variables
\begin{gather*}
	\xi_{\eta_L} = \begin{cases}
		\dagger & \text{ if } v_L+\eta_L\not\subset \bar{\gamma}_0^{M+1},
		\\
		-v_L+ \bar{\gamma}_0^{k_-} & \text{ else},
	\end{cases}
	\quad
	\xi_{\eta_R} = \begin{cases}
		\dagger & \text{ if } v_R-\displace(\eta_R) +\eta_R \not\subset \bar{\gamma}_0^{M+1},
		\\
		\bar{\gamma}_{k_+}^{M+1}  & \text{ else},
	\end{cases}
\end{gather*}where \(\bar{\gamma}_i^j = \gamma_i\concatenate \dots \concatenate \gamma_j\), and
\begin{equation*}
	k_- = \min\{k\geq 0:\ v_L+\eta_L\subset \bar{\gamma}_0^k \},
	\quad
	k_+ = \max\{k\geq 0:\ v_R-\displace(\eta_R) +\eta_R \subset \bar{\gamma}_k^{M+1} \}.
\end{equation*}Also introduce the events
\begin{multline*}
	\Forget_{\eta_L,\eta_R}=
	\big\{\rmT_{\eta_L}\neq \dagger,\, \norm{\displace(\gamma_0\concatenate\dots\concatenate \gamma_{\rmT_{\eta_L}})}_{\infty} \leq \scale_n^{110} \big\}
	\\
	\cap \big\{\rmT_{\eta_R}\neq \dagger,\, \norm{\displace(\gamma_{\rmT_{\eta_R}}\concatenate\dots\concatenate \gamma_{M+1})}_{\infty} \leq \scale_n^{110}\big\},
\end{multline*}and define similarly \(\xi_{\eta_L}^*,\xi_{\eta_R}^*\), and \(\Forget_{\eta_L,\eta_R}^*\) using \(\gamma_0^{M+1}\sim \OZDec_{\eta_L,\eta_R}^*\) instead. Note that \(110\) is as arbitrary as our choice of \(50\) in the previous powers of \(\scale_n\): it just needs to be larger than \(2\cdot 50\).

Then, the next Lemma is a particular case of~\cite[Lemma 8.10]{DobGlaOtt25}.
The number \(54\) is again arbitrary and needs to be larger than \(50\) and less than \(110/2\).
\begin{lemma}
	\label{lem:DoubleBridge:condOZ_to_ForgetCondOZ}
	Let \(0<J<U\) satisfy \(\sinh 2J=e^{-2U}\). There are \(n_0\geq 1, c,c_0\in (0,+\infty)\) such that for any \(n\geq n_0\), \(m_n\geq c_0n\), any \(\eta_L\in \SetRootMarkBackCont, \eta_R \in \SetRootMarkForwCont\) satisfying \(\norm{\displace(\eta_L)}_{\infty},\norm{\displace(\eta_R)}_{\infty} \leq \scale_n^{54}\),
	one has
	\begin{equation*}
		\OZDec_{\eta_L,\eta_R}(\Forget_{\eta_L,\eta_R}) \geq 1-e^{-c\scale_n^{108}},
		\quad
		\OZDec_{\eta_L,\eta_R}^*(\Forget_{\eta_L,\eta_R}^*) \geq 1-e^{-c\scale_n^{108}}.
	\end{equation*}
\end{lemma}
From this Lemma and Lemma~\ref{lem:DoubleBridge:LB_GCl_product_meas}, we get that for any \(\eta_L,\eta_R,\eta_L',\eta_R'\) as in the statement of Lemma~\ref{lem:DoubleBridge:LB_GCl_product_meas},
\begin{equation}
	\label{eq:DoubleBridge:LB_Forget}
	\OZDec^2_{\eta_L,\eta_R,\eta_L',\eta_R'}(\Forget_{\eta_L,\eta_R} \times \Forget_{\eta_L,\eta_R}^* \given \goodCl_n)
	\geq
	1 - 2\frac{e^{-c\scale_n^{108}}}{e^{-c'\scale_n^{100}}}
	\geq
	1- e^{-c\scale_n^{108}},
\end{equation}where we used a union bound.

We can now conclude the proof of Theorem~\ref{thm:coupling_with_avoiding_bridges}.
\begin{proof}[Proof of Theorem~\ref{thm:coupling_with_avoiding_bridges}]
	Introduce the boundary conditions measures on \((\SetRootMarkBackCont\times \SetRootMarkForwCont)^2\) with density \(r_1,r_2\) given by
	\begin{equation*}
		r_1(\eta_L,\eta_R,\eta_L',\eta_R')
		=
		\Phi_n\big(\tilde{\calK}_L=\eta_L,\,\tilde{\calK}_L^*=\eta_L',\, \tilde{\calK}_R=\eta_R,\, \tilde{\calK}_R^*=\eta_R' \bgiven (\calC,\calC')\in \goodCl_n\big),
	\end{equation*}and
	\begin{multline*}
		r_2(\zeta_L,\zeta_R,\zeta_L',\zeta_R')
		=
		\sum_{\eta_L,\eta_R,\eta_L',\eta_R'} r_1(\eta_L,\eta_R,\eta_L',\eta_R')
		\\
		\cdot
		\OZDec_{\eta_L,\eta_R,\eta_L',\eta_R'}^2 \big( \{\xi_L = \zeta_L,\, \xi_R=\zeta_R\}\times \{\xi_L^* = \zeta_L',\, \xi_R^*=\zeta_R'\} \bgiven 	\Forget_{\eta_L,\eta_R,\eta_L',\eta_R'}^2 \cap \goodCl_n \big),
	\end{multline*}with \(\Forget_{\eta_L,\eta_R,\eta_L',\eta_R'}^2 \equiv \Forget_{\eta_L,\eta_R}\times 	\Forget_{\eta_L,\eta_R}^*\). By definition, \(r_1\) is supported on graphs with sup-norm displacement at most \(2\scale_n^{50}\) and which satisfy \((v_L+\displace(\eta_L) - v_L' - \displace(\eta_L'))\cdot \rme_2 \geq \scale_n\) and \((v_R-\displace(\eta_R) - v_R' - \displace(\eta_R'))\cdot \rme_2 \geq \scale_n\), and \(r_2\) on graphs with sup-norm displacement at most \(\scale_n^{110}\) with the same constraint on the second coordinate of their displacement.
	
	First, by Lemma~\ref{lem:DoubleBridge:good_cluster_have_proba_one}, \(A\mapsto \Phi_n\big((\calC,\calC')\in A \bgiven \connection,\connection'\big)\) is at total variation distance at most \(\frac{1}{\scale_n^{11}}\) from \(A\mapsto \Phi_n\big((\calC,\calC')\in A \bgiven \goodCl_n\big)\). Then, by Lemma~\ref{lem:DoubleBridge:density_swap}, this last measure is at total variation distance at most \(e^{-c\scale_n}\) from
	\begin{equation}
		A\mapsto \sum_{\eta_L,\eta_L'\in \SetRootMarkBackCont} \sum_{\eta_R,\eta_R'\in \SetRootMarkForwCont}  r_1(\eta_L,\eta_R,\eta_L',\eta_R') \condPrimalDualClusterMeas_{\eta_L,\eta_L'}^{\eta_R,\eta_R'}(A \given \goodCl_n),
	\end{equation}which, by~\eqref{eq:DoubleBridge:tot_var_dist_product_meas_cond_on_good}, is at total variation distance at most \(e^{-cn}\) from
	\begin{equation}
		A\mapsto \sum_{\eta_L,\eta_L'\in \SetRootMarkBackCont} \sum_{\eta_R,\eta_R'\in \SetRootMarkForwCont}  r_1(\eta_L,\eta_R,\eta_L',\eta_R') \OZDec^2_{\eta_L,\eta_R,\eta_L',\eta_R'}(A \given \goodCl_n).
	\end{equation}By~\eqref{eq:DoubleBridge:LB_Forget}, this measure is in turn at total variation distance at most \(e^{-c\scale_n^{108}}\) from
	\begin{equation*}
		A\mapsto \sum_{\eta_L,\eta_L'\in \SetRootMarkBackCont} \sum_{\eta_R,\eta_R'\in \SetRootMarkForwCont}  r_1(\eta_L,\eta_R,\eta_L',\eta_R') \OZDec^2_{\eta_L,\eta_R,\eta_L',\eta_R'}\big(A \given \goodCl_n\cap \Forget_{\eta_L,\eta_R,\eta_L',\eta_R'}^2\big).
	\end{equation*}But, partitioning with respect to the realization of \(\xi_L,\xi_R,\xi_L^*,\xi_R^*\), this last measure is equal to the probability measure 
	\begin{align*}
		&A\mapsto \sum_{\eta_L,\eta_L'\in \SetRootMarkBackCont} \sum_{\eta_R,\eta_R'\in \SetRootMarkForwCont}  r_2(\eta_L,\eta_R,\eta_L',\eta_R') P\big( (v_L+\eta_L\concatenate \bar{\gamma}_1^{M}\concatenate \eta_R, v_L'+\eta_{L}'\concatenate \bar{\zeta}_{1}^{M'}\concatenate \eta_{R}') \in A
		\\
		&\phantom{A\mapsto \sum_{\eta_L,\eta_L'\in \SetRootMarkBackCont} \sum_{\eta_R,\eta_R'\in \SetRootMarkForwCont}  r_2(\eta_L,\eta_R,\eta_L',\eta_R') P\big(}  \bgiven \rmd(v_L +\displace(\eta_L)+ \bar{\gamma}_1^{M}, v_L'+\displace(\eta_L') + \bar{\zeta}_{1}^{M'}) \big)
	\end{align*}where \(\gamma_1^M\sim \OZ_{v_R-v_L-\displace(\eta_L)-\displace(\eta_R)}\) and \(\zeta_1^{M'}\sim \OZ_{v_R'-v_L'-\displace(\eta_L')-\displace(\eta_R')}\) are independent, and we denoted
	\begin{equation*}
		\bar{\gamma}_1^M = \gamma_1\concatenate \dots \concatenate \gamma_M,
		\quad
		\bar{\zeta}_1^{M'} = \zeta_1\concatenate \dots \concatenate \zeta_{M'}.
	\end{equation*}
	This is the form of the measure claimed in Theorem~\ref{thm:coupling_with_avoiding_bridges}.
\end{proof}

\appendix

\section{Local Limit Theorem and consequences}

\begin{theorem}
	\label{thm:LLT_in_Zd}
	Let \(d\geq 1\). Let \(X_1,X_2,\dots\) be an independent family of identically distributed random variables taking values in \(\Z^d\). Define \(S_n = \sum_{k=1}^n X_k\). Assume that there are \(\mu\in \R^d\), \(\delta>0\) such that
	\begin{equation*}
		E(X_1) = \mu,
		\quad
		\Sigma>0,
		\quad
		\sup_{t\in \R^d:\norm{t}\leq \delta}E(e^{t\cdot X_1}) <\infty,
		\quad
		\big|E(e^{\rmi \theta \cdot X_1})\big|<1, \ \forall\, \theta\in [-\pi,\pi]^d\setminus \{0\},
	\end{equation*}where \(\Sigma_{ij} = \mathrm{Cov}(X_{1,i},X_{1,j})\), and \(X_{k,i} = X_k \cdot \rme_i\). Let \(\delta_n\) be a positive decreasing sequence with \(\delta_n \to 0\). Then, for any \(n\geq 1\), \(x\in \Z^d\) with \(\norm{x-\mu n}_{\infty} \leq \delta_n n^{2/3}\),
	\begin{equation*}
		P(S_n = x) = \frac{1}{\sqrt{(2\pi n)^d \det(\Sigma)} }e^{-\frac{1}{2n}(x-\mu n)\cdot \Sigma^{-1}(x-\mu n)}(1+o_n(1)),
	\end{equation*}where the \(o_n(1)\) is uniform over \(x\) as above.
\end{theorem}

The next Theorem is a particular case of the computation of~\cite[section 8]{AouOttVel24}.
\begin{theorem}
	\label{thm:OZ_asymp_from_LLT}
	Let \(d\geq 1\). Let \((T_1,X_1),(T_2,X_2),\dots\) be an independent family of identically distributed random variables taking values in \(\{1,2,\dots\}\times \Z^d\). Define \(\calT_n = \sum_{k=1}^n T_i\), \(S_n = \sum_{k=1}^n X_k\). Assume \(P(T_1=1)>0\), that \(E(X_1  \given T_1 = t) = 0\) for all \(t\) in the support of \(T_1\), and that there are \(\mu\geq 1\), \(\delta>0\) such that for any \(\lambda\in \R^{d+1}:\norm{\lambda}\leq \delta\)
	\begin{equation*}
		E(T_1) = \mu,
		\quad
		\Sigma>0,
		\quad
		E(e^{\lambda\cdot (T_1,X_1)}) <\infty,
		\quad
		\big|E(e^{\rmi \theta \cdot X_1})\big|<1, \ \forall\, \theta\in [-\pi,\pi]^d\setminus \{0\},
	\end{equation*}where \(\Sigma_{ij} = \mathrm{Cov}(X_{1,i},X_{1,j})\), and \(X_{k,i} = X_k \cdot \rme_i\). Let \(\alpha\in [\frac{1}{2},\frac{2}{3})\) and \(K>0\). Then, for \(n\geq 1\), \(x\in \Z^d\) with \(\norm{x}_{\infty} \leq K n^{\alpha}\),
	\begin{equation*}
		P\big(\exists k\geq 1:\ T_k = n,\, S_k = x\big) = \frac{1}{\mu\sqrt{(2\pi n)^d \det(\Sigma)} }e^{-\frac{1}{2n}x\cdot \Sigma^{-1}x}(1+o_n(1)),
	\end{equation*}where the \(o_n(1)\) is uniform over \(x\) as above.
\end{theorem}
\begin{proof}
	Note that if \(P(T_1=1)\), the claim follows from Theorem~\ref{thm:LLT_in_Zd}. Suppose \(P(T_1=1)<1\). The variables \(Y_i = (Y_{i,0},Y_{i,1},\dots,Y_{i,d}) = (T_i,X_{i,1},\dots,X_{i,d})\) are i.i.d. random variables with exponential tails, and have an aperiodic distribution on \(\Z^{d+1}\), as \(P(T_1=1)>0\), and the \(X_i\)'s are aperiodic. Moreover, by our assumptions, one has
	\begin{equation*}
		\Sigma'
		\coloneqq
		\big(\mathrm{Cov}(Y_{1,i},Y_{1,j})\big)_{i,j=0}^{d}
		=
		\begin{pmatrix}
			\sigma^2 & 0 \\
			0 & \Sigma
		\end{pmatrix}
		>0,
	\end{equation*}with \(\sigma^2 = \mathrm{Var}(T_1)>0\). So, we can apply Theorem~\ref{thm:LLT_in_Zd} to the \(Y_i\)'s. Let \(\alpha' \in (\alpha, \frac{2}{3})\). Now, from straightforward large deviation bounds (exponential Chebychev + Taylor expansion of the log-moment generating function of \(T_1-E(T_1)\)),
	\begin{equation*}
		P\big(|T_k-k\mu| \geq R\big)
		\leq
		\begin{cases}
			2e^{-c \frac{R^2}{k}} & \text{ if } R\leq \rho k,
			\\
			e^{-c'R} & \text{ if } R> \rho k,
		\end{cases}
	\end{equation*}
	for some \(\rho,c,c'>0\), and every \(k,R>0\), and so
	\begin{equation*}
		P\big(\exists k \notin [\tfrac{n}{\mu}- n^{\alpha'}, \tfrac{n}{\mu} + n^{\alpha'} ]:\ T_k = n\big)
		\leq
		Ce^{-c n^{2\alpha'-1}}
	\end{equation*}for some \(C,c>0\). Thus,
	\begin{equation*}
		P\big(\exists k\geq 1:\ T_k = n,\, S_k = x\big)
		=
		\mathrm{err}_n + \sum_{\frac{n}{\mu} - n^{\alpha'} \leq k \leq \frac{n}{\mu} + n^{\alpha'}} P\big(T_k = n,\, S_k = x\big)
	\end{equation*}with \(0\leq \mathrm{err}_n\leq Ce^{-cn^{2\alpha'-1}}\), which is negligible in view of the target estimate. Now, by Theorem~\ref{thm:LLT_in_Zd}, the sum in the last display is equal to
	\begin{equation*}
		(1+o_n(1))\sum_{\frac{n}{\mu} - n^{\alpha'} \leq k \leq \frac{n}{\mu} + n^{\alpha'}} \frac{1}{\sqrt{(2\pi n)^{d+1} \sigma^2 \det(\Sigma)} }e^{-\frac{(k\mu - n)^2}{2n \sigma^2}}e^{-\frac{1}{2n} x\cdot \Sigma^{-1}x}
	\end{equation*}with \(o_n(1)\) uniform over \(\norm{x}\leq n^{\alpha}\), as
	\begin{equation*}
		\det(\Sigma') = \sigma^2 \det(\Sigma),
		\quad
		(\Sigma')^{-1}
		=
		\begin{pmatrix}
			\sigma^{-2} & 0 \\
			0 & \Sigma^{-1}
		\end{pmatrix}.
	\end{equation*}This is then equal to
	\begin{multline*}
		(1+o_n(1)) \frac{1}{\sqrt{(2\pi n)^{d} \det(\Sigma)} } e^{-\frac{1}{2n} x\cdot \Sigma^{-1}x} \frac{1}{\sqrt{2\pi n \sigma^2} }\int_{\frac{n}{\mu} - n^{\alpha'}}^{\frac{n}{\mu} + n^{\alpha'}}e^{-\frac{(x\mu - n)^2}{2n \sigma^2}} dx
		\\
		=
		(1+o_n(1)) \frac{1}{\sqrt{(2\pi n)^{d} \det(\Sigma)} } e^{-\frac{1}{2n} x\cdot \Sigma^{-1}x} \frac{1}{\mu \sqrt{2\pi} }\int_{- \frac{\mu}{\sigma}n^{\alpha'-\frac{1}{2}}}^{\frac{\mu}{\sigma} n^{\alpha'-\frac{1}{2}}}e^{-\frac{x^2}{2}} dx
		\\
		=
		(1+o_n(1)) \frac{1}{\mu\sqrt{(2\pi n)^{d} \det(\Sigma)} } e^{-\frac{1}{2n} x\cdot \Sigma^{-1}x}
	\end{multline*}as wanted.
\end{proof}

\section{Imported result on Ordered Bridges}

We state here the result that we used in Step~3 of the proof of Theorem~\ref{thm:invariance_princ_FK}.
We import it from~\cite{DAl24} which is largely based on the results of~\cite{DenWac15,DurWac20}. We use the notations of Section~\ref{subsec:DoubleBridge:RW_system}.
Recall
\begin{equation*}
	P_{x,x'}^{y,y'} = P(\cdot \given S_0=x,\, S_0'=x',\, \HitEvent_{y,y'}),
\end{equation*}and define the linear interpolations of \((S,S')\) under \(P_{x,x'}^{y,y'}\) for \(-n\leq x_1<y_1\leq n\), \(-n\leq x_1'<y_1'\leq n\): for \(t\in [-n,n]\), let \(\mathrm{Li}_n(t)\) be the value of the second coordinate of the linear interpolation of \(((-n,0), S_0,\dots,S_{T_y},(n,0))\) corresponding to the first coordinate \(t\), and define similarly \(\mathrm{Li}'(t)\).
\begin{theorem}
	\label{thm:InvPrinc_AvoidingBridges}
	Let \(\delta_n,n\geq 1\) be a decreasing sequence that converges to \(0\). Then, there exists \(\sigma
	>0\) such that uniformly over sequences \(x_n,y_n\in \Z^2\), \(x_n',y_n'\in (1/2,1/2)+\Z^2 \) with
	\begin{itemize}
		\item \(\delta_n n^{1/2} \geq x_n\cdot \rme_2 > x_n'\cdot \rme_2 +\ln^2(n) \geq -\delta_n n^{1/2}\), \(\delta_n n^{1/2} \geq y_n\cdot \rme_2 > y_n'\cdot \rme_2 +\ln^2(n) \geq -\delta_n n^{1/2}\),
		\item \(-n\leq x_n\cdot \rme_1, x_n'\cdot \rme_1 \leq  -n +\delta_n n^{1/2}\), \(n-\delta_n n^{1/2}\leq y_n\cdot \rme_1, y_n'\cdot \rme_1 \leq n\),
	\end{itemize}the following holds. If \(S,S'\) are sampled under \(P_{x_n,x_n'}^{y_n,y_n'}(\cdot \given \DiaEnv\cap \DiaEnv' = \varnothing)\), where \(\DiaEnv = \DiaEnv(S_0,\dots, S_{T_{y_n}})\), \(\DiaEnv'=\DiaEnv(S_0',\dots, S_{T_{y_n'}'}')\), then
	\begin{equation*}
		\Big(\frac{1}{\sqrt{n}} \big(\mathrm{Li}_n(2tn-n), \mathrm{Li}_n'(2tn-n)\big)\Big)_{t\in [0,1]} \xrightarrow{law} \Big(\sigma \mathrm{BW}_t^{(2)}\Big)_{t\in [0,1]}
	\end{equation*}where \(\mathrm{BW}^{(2)}\) is the Brownian watermelon with 2 bridges, and the convergence is in \(C([-1,1],\R^2)\). The uniformity means that the families of measures are uniformly tight, and the convergence of finite dimensional moments is uniform.
\end{theorem}
\begin{proof}
	This is a straightforward adaptation of the proof of~\cite[Theorem 5.5]{DAl24}: the only thing needed is to shift the second walk by \((1/2,1/2)\), which is indistinguishable from the non-shifted case, as we took the starting points at large distance from one another. The uniformity is direct by the uniformity of the results imported in~\cite{DAl24} from~\cite{DenWac15}, and~\cite{DurWac20}.
\end{proof}

\section{Elementary facts on total variation distance}
\label{app:tot_var_dist}

We gather here some elementary results on total variation distance. Recall that for two finite measures \(\mu,\nu\) on a finite or countable set \(\Omega\) (with discrete sigma-algebra), one can define their total variation distance by
\begin{equation*}
	\tvd(\mu,\nu)
	\coloneqq
	\sup_{A\subset \Omega} |\mu(A) - \nu(A)|
	=
	\tfrac{1}{2}\sum_{\omega\in \Omega} |\mu(\omega) - \nu(\omega)|
	=
	\sup_{\norm{f}_{\infty} \leq 1} \bigl|\mu[f] -\nu[f]\bigr|
\end{equation*}where the second supremum is over functions \(f:\Omega\to \R\), and we denoted \(\mu[f] = \int f d\mu\). In the case of \(\mu,\nu\) probability measures, \(\tvd(\mu,\nu) = \min_{\pi\in \Pi(\mu,\nu)} \pi(X\neq Y)\) where \(\Pi(\mu,\nu)\) is the set of couplings of \(\mu\) and \(\nu\), and \((X,Y)\sim \pi\).

\begin{lemma}
	\label{lem:app:tot_var_dist_prop}
	Let \(\epsilon,\delta>0\). Let \(\Omega_1,\Omega_2\) be finite or countable sets. Let \(\mu,\mu'\) be probability measures on \((\Omega_1,\calP(\Omega_1))\), and \(\nu,\nu'\) be probability measures on \((\Omega_2,\calP(\Omega_2))\). Then, one has the following.
	\begin{enumerate}
		\item If \(\tvd(\mu,\mu')\leq \epsilon\) and \(\tvd(\nu,\nu')\leq \epsilon\), then \(\tvd(\mu\otimes \nu,\mu'\otimes \nu')\leq 2\epsilon\).
		\item If \(\tvd(\mu,\mu')\leq \epsilon\) and \(A\subset \Omega_1\) satisfies \(\mu(A)\geq \delta\), \(\mu'(A)\geq \delta\), then \(\tvd(\mu_A,\mu'_A)\leq \epsilon(\frac{1}{\delta}+\frac{1}{\delta^2})\), where \(\mu_A \equiv \mu(\cdot \given A)\).
	\end{enumerate}
\end{lemma}
\begin{proof}
	For the first claim, for \(B\subset \Omega_1\times \Omega_2\), \(\omega\in \Omega_1\) let \(B_{\omega} = \{\eta\in \Omega_2:\ (\omega,\eta)\in B\}\). Then,
	\begin{equation*}
		|\mu\otimes \nu(B)- \mu\otimes \nu'(B) | \leq \sum_{\omega\in \Omega_1} \mu(\omega) |\nu(B_{\omega})-\nu'(B_{\omega})|
		\leq \epsilon.
	\end{equation*}In the same fashion, \(|\mu\otimes \nu'(B)- \mu'\otimes \nu'(B) | \leq \epsilon\). Triangle inequality concludes.
	
	To get the second claim, for \(B\subset \Omega_1\), one has
	\begin{align*}
		|\mu_A(B)-\mu_A'(B)|
		&\leq
		\frac{|\mu(A\cap B)-\mu'(A\cap B)|}{\mu(A)} + \mu'(A\cap B)\Big|\frac{1}{\mu(A)} - \frac{1}{\mu'(A)}\Big|
		\\
		&\leq
		\frac{\epsilon}{\delta} + \frac{|\mu'(A)- \mu(A)|}{\mu(A)\mu'(A)|}
		\leq
		\frac{\epsilon}{\delta} + \frac{\epsilon}{\delta^2}. \qedhere
	\end{align*}
\end{proof}

\bibliographystyle{amsalpha}
\bibliography{biblicomplete}

\end{document}